\documentclass{amsart}
\usepackage{amssymb}
\usepackage{amsmath}
\usepackage{amscd}
\usepackage{amsthm}
\usepackage{url}
\usepackage{etoolbox}
\usepackage{mathtools}
\usepackage{longtable}
\usepackage{hyperref}
\usepackage{xypic}

\newtheorem{prop}[equation]{Proposition}
\newtheorem{prop*}[equation]{Proposition*}

\newtheorem{lem}[equation]{Lemma}
\newtheorem{lem*}[equation]{Lemma*}
\newtheorem{thm}[equation]{Theorem}
\newtheorem{thm*}[equation]{Theorem*}
\newtheorem{cor}[equation]{Corollary}
\theoremstyle{definition}
\newtheorem{assumption}[equation]{Assumption}
\newtheorem{exo}[equation]{Example}
\newtheorem{remar}[equation]{Remark}
\numberwithin{equation}{section}
\preto{\table}{\stepcounter{equation}}

\newcommand{\ord}{{\mathrm {ord}}}

\newcommand{\Hom}{{\mathrm {Hom}}}

\newcommand{\Fil}{{\mathrm {Fil}}}
\newcommand{\Nm}{{\mathrm {Nm}}}

\newcommand{\cris}{{\mathrm {cris}}}

\newcommand{\dR}{{\mathrm {dR}}}

\newcommand{\et}{\text{\'et}}

\newcommand{\Frob}{{\mathrm {Frob}}}

\newcommand{\Gal}{\mathrm {Gal}}
\newcommand{\Spec}{{\mathrm {Spec}}}

\newcommand{\res}{{\mathrm {res}}}

\newcommand{\Tam}{{\mathrm {Tam}}}

\newcommand{\reg}{{\mathrm {reg}}}

\newcommand{\CC}{{\mathbb C}}
\newcommand{\C}{{\mathbb C}}
\newcommand{\RR}{{\mathbb R}}
\newcommand{\R}{{\mathbb R}}
\newcommand{\PP}{{\mathbb P}}
\newcommand{\QQ}{{\mathbb Q}}
\newcommand{\Q}{{\mathbb Q}}
\newcommand{\ZZ}{{\mathbb Z}}
\newcommand{\Z}{{\mathbb Z}}

\newcommand{\EE}{{\mathcal E}}

\newcommand{\Chi}{{\mathcal X}}

\newcommand{\AAA}{{\mathcal A}}
\newcommand{\CCC}{{\mathcal C}}
\newcommand{\PPP}{{\mathcal P}}

\newcommand{\MMM}{{\mathcal M}}

\newcommand{\OOO}{{\mathcal O}}

\newcommand{\DDD}{{\mathcal D}}
\newcommand{\YYY}{{\mathcal Y}}

\newcommand{\FF}{{\mathbb F}}

\newcommand{\ch}{\mathrm {ch}}

\newcommand{\Qbar}{\overline{\mathbb Q}}

\newcommand{\CH}{\mathrm{CH}}

\def\tf{\textup{tf}}
\def\tor{\textup{tor}}
\def\unr{\textup{unr}}
\def\INT #1 {K_2^T(#1)_{\textup{int}}}
\def\et{\textup{\'et}}
\def\het{H_{\et}}
\def\symb #1 #2 {\left\{ #1 , #2 \right\}}
\def\F{\mathbb{F}}
\def\Tate{\textup{Tate}} % to avoid confusion with $ T_\ell $, etc.
\def\VTate#1{V\textup{Tate}_{#1}} % to match \Tate

\def\vph{\vphantom{$b^{b^{b^b}}p_{p_{p_{p_p}}}$}} % for vertical space in tables
\def\vphs{\vphantom{$b^{b^{b}}p_{p}$}} % for vertical space in tables

\DeclareFontFamily{U}{wncy}{}
    \DeclareFontShape{U}{wncy}{m}{n}{<->wncyr10}{}
    \DeclareSymbolFont{mcy}{U}{wncy}{m}{n}
    \DeclareMathSymbol{\Sh}{\mathord}{mcy}{"58} 

\begin{document}
\title[The 2-part of the Bloch-Kato conjecture for some elliptic curves]{\boldmath The 2-part of the Bloch-Kato conjecture,
and indivisibility results, for $ K_2 $ of some elliptic curves}
\author{Neil Dummigan}
\address{University of Sheffield\\ School of Mathematical and Physical Sciences\\
Hicks Building\\ Hounsfield Road\\ Sheffield, S3 7RH\\
United Kingdom}
\email{n.p.dummigan@shef.ac.uk}
\author{Vasily Golyshev}
\address{PRIMI (Polo di Ricerca Interuniversitario in Matematica e Informatica)
and IGAP (Institute for Geometry and Physics), c/o Math Section of the ICTP, Strada Costiera~11, Trieste~34151, Italy}\email{vasily.v.golyshev@gmail.com}
\author{Rob de Jeu}
\address{Faculteit der B\`etawetenschappen\\Afdeling Wiskunde\\Vrije Universiteit Amsterdam\\De Boelelaan~1111\\1081 HV Amsterdam\\The Netherlands}\email{r.m.h.de.jeu@vu.nl}
\author{Matt Kerr}
\address{Department of Mathematics, Washington University in St. Louis, 1 Brookings Drive, St.~Louis, MO 63132, USA}\email{matkerr@wustl.edu}

\thanks{Matt Kerr gratefully acknowledges support from the Simons Foundation and from NSF Grant DMS-2502708.}

\date{8th May, 2026}

\begin{abstract}
For certain integers~$ u $, we investigate the 2-part of the Bloch-Kato conjecture
for $L(E_u,2)$, where~$E_u: y^2=x(x+1)(x+u^2)$ is part of
a (twisted) Legendre family that is 2-isogenous to a family studied by Boyd \cite{Bo}.
For this, we first work out the corresponding 2-parts of the Tamagawa factors
and Galois invariants.
Then we give an explicit description of the 2-torsion in the Selmer group $
H_f^1(\Q, E_u[2^\infty](-1) ) $.
We construct a specific element in the kernel of the tame symbol
for~$K_2$ on an integral model of~$ E_u $, with non-vanishing real and 2-adic regulators.
Using techniques involving the norm residue isomorphism of Merkur'ev-Suslin,
we prove indivisibility of this element by~2 in that kernel, even modulo torsion,
even though it is explicitly divisible by~2 in the kernel of
the tame symbol for~$ K_2 $ on~$ E_u $.
We also bound the 2-divisibility of the images of these elements under the 2-adic regulator
map.
Finally, in many cases we investigate numerically the validity
of the 2-part of the Bloch-Kato conjecture.
\end{abstract}

\subjclass[2010]{11G40, 19F27}

\keywords{$ K_2 $, curve, tame symbol, regulator, Bloch-Kato conjecture, indivisibility}

\maketitle

\section{Introduction}
In this paper, we study the family of elliptic curves defined
by
\begin{equation} \label{Eueq}
E_u : y^2 = x (x+1) (x+u^2)
\qquad
(u \text{ in } \Q \text{ with } u^2 \ne 0,1)
.
\end{equation}
It is 2-isogenous to the family defined by $ X + X^{-1} + Y + Y^{-1} = 4 u $
in \cite{Bo} (see Proposition~\ref{boyd-isogeny}).

In \emph{Disquisitiones Arithmeticae}, Gauss merges
proper equivalence classes of binary quadratic forms of discriminant $D$ according to a coarser equivalence relation. Let $F$ be a primitive form with discriminant $D$
and $p$ a prime number dividing~$D$. Then the numbers not divisible by $p$
that can be represented by the form $F$ agree in that they are either
all quadratic residues of $p$, or all non-residues. Two classes are said to belong to the same genus if, for every $p|D$, (forms in) those classes agree on whether they represent residues or non-residues. He then introduces a composition law on the set of proper equivalence classes, turning it into an Abelian group, and shows that the composition law respects  genus equivalence.  

The main theorem (Article 261) in that part of \it Disquisitiones \rm states that `half of all the assignable characters for a
positive non-square determinant can correspond to no properly
primitive genus and, if the determinant is negative, to no properly
primitive positive
genus'. To rephrase this in modern language,  let $K$ be a quadratic extension of $\Q$, of discriminant $D$ and ramified at $r$ primes.
Then the $2$-torsion subgroup of the narrow class group, $\mathrm{Cl}^+(\mathcal{O}(K))[2]$, is isomorphic to $ (\Z/2\Z)^{r-1}$,
and is generated by the classes of the ramified primes
of~$ K $~\cite[p.181]{F-T}. One can think of generalizing this result directly by considering higher
degree extensions~\cite{BC}, or changing the base field~\cite{Beau}, or both.

We shall pursue a different analogy, viewing a quadratic field as a kind of $0$-dimensional Calabi-Yau variety, analogous to a cubic curve, a quartic surface or a quintic threefold. The order of its class group,
which occurs as a factor in the leading term at $s=0$ of its Dedekind zeta function,
is influenced by the number of primes of `bad reduction'.
We shall explore how the $2$-parts of various factors, appearing in the Bloch-Kato conjecture for~$L(E_u, 2)$, are influenced by primes of bad reduction.

Let us work out two examples in parallel. Consider first the family of quadratic extensions of the form~$K_u=\Q (\sqrt d)$, where
$ u \ge 5 $ is an integer and $d=u^2-4$ is squarefree.
Let $x>1$ satisfy~$x^2-ux+1=0$. We remark that $x$ is a fundamental unit
of~$ \mathcal{O}(K_u) $
by \cite[Satz~1]{Degert}. Its regulator $F_1(u) =\log x$ can be expressed analytically
as
$$ F_1(u) = \log(u)- \sum_{n=1}^\infty {2n  \choose n} \frac{u^{-2n}}{2n},$$
so that, by Dirichlet's theorem combined with the fact that  $\zeta (0)=-\frac12$,
\begin{equation*}
\frac{L'(\chi_d,0)}{F_1(u)} =
\# \mathrm{Cl} (\mathcal{O}(K_u)) = 
\frac 12 \# \mathrm{Cl^+}(\mathcal{O}(K_u))
.
\end{equation*}
The second equality follows from \cite[Chapter~V, (1.14)]{F-T}, since $ K_u  $
is a real quadratic field with fundamental unit of norm~1.

\medskip

This leads us  to the following key observations.
\begin{enumerate}
\item[(a)]
The 2-valuation of the $L$-ratio (i.e., the left-hand side) cannot go below a certain bound, because it is controlled by the order of~$ \mathrm{Cl}(\mathcal{O}(K_u)) $.
\item[(b)]
When the 2-valuation of the $L$-ratio is minimal then~$ \mathrm{Cl}(\mathcal{O}(K_u)) $ has trivial
2-torsion. In particular, then~$ \mathrm{Cl}^+(\mathcal{O}(K_u))[2] $
has order~2, so that, by Gauss, this can occur only
when~$u-2$ and $u+2$ are prime numbers.\footnote{More precisely, by Proposition~9.2 and Theorem~9.10 of~\cite{Lem21}, the 2-torsion in $ \mathrm{Cl}(\mathcal{O}(K_u)) $ is trivial precisely when~$ u-2 $ and~$ u+2 $ are prime numbers congruent to~3 modulo~4.}
(Think, for instance, of the case $u=105$,
with class number~7.)
\item[(c)]
The classes of the ramified primes generate~$ \mathrm{Cl}^+ (\mathcal{O}(K_u))[2] $,
and their images in~$ \mathrm{Cl} (\mathcal{O}(K_u))[2] $ generate
a subgroup\footnote{Corollary 1 to Theorem~39
of~\cite{F-T} is not correct as stated, since the map~$ \pi_N $ in~(1.8) there need not be surjective when restricted to 2-torsion. An example is $ K_{15} $,
with~$ \mathrm{Cl}^+ (\mathcal{O}(K_{15})) \simeq \Z/4\Z $ and~$ \mathrm{Cl} (\mathcal{O}(K_{15})) \simeq \Z/2\Z $.} of index~1 or~2; moreover, these classes
are subject to only one or two relations.
\item[(d)]
$x=(u+\sqrt{u^2-4})/2$ gives us an explicit fundamental unit.
\end{enumerate}

Let us consider an analogue in dimension 1 of the family~$\Q(\sqrt{u^2-4})$: the family of elliptic curves defined
by~\eqref{Eueq}, but for~$ u $ a positive integer with
$ u \equiv 4 $ modulo~8 and~$ \frac14 u(u^2-1) $ squarefree. All observations (a) through~(d) have (partly conjectural) counterparts in this elliptic curve case, but,
as we shall explain below, the analogy is imperfect, and there are important caveats and differences.

In order to obtain a regulator, first of all we need an element in the
kernel~$ \INT E_u $ of the tame symbol on an integral model of the curve,
a subgroup of the kernel of the tame symbol~$ K_2^T(E_u) $ on
the curve (see Sections~\ref{ellreg} and~\ref{K2elements}).
One can produce an explicit element~$ \alpha_u =\{v,w\}+\{-1,h\}$ in~$ K_2^T(E_u) $, with the property that $2\alpha_u$ (but not $\alpha_u$) is in~$ \INT E_u $ whenever $4u\in\mathbb{Z}$
(see Proposition~\ref{tameINT} and Remark~\ref{tamenotINTrem}).
Here $h=\frac{u(x+1)+y}{x+u}$, $v=\frac{x+u^2}{y}$ and $w=\frac{u-xv}{u+xv}$ are functions on~$E_u$ with divisors supported in the set of
its eight rational $4$-torsion points.

Under the 2-isogeny $\phi\colon E_u\to C_{4u}$ to Boyd's curve $X+Y+X^{-1}+Y^{-1}=4u$, given by $\phi^*(X,Y)=(-vw,v/w)$, $-2\alpha_u$ is the pullback of $\{X,Y\}|_{C_{4u}}\in K^T_2(C_{4u})$.
Reinterpreting the latter as a higher Chow cycle and applying Bloch's integral cycle class map (as realized in \cite{KLM,KL}), in Proposition \ref{reg-prop} we compute\footnote{This sort of computation has a long history (see, for instance, \cite{V} and \cite[\S4]{DK}). The main difference is that here we work integrally.}
their regulators as $F(u):=\frac{1}{2\pi i}\int_{\gamma^-}\mathrm{reg}(\alpha_u)=\frac{1}{2\pi i}\int_{\gamma_C^-}\mathrm{reg}(\{X,Y\}|_{C_{4u}})$ for $u\in \mathbb{R}_{\geq 1}$, where~$\gamma^-$, resp.~$\gamma_C^-$, generate the
subgroups of  anti-invariants
for the action of complex conjugation on~$H_1(E(\mathbb{C}),\mathbb{Z})^-$, resp.~$H_1(C_{4u}(\mathbb{C}),\mathbb{Z})^-$.
The resulting function 
\begin{equation} \label{R2udef}
F(u) = \log(4u)- \sum_{n=1}^\infty {2n  \choose n}^2 \frac{(4u)^{-2n}}{2n}
\end{equation}
has several interesting properties:
\begin{itemize}
\item [(a)] $F(u)$ is positive for $u\geq 1$, and computes the Beilinson regulator if additionally $4u\in\mathbb{Z}$;
\item [(b)] $u\frac{d}{du}F(u)=\,_2F_1(\frac{1}{2},\frac{1}{2};1;\frac{1}{u^2})$ is $\frac{1}{2\pi}$ times the real period of $\frac{dx}{y}$ on the Legendre elliptic curve $y^2=x(x-1)(x-\frac{1}{u^2})$, which is isomorphic to the $(-1)$-twist of $E_u$; and
\item [(c)] $F(u)$ is equal to the logarithmic Mahler measure of the Laurent polynomials $4u\pm(X+Y+X^{-1}+Y^{-1})$.
\end{itemize}
The Beilinson conjecture predicts that~$ H^2_{\mathcal{M}}(E_u,\Q(2))_\Z : = \INT E_u \otimes_\Z \Q \simeq \Q $, which by (a) (and Proposition~\ref{tameINT}) is necessarily spanned by the image of $\alpha_u$ for $4u\in\mathbb{Z}_{\geq 4}$.
It further predicts that~$\frac{L'(E_u,0)}{F(u)}$, and hence, by the functional equation of the~$ L $-function, also~$ \frac{L(E_u,2)}{(2 \pi i)^2 F(u)} $, is in~$ \Q^\times $.
(From (c) and \cite{RZ}, we even get a direct \emph{proof} that this number is $-1/165$ for $u=4$; see Remark \ref{remark:u=4}.)

The Bloch-Kato conjecture interprets this number, roughly speaking, as an `adelic covolume' of the motive $h^1(E_u)(2)$.
Using~$ \ell $-adic regulator maps for all prime
numbers~$ \ell $, it defines a subgroup~$H^2_{\mathcal{M}}(E_u,\Q(2))_{\Z,T}$
of~$H^2_{\mathcal{M}}(E_u,\Q(2))_\Z $ that contains the image of~$ \INT E_u $.
This subgroup in our case is supposed to
be isomorphic to~$ \Z $, and~$ \Z \alpha_u $ would be of finite index~$ \iota_u $
in it. The Bloch-Kato conjecture, formulated at~$s=2$,
predicts the supposedly rational number~$ \frac{\iota_u L(E,2)}{(2 \pi i)^2 F(u)} $ up to sign.

\medskip

The situation may seem hopelessly different from Dirichlet's.
For example,
\begin{enumerate}
\item[(1)] it is not known whether the rank of $H^2_{\mathcal{M}}(E_u,\Q(2))_{\Z,T}$ is 1,
let alone the truth of the Beilinson conjecture or the Bloch-Kato conjecture\footnote{Although
for some curves~$ C $ the expected relation between~$ L(C,2) $ and the Beilinson
regulator holds for certain elements of~$ \INT C \otimes_\Z \Q $ (see, e.g.,
\cite{Be86, dewi88, den89} and the seminal~\cite{Bl1}), it is not
known if this vector space is even finite dimensional.
For results on a weak form of the Bloch-Kato conjecture for elliptic curves with
complex multiplication we refer to \cite[Theorem~1.1.5]{Ki1} and \cite[Proposition~7.4]{BK},
or \cite[Theorem~2.2.2]{Ki2} for a survey,
noting that~$ \ell = 2 $ and~3 are excluded.};

\item[(2)] even if that rank is 1, we do not know whether a generator,
for each~$E_u$, spreads across the family, which
would be needed for uniform behaviour.
\end{enumerate}

We prove at least some results in this direction.
In order to avoid technical statements, in this introduction we summarise
our main results
(Theorem~\ref{new-index-theorem}, Corollary~\ref{new-index-cor}, Corollary~\ref{2torscor},
and Theorem~\ref{BK2})
mostly under stronger assumptions on~$ u $,
leaving the more general results to be found in the body of the paper.

In particular, the next theorem shows that the image of~$ \alpha_u $
under the 2-adic regulator map is non-zero, and provides a 
hint towards the expected finite generation of~$ K_2^T(E_u) $ and~$ \INT E_u $
modulo torsion, by bounding the 2-power divisibility of~$ \alpha_u $ modulo
torsion. Part~(3) may be viewed as a partial counterpart to ``key observation''\
(d) that the explicit unit $(u+\sqrt{u^2-4})/2$ is fundamental.

\medskip

{\bf Theorem 1.}
Let~$ u $ be a positive integer congruent to~4 modulo~8, and
such that~$ \ord_p(u(u^2-1))$ for each odd prime is zero or odd.
Let~$ m_u = 1 $ if $ u+1 $ has a prime factor congruent to~3 modulo~4,
and~2 otherwise.
Then the following hold.
\begin{enumerate}
\item
The image of $ \alpha_u $ in~$ H^1(\Q, \het^1(E_{u, \Qbar} , \Z_2(2))) $ modulo torsion
under the 2-adic regulator map is not divisible by~$ 2^{m_u} $.

\item
$ \alpha_u $ is not in~$ 2^{m_u} K_2^T(E_u) + K_2^T(E_u)_\tor $.

\item
$ 2 \alpha_u $ is in~$ \INT E_u $ but not in~$ 2 \INT E_u + \INT E_u {}_{,\tor} $.
\end{enumerate}

Computationally, approximately~58.5\% of all~$ u \equiv 4 $ modulo~8
in~$ \{ 4 , \dots, 10^9 \} $ satisfy the conditions 
in the theorem (and for approximately~50.2\% of the~$ u $ in this
range the resulting~$ u (u^2-1)/4 $ is squarefree, as required in Proposition~\ref{Vu-prop}).

In order to weaken our assumptions, we concentrate on the 2-part
of the Bloch-Kato conjecture, in the following sense. We still
assume that $\frac{L'(E_u,0)}{F(u)}$ is a non-zero rational number, and that~$ \INT E_u \otimes_\Z \Q $ is 1-dimensional, but in
this 1-dimensional~$ \Q $-vector space define a rank one~$ \Z_{(2)} $-submodule
using only the~2-adic regulator map (see~\eqref{MMTell-def}).
(Note that~$ \Z_{(2)} $ is the localisation of~$ \Z $ at the maximal ideal~$ (2) $.)
For~$ u $ as in Theorem~1, $ \Z_{(2)} \alpha_u $ is then
of finite index~$ \iota_{u,2} $ in the rank one~$ \Z_{(2)} $-submodule, and~$ \iota_{u,2} $ divides~$ 2^{m_u-1} $.
In particular, if~$ m_u = 1 $ then~$ \iota_{u,2} = 1 $ and
$ \Z_{(2)} \alpha_u $ is the full submodule.
If the construction of Bloch-Kato leads to a rank~1 $ \Z $-lattice,
we have tensored it with~$ \Z_{(2)} $, and~$ \iota_{u,2} $
is the 2-part of~$ \iota_u $.

We emphasise that~$ 2 \alpha_u $ has 2-divisible image
under the 2-adic regulator map, but, by part~(3) of the above theorem,
is not itself 2-divisible in~$ \INT E_u $ modulo torsion.
Thus, from the point of view of~$ K $-theory it would be natural to
consider the~$ \Z_{(2)} $-module of~$ \INT E_u \otimes_\Z \Q $
generated by~$ 2 \alpha_u $, but
this is of non-trivial index~$ 2 \iota_{u,2} $ in the module
used in this formulation of the 2-part of the Bloch-Kato conjecture.

Recall part of ``key observation'' (c), that the classes of the ramified primes generate
a subgroup of index~1 or~2 in the 2-torsion in $\mathrm{Cl} (\mathcal{O}(K_u))$,
and that those classes are subject to only one or two relations. Viewing this
a different way helps us to understand in what sense this carries over
to~$E_u$ (where of course $u$ does not mean quite the same
thing). By class field theory, we may view~$\mathrm{Cl}^+ (\mathcal{O}(K_u)) $
as the Galois group~$ G^+ $ of the ray class field of~$K_u=\Q(\sqrt{d})$ with
modulus~$ \infty_1+\infty_2 $, and its homomorphisms to~$ \Z/2\Z $ are 
precisely the continuous homomorphisms~$\Gal(\Qbar/K_u) \to \Z/2\Z $ that
are unramified at every finite place. Using that~$ \Z/2\Z \simeq \{\pm1\} $, these can be described using
Kummer theory.

If~$ p $ is an (odd) prime dividing~$ d $, and $ p^* = \pm p $ is congruent to~1
modulo~4, then~$ \Q(\sqrt{p^*}) /\Q $ is unramified at all finite
primes other than~$ p $, and it is easy to see that~$ \Q_p(\sqrt d, \sqrt{p^*}) / \Q_p(\sqrt d) $
is also unramified. Thus~$ p^* $ defines a homomorphism~$ G^+ \to \Z/2\Z $,
unramified at every finite prime.
Because the kernel of the natural homomorphism~$ \Q^\times/2 \to K_u^\times/2 $
is generated by the class of~$ d $, these homomorphisms generate
exactly~$ \Hom(G^+, \Z/2\Z) \simeq (\Z/2\Z)^{r-1} $.
Under class field theory, the class group~$ \mathrm{Cl}(\mathcal{O}(K_u)) $
corresponds to the quotient group~$ G $ of~$ G^+ $ that is the
Galois group of the Hilbert class field of~$ K_u $. 
The corresponding subgroup~$ \Hom(G, \Z/2\Z) $ of $ \Hom(G^+, \Z/2\Z) $
is obtained by using only products of the~$ p^* $ that are positive,
as they avoid ramification at~$ \infty_1 $ and~$ \infty_2 $,
giving the trivial homomorphism on~$ \Gal(\C/\R) $.
This is no restriction precisely when~$ p^* = p $ always, i.e., the prime divisors of~$ d $ are congruent to~1 modulo$~4 $.

Using \cite[Example~3.9]{BK}, one sees that the group $ \Hom(G, \Z/2\Z) = H^1(G, \Z/2\Z)$
is the 2-torsion subgroup of~$ H_f^1(K_u, \Z/2^\infty\Z) $ as
defined in loc.\ cit. It should thus be considered the analogue
of the 2-torsion in~$H_f^1(\Q, E_u[2^{\infty}](-1))$, which we shall define in Section~\ref{selmer}.

We can construct 2-torsion classes in $H^1(\Q, E_u[2^{\infty}](-1))$, unramified at all odd primes
where~$ E_u $ has good reduction, 
in a way similar to that for the class group discussed above, since $E_u[2](-1)\simeq \Z/2\Z\oplus\Z/2\Z$. But which of these satisfy the local conditions at the
remaining primes, for membership of~$ H_f^1(\Q, E_u[2^\infty](-1) ) $, depends on the
larger group~$E_u[2^\infty](-1)$, and is much more subtle.
This is evident in the following theorem, and also in the numerical examples in (4) of Example~\ref{exo-special}, where there are lots of bad primes but~$ H_f^1(\Q, E_u[2^\infty](-1) ) $ is trivial.
Thus the simple relationship between the number of ramified primes and the order of $\mathrm{Cl} (\mathcal{O}(K_u))[2]$ is not emulated for the $2$-torsion in $ H_f^1(\Q, E_u[2^\infty](-1) ) $.

In order to make explicit the 2-part of the Bloch-Kato conjecture,
we have to assume that the 2-Selmer group~$ H^1_f(\Q, E_u[2^{\infty}](-1)) $
is finite. It is a 2-power torsion group,
and we can compute its 2-torsion using the next theorem.
In practice this leads to many~$ u $ for which this 2-torsion
is trivial, thus \emph{proving} that~$  H_f^1(\Q, E_u[2^\infty](-1) ) $ is trivial
for such~$ u $.

\medskip

{\bf Theorem 2.}
For~$ u $ as in Theorem~1,
let~$ S $ be the set of prime divisors of~$ u^2-1 $, and~$ S' $
the set of prime divisors of~$ u $ that are congruent to~1 modulo~4.
Then the 2-torsion in~$ H_f^1(\Q, E_u[2^\infty](-1) ) $ is in bijection
with pairs~$ (D, D') $ of positive squarefree integers, where
the prime factors of~$ D $ are in~$ S $ and those of~$ D' $ in~$ S' $,
and which satisfy
\begin{itemize}
\item
$ D' $ is a square modulo~$ p $ for every~$ p $ in~$ S $;
	
\item
$ 2 ^{\ord_p(D')} D $ is a square modulo~$ p $ for every~$ p $ in~$ S' $;
	
\item
$ D \equiv 1 $ modulo~8.
\end{itemize}

\medskip

Making the 2-part of the Bloch-Kato conjecture explicit now
gives the following.
For~$ u = 4 $ some of the assumptions in the theorem are
known (see Remark~\ref{uis4}).

\medskip

{\bf Theorem 3.}
Let~$ u $ be as in Theorem~1.
Assume that~$ \INT E_u \otimes_\Z \Q $ is 1-dimensional and that~$ \frac{ L(E_u, 2)}{(2\pi i)^2 F(u)} $
is a non-zero rational number~$ q_u $.
Let~$\iota_{u,2}$ be as described after Theorem~1, and let~$ \omega(n) $ (and~$\omega_1(n)$, respectively~$\omega_3(n)$)
denote the number of
distinct prime divisors of~$n$ (or only those that are congruent to~1, respectively~3, modulo~4).
Assume that~$ H^1_f(\Q, E_u[2^{\infty}](-1)) $ is finite, and set $s_u=\ord_2(\# H^1_f(\Q, E_u[2^{\infty}](-1)))$. Then the Bloch-Kato conjecture predicts that
\begin{equation} \label{intro-BK}
\ord_2(q_u) + \ord_2(\iota_{u,2}) + 2 = 2 \omega_1(u) + \omega_3(u) + \omega(u^2-1) + s_u.
\end{equation}

\medskip
The first three terms on the right-hand side come from $2$-parts of Tamagawa factors, whose analogues in the quadratic field/class group case would be trivial. They account for the fact that the equality between the $2$-parts of the $L$-ratio and the order of the class group does not carry across to the elliptic curve case; there is no such simple relationship between $\ord_2(q_u)$ and $s_u$.
(See also the numerical examples in (2) and (4) of Example \ref{exo-special}.)

We do not find any evidence that~$ \iota_{u,2} = 2 $ occurs in our family.
Thus, the truth of Bloch-Kato would suggest that the minimum of~$ \ord _2 (q_u)$
can only be attained when the number of prime divisors of~$ \frac14 u(u-1)(u+1) $,
which are precisely the places of bad reduction of $E_u$, is
minimal, and~$ s_u = 0 $.

If~$u $ is as in the above theorems, and~$ u (u^2-1)/4 $ is squarefree,
then, apart from~$ u = 4 $ or~12, the predicted value for~$ \ord_2(q_u)$
is at least~2, because~$\frac14 u(u-1)(u+1)$ has at least~4 prime divisors.
Indeed, numerical computation for such~$ u $,
with~$ 12 < u < 25000$, suggests the minimal 
2-valuation is~2, attained for $u = 228 $, 1668, 3252, 4548, 8292, 8628, 9012,
10068, 12612, 17988, 18132 and~19428.
(We refer to Section~\ref{num-sec} for some of this data.)
Each such~$u$ has the property 
that $u-1$,  $u+1$ and $u/12$ are prime numbers, with the last
congruent to~3 modulo~4, and using Theorem~2
above one then sees that~$ H_f^1(\Q, E_u[2^\infty](-1) ) $ is
trivial, so~$ s_u = 0 $.
By Proposition~\ref{Vu-prop}, such~$ u $ are precisely those
for which the right-hand side of~\eqref{intro-BK} has value~4, which would
match the prediction together with~$ \ord_2( \iota_{u,2} ) = 0 $.
(We refer to Remark~\ref{minimalremark} for a precise discussion,
and to Remark~\ref{notsquarefree} for some additional examples
when~$ u $ only satisfies the condition of the theorems.)

Thus we see that ``key observations'' (a) and (b) conjecturally
carry over\footnote{By contrast, recall that~(c) does not carry over so well, and that for (d) we know only 2-indivisibility.}
pretty well, although for~(a) not only~$ s_u $ is
involved. As to~(b), in fact, the philosophy that the minimum should be obtained
among the~$ u $ for which~$ E_u $ has the smallest
number of places of bad reduction, was the starting point for doing
numerical calculations, and with that for this paper.

\medskip

{\bf The structure of the paper.}
We start by studying the elliptic curve $E_u$, its torsion and its reduction type in Section~\ref{basic-section}. We review the content of the Beilinson conjecture
for the motivic cohomology group $H^2_{\MMM}(E_u, \Q(2)) = K_2^T(E_u) \otimes_\Z \Q$ in Section~\ref{regBeil}, and the Bloch-Kato conjecture in Section~\ref{BlKa}. 
In Sections~\ref{tam-away} and~\ref{2Tam2} we study the various factors appearing in the statement of the
2-part of the Bloch-Kato conjecture as in~\eqref{BKconj1}: the
global Galois invariants $H^0(\Q, E_u[2^{\infty}](1))$ and $H^0(\Q, E_u[2^{\infty}](-1))$,
the 2-Tamagawa factors away from 2, and
the 2-Tamagawa factor at~2. We address the Bloch-Kato 2-Selmer group in Section~\ref{selmer}, describing its
2-torsion precisely in 
Corollary~\ref{2torscor}, as quoted in Theorem~2 above.   

After preparatory work on real regulators in Section~\ref{K2reg}, 
we define  $K_2^T(E_u)$, the kernel of the tame symbol $T$, from the exact localisation
sequence 
\begin{equation*}
\dots \to \oplus_P K_2(\Q(P)) \to K_2(E_u) \to K_2(\Q(E_u)) \overset{T}{\to} \oplus_P \, \Q(P)^\times \to \dots
\,,
\end{equation*}
where~$ P $ runs through the closed points of~$ E_u $.
It appears together with the $\ell$-adic regulator map~$ \reg_{\ell} $
and the~$ \ell $-adic Chern class map~$ \mathrm{ch}_{\ell} $ in the
commutative diagram
\begin{equation*}
\xymatrix{
K_2(E_u) \ar[r] \ar[d]_-{\mathrm{ch}_{\ell}} & K_2^T(E_u) \ar[d]^{\reg_{\ell}}
\\
\het^2(E_u, \Z_\ell(2)) \ar[r]^-{\pi_\ell} & H^1(\Q, \het^1(E_{u,\Qbar}, \Z_\ell(2)))_\tf
}
\end{equation*}
in Section~\ref{ellreg}. An integral version $K_2^T(E_u)_\text{int}$ is obtained as 
the image in~$ K_2(\Q(E_u)) $ of~$ K_2(\mathcal{E}_u) $ when
localising a flat, proper, integral model $ \mathcal{E}_u $ over~$ \Z $
to its generic point.
We construct the element $ \alpha_u $ in~$ K_2^T(E_u) $, show that $2 \alpha_u$ is in~$ K_2^T(E_u)_\text{int}$,
and bound its divisibility by powers of~2, and that of its image
under the 2-adic regulator, in Section~\ref{K2elements}.
For this we introduce a new technique, by combining  the norm
residue isomorphism of Merkur'ev-Suslin, as well as the 2-adic
regulator map, with pullbacks to closed points on the curve.

\smallskip

In Section \ref{hyper-sec}, we relate Boyd's family
$C_k : X+Y+X^{-1}+Y^{-1}=k$ to the one in~\eqref{Eueq} by a 2-isogeny,
prove that the $ K_2 $-symbol $\{X,Y\}$ on~$ C_{4u} $ pulls back to $-2 \alpha_u$
on~$ E_u $, and from this obtain a formula for the real regulator
of~$ \alpha_u $ that involves the hypergeometric formula~\eqref{R2udef}.
Putting everything together in Section~\ref{2-BK}, we arrive at our final theorem,
Theorem~\ref{BK2}.
We conclude the paper with the results of numerical experiments in Section~\ref{num-sec},
where we also discuss when the rational number in the Bloch-Kato
conjecture can have minimal 2-valuation.

\medskip

{\bf Notation.}
If~$ A $ is an Abelian group, then we let~$ A_\tor $ be
its torsion subgroup, and~$ A_\tf = A / A_\tor $ its maximal torsion
free quotient group.

\bigskip \bf Acknowledgements. \rm Discussions with Jan Stienstra of
his two papers \cite{Stien, Stienbis}, and his
insistence on the interpretation of the integrated periods in Boyd's
family as a $K_2$-regulator, influenced one of the authors in the
2000s.
We thank Matthias Flach for useful correspondence about Section~\ref{selmer}.

\section{The elliptic curve \texorpdfstring{$E_u$}{Eu}} \label{basic-section}
For~$u \ne 0, \pm1$ in~$ \Q $, let $E_u/\Q$ be the elliptic curve defined by the Weierstrass equation~\eqref{Eueq}.
Elementary computations give the following.

\begin{prop}\label{torsion}
(1)
The $2$-torsion of $E_u$ is rational; in fact, $$E_u[2]=\{O, (0,0), (-1,0), (-u^2,0)\}.$$

(2)
There are precisely four rational points of order $4$, namely $(u,\pm u(u+1))$ and $(-u, \pm u(u-1))$, forming a single coset of $E_u[2]$. All satisfy $2P=(0,0)$.

(3)
The $4$-torsion subgroup $E_u[4]$ is generated by $$(u, u(u+1))\,\,\text{and}\,\,(-u^2+u\sqrt{u^2-1}, i(u(u^2-1)-u^2\sqrt{u^2-1})).$$ If $Q$ is the second given generator then $2Q=(-u^2, 0)$.

(4)
The points of order~4 form the three cosets~$ (u,\pm u(u+1)) + E_u[2] $,
$ Q + E_u[2] $ and~$ ( -1 + \sqrt{1-u^2}, i ( u^2-1 + \sqrt{1-u^2} ) ) + E_u[2] $,
with elements in the last coset doubling to $ (-1,0) $.
\end{prop}

If~$ u $ is an integer, then we can also determine the reduction type of $E_u$ at every prime, subject to a modest assumption to simplify what happens at $p=2$.

\begin{prop}\label{minimal}
Let~$ u $ be an integer with~$ |u| > 1 $.
\begin{enumerate}
\item $E_u$ has multiplicative reduction at any odd prime dividing $u(u-1)(u+1)$, and good reduction at all other odd primes.
\item At an odd prime $p\mid u$, or at an odd prime $p\mid (u^2-1)$ such that $p\equiv 1\pmod{4}$, the multiplicative reduction is split. At an odd prime $p\mid (u^2-1)$ such that $p\equiv 3\pmod{4}$, it is non-split.
\item Assume that $4\parallel u$. Then letting $y=8y'+4x'$ and $ x=4x'$, one obtains a global minimal Weierstrass equation $${y'}^2+x'y'={x'}^3+\frac{u^2}{4}{x'}^2+\frac{u^2}{16}x',$$
with minimal discriminant $\Delta=\left(\frac{u}{4}\right)^4(u^2-1)^2$. In particular, $E_u$ has good, ordinary reduction at $2$.
\item If $4\mid\mid u$ then the conductor of $E_u$ is $N_u=\underset{\substack{\text{odd primes}\\ p\mid u(u^2-1)}}{\prod\,\,p}$.
\end{enumerate}
\end{prop}
\begin{proof}
(1)
With notation as in \cite[III.1]{S1}, the equation $y^2=x(x+1)(x+u^2)$ has associated quantities $c_4=16(1-u^2+u^4)$ and $\Delta=16u^4(u^2-1)^4$, whose greatest common divisor is visibly $16$. Hence if $p\mid u(u^2-1)$ is an odd prime then~$p\nmid c_4$, so the equation is minimal at $p$, and $E_u$ has multiplicative reduction at~$p$.

(2)
Modulo an odd prime $p\mid u$, the equation is $y^2\equiv x^2(x+1)$, with singular point $(0,0)$. The quadratic part is $y^2-x^2\equiv 0$, which factors as $(y-x)(y+x)\equiv 0$, giving the equations of the two ``tangent lines'' at the node.

Modulo an odd prime $p\mid u^2-1$, the equation is $y^2\equiv x(x+1)^2$, with singular point $(-1,0)$. About this point, the equation is $y^2\equiv [(x+1)-1](x+1)^2$, with quadratic part $y^2+(x+1)^2\equiv 0$. This factors over $\FF_p$ if and only if $p\equiv 1\pmod{4}$.

(3)
It is straightforward to verify the statement on the minimality
of the discriminant, and that~$ y'^2 + x' y' = x'^3 + x' $ defines
an elliptic curve over~$ \F_2 $ with as~$ \F_2 $-rational points~$O, (0,0), (1,0)$ and $ (1,1) $.

(4) is a direct consequence of the first three parts, cf.  \cite[p.256]{S1}.
\end{proof}

\section{The regulator and Beilinson's conjecture}\label{regBeil}

The $L$-function of $E_u$ is defined by the Euler product
$L(E_u,s)=\prod_p L_p(E_u,s)$, where 
$$L_p(E_u, s):=\begin{cases} (1-a_p p^{-s}+p^{1-2s})^{-1} & \text{for good reduction at $p$};\\(1-p^{-s})^{-1} & \text{for split multiplicative reduction at $p$;}\\(1+p^{-s})^{-1} & \text{for non-split multiplicative reduction at $p$,}\end{cases}$$
and in the case of good reduction $\#E(\FF_p)=1+p-a_p$. Since $|a_p|<2\sqrt{p}$, the series converges for $\Re(s)>\frac{3}{2}$, but by modularity of $E_u$ it has an analytic continuation to the whole complex plane.
If $N_u$ is the conductor then the function~$\Lambda(s):=N_u^{s/2}(2\pi)^{-s}\Gamma(s)L(E_u,s)$
satisfies a functional equation $\Lambda(2-s)=\pm\Lambda(s)$.

The famous conjecture of Birch and Swinnerton-Dyer equates the order of vanishing of $L(E_u,s)$ at $s=1$ with the rank of the finitely
generated Abelian group $E_u(\Q)$, and gives a conjectural formula for the leading term in its Taylor expansion about $s=1$.
We shall look instead at $L(E_u,2)$, an example of a non-critical value, and see what the Bloch-Kato conjecture \cite{BK, F} predicts about it.

The terms appearing in the conjecture are associated with the second Tate twist $h^1(E_u)(2)$ of the Grothendieck motive for the first cohomology of $E_u$. Betti, de Rham and $\ell$-adic realisations of $h^1(E_u)$ are $H^1_B(E_u(\C), \Q)$, $H^1_{\dR}(E_u/\Q)$ and $H^1_{\text{\'et}}(E_{u,\Qbar}, \Q_{\ell})$, the singular, algebraic de Rham and \'etale $\ell$-adic  first cohomology groups. A Tate twist by $r\in\Z$ involves multiplying the coefficients of the Betti realisation by $(2\pi i)^r$, shifting the numbering in the Hodge filtration of the de Rham realisation by $r$, and multiplying the representation of $\Gal(\Qbar/\Q)$ on each $\ell$-adic realisation by the $r^{\mathrm{th}}$ power of the $\ell$-adic cyclotomic character. In particular, for a prime $p \ne \ell$, the Tate twist does not influence the action of the ramification subgroup, and the action of any Frobenius element $\Frob_p$ is multiplied by $p^r$.

For use in the next section, we fix the~$\Z$-lattice $H^1_B(E_u(\C), \Z)$ in $H^1_B(E_u(\C), \Q)$, which,
under the comparison isomorphisms $H^1_B(E_u(\C), \Q)\otimes \Q_{\ell}\simeq H^1_{\text{\'et}}(E_{u,\Qbar}, \Q_{\ell})$,
is compatible with the~$\Gal(\Qbar/\Q)$-invariant $\Z_{\ell}$-lattices $T_{\ell}:=H^1_{\text{\'et}}(E_{u,\Qbar}, \Z_{\ell})$ in $V_{\ell}:=H^1_{\text{\'et}}(E_{u,\Qbar}, \Q_{\ell})$.
We also choose the $\Z$-lattice $H^1_{\dR}(\mathcal{E}_u/\Z)$ in $H^1_{\dR}(E_u/\Q)$, where~$\mathcal{E}_u$ is a minimal proper, flat, regular model.

According to a special case of a conjecture of Beilinson \cite{Be}, the value $L(E_u,2)$ is given, up to a non-zero rational multiple, by a certain regulator. The Bloch-Kato conjecture then pins down the rational multiple, up to sign. (We should observe that Beilinson's conjecture was partly inspired by work of Bloch \cite{Bl1}, and corrected following work of Bloch and Grayson \cite{BG}, both on values of $L$-functions of elliptic curves at $s=2$.)
Before stating what the Bloch-Kato conjecture predicts for $L(E_u,2)$, we take a preliminary look at the definition of the regulator. It is the determinant of a certain map from motivic cohomology to Deligne cohomology.

First, the motivic cohomology $H^2_{\MMM}(E_u, \Q(2))$. More generally we would be looking at $H^{m+1}_{\MMM}(X, \Q(r))$ if we were dealing with $h^m(X)(r)$, i.e., with $L(X,r)$, where~$X$ is a nonsingular projective variety over $\Q$, $L(X,s)$ is the $L$-function obtained from the Galois representation on the $m^{\mathrm{th}}$ $\ell$-adic cohomology of $X$, and $r>\frac{m}{2}+1$. Beilinson defines $H^{m+1}_{\MMM}(X, \Q(r))$ in terms of a graded piece of
the $ \gamma $-filtration on algebraic $K$-theory, as $\mathrm{gr}^r(K_{2r-m-1}(X)\otimes\Q)$. Bloch \cite{Bl2} provided an alternative construction $H^{m+1}_{\MMM}(X, \Q(r))=\mathrm{CH}^r(X, 2r-m-1)\otimes\Q$, and proved its equivalence to the other one. Here $\mathrm{CH}^r(X, 2r-m-1)$ is a Bloch higher Chow group (with integer coefficients), a certain set of equivalence classes of algebraic cycles of codimension~$r$, defined over $\Q$, on the product of $X$ with an affine space or algebraic simplex of dimension $2r-m-1$. (See \cite[\S 1.1]{DK} and \cite[\S 1]{DS} for slightly different versions.) In our case we are looking at $H^2_{\MMM}(E_u, \Q(2))$, i.e., at $\mathrm{gr}^2(K_2(E_u)\otimes\Q)$ or $\mathrm{CH}^2(E_u, 2)\otimes\Q$.
We shall use mostly the former (see Sections~\ref{K2reg} through~\ref{K2elements}), but shall also use the relation with the latter in
Section~\ref{hyper-sec}.

A useful exposition of the construction of Deligne cohomology $H^{m+1}_{\DDD}(X_{\R}, \R(r))$, and its key properties, may be found in \cite[\S 2]{DS}. For $r>\frac{m}{2}+1$, there is an isomorphism \cite[(2.3.1)$\DDD$]{DS}
\begin{equation}\label{D-iso} H^{m+1}_{\DDD}(X_{\R}, \R(r))\simeq \frac{H^m_{\dR}(X_{\R})/\Fil^r H^m_{\dR}(X_{\R})}{H^m_B(X(\C), \R(2\pi i)^r)^+},\end{equation}
where the ``$+$'' indicates the fixed subspace under the simultaneous action of complex conjugation on $X(\C)$ and on $(2\pi i)^r$. Note that $\Fil^r H^m_{\dR}(X_{\R})$ is $\Fil^0$ of the de Rham realisation of $h^m(X)(r)$, tensored with $\R$. For us, since $\Fil^2 H^1_{\dR}(E_u)=\{0\}$, we have
\begin{equation*}
H^2_{\DDD}(E_{u,\R}, \R(2))\simeq \frac{H^1_{\dR}(E_{u,\R})}{(2\pi i)^2H^1_B(E_u(\C), \R)^+}
.
\end{equation*}

There is a regulator map
\begin{equation} \label{regdef}
\reg: H^{m+1}_{\MMM}(X, \Q(r))\rightarrow H^{m+1}_{\DDD}(X_{\R}, \R(r))
.
\end{equation}

For Beilinson this is a generalised Chern character in algebraic $K$-theory, for Bloch a generalised cycle map \cite{Bl3},\cite[\S 2.8]{DS}, \cite[\S 1.2]{DK}.
One can define a $\Q$-vector subspace $H^{m+1}_{\MMM}(X, \Q(r))_{\Z}$ of ``integral'' elements as the image of $H^{m+1}_{\MMM}(\Chi, \Q(r)):=\mathrm{gr}^r(K_{2r-m-1}(\Chi)\otimes\Q)$
in $H^{m+1}_{\MMM}(X, \Q(r))= \mathrm{gr}^r(K_{2r-m-1}(X)\otimes\Q)$
under localisation, where $\Chi/\Z$ is a proper, flat, regular model of $X/\Q$.
This is independent of any choice of $\Chi$.
If~$ X $ is a curve then such a model exists, but this is not
known for higher dimensional varieties. However, 
Scholl in \cite[\S1]{scho00}
used alterations to define a subspace 
such that it coincides with what comes from such a model~$ \Chi $ if
it exists. (See \cite[pp.1--2]{dJ08} for a summary of his construction.)

Conjecturally, the regulator map induces an isomorphism
\begin{equation} \label{Bei-conj}
\reg: (H^{m+1}_{\MMM}(X, \Q(r))_{\Z})\otimes\R\simeq H^{m+1}_{\DDD}(X_{\R}, \R(r))
,
\end{equation}
hence an isomorphism between highest exterior powers
$$\det(\reg): (\det H^{m+1}_{\MMM}(X, \Q(r))_{\Z})\otimes\R\rightarrow\det H^{m+1}_{\DDD}(X_{\R}, \R(r)).$$
Using (\ref{D-iso}), we get an isomorphism
$$\det H^{m+1}_{\DDD}(X_{\R}, \R(r))\simeq \det (H^m_{\dR}(X_{\R})/\Fil^r H^m_{\dR}(X_{\R}))\otimes{\det}^{\vee}(H^m_B(X(\C), \R(2\pi i)^r)^+),$$
and in the right-hand side here we see a rational line
$$D_{m,r}:=\det (H^m_{\dR}(X_{\Q})/\Fil^r H^m_{\dR}(X_{\Q}))\otimes{\det}^{\vee}(H^m_B(X(\C), \Q(2\pi i)^r)^+).$$
According to Beilinson's conjecture \cite[3.1.2]{DS},
\begin{equation} \label{Qlines}
\det(\reg)(\det H^{m+1}_{\MMM}(X, \Q(r))_{\Z})=L(X,r) D_{m,r}
.
\end{equation}
This determines $L(X,r)$ as an element of $\R^{\times}/\Q^{\times}$.

\section{The Bloch-Kato conjecture for \texorpdfstring{$L(E_u, 2)$}{L(Eu,2)}}\label{BlKa}

To remove the $\Q^{\times}$-ambiguity (at least up to sign) implied
by~\eqref{Qlines}, we need to replace the $\Q$-lines by $\Z$-lattices. For simplicity we specialise to the case $L(E_u,2)$, but the generalisation to $L(X,r)$ is obvious.
On the right-hand side this is simple enough. We just define
$$\DDD_{1,2}:=\det (H^1_{\dR}(\mathcal{E}_u/\Z))\otimes{\det}^{\vee}(H^1_B(E_u(\C), \Z(2\pi i)^2)^+).$$
On the left-hand side (of~\eqref{Qlines}), we already saw the ``real'' regulator map in the previous section, and it is natural now to employ also $\ell$-adic regulator maps for finite primes~$\ell$. Recall that in the previous section we set~$T_{\ell}=H^1_{\text{\'et}}(E_{u,\Qbar}, \Z_{\ell})$ and~$V_{\ell}=H^1_{\text{\'et}}(E_{u,\Qbar}, \Q_{\ell})$,
and define, for any finite prime $\ell$,
\begin{equation} \label{MMTell-def}
H^2_{\MMM}(E_u, \Q(2))_{\Z, T_{\ell}} := \reg_{\ell,\Z}^{-1}(H^1(\Q, T_{\ell}(2))_\tf)
,
\end{equation}
where~$\reg_{\ell,\Z}: H^2_{\MMM}(E_u, \Q(2))_{\Z}\rightarrow H^1(\Q, V_{\ell}(2))$ (continuous Galois cohomology) is
the composition of the inclusion~$H^2_{\MMM}(E_u, \Q(2))_{\Z} \subseteq H^2_{\MMM}(E_u, \Q(2)) $ with
the~$ \ell $-adic regulator map~$ \reg_{\ell}^{\Q_\ell} :H^2_{\MMM}(E_u, \Q(2)) \to H^1(\Q, V_{\ell}(2)) $.
More details on~$\reg_{\ell}^{\Q_\ell}$ will be given in Sections~\ref{ellreg}
and~\ref{2-BK}, but we observe here that by \cite[Proposition~2.3]{Ta}
and the discussion following it, the map~$ H^1(\Q, T_{\ell}(2))_\tf \to H^1(\Q, V_{\ell}(2)) $
is injective with torsion cokernel, so that~$  H^1(\Q, V_{\ell}(2)) =  H^1(\Q, T_{\ell}(2)) \otimes_{\Z_{\ell}} \Q_{\ell} $.

In the $\Q$-vector space $H^2_{\MMM}(E_u, \Q(2))_{\Z}$,
which we assume to have~$ \Q $-dimension~1 because
of the conjectured isomorphism in~\eqref{Bei-conj},
we might now attempt to define a $\Z$-lattice
\begin{equation} \label{MMT-def}
H^2_{\MMM}(E_u, \Q(2))_{\Z, T}:=\cap_{\ell} \, H^2_{\MMM}(E_u, \Q(2))_{\Z, T_{\ell}} 
,
\end{equation}
(cf.\ \cite[\S 5]{BK}, where our $H^2_{\MMM}(E_u, \Q(2))_{\Z}$ is their~$\Phi$ and our $H^2_{\MMM}(E_u, \Q(2))_{\Z, T}$ is isomorphic to their $A(\Q)_{\tf}$.)
Although $H^2_{\MMM}(E_u, \Q(2))_{\Z, T}$ is a well-defined $\Z$-module, it is not clear that it is a~$\Z$-lattice, not least because local conditions are imposed at infinitely many $\ell$. That it is would follow from \cite[Conjecture 5.3]{BK}. 
For an alternative viewpoint and construction, see \cite[\S 11.6]{F}, where our $H^2_{\MMM}(E_u, \Q(2))_{\Z, T}$ ought to be an instance of his~$H^1_f(\Q, (M,\Theta))$.
Assuming that~$H^2_{\MMM}(E_u, \Q(2))_{\Z, T}$ is a $\Z$-lattice, and that~$\reg: (H^2_{\MMM}(E_u, \Q(2))_{\Z})\otimes\R\simeq H^2_{\DDD}(E_{u,\R}, \R(2)),$ we define a regulator $R_u$, up to sign, by
\begin{equation} \label{Ru-def}
\det(\reg)(\det(H^2_{\MMM}(E_u, \Q(2))_{\Z, T}))=R_u \,\DDD_{1,2}
.
\end{equation}
The Bloch-Kato conjecture gives a formula for the (conjecturally) rational number $\frac{L(E_u,2)}{R_u}$,
up to sign, which is equivalent to specifying $\ord_{\ell}\left(\frac{L(E_u,2)}{R_u}\right)$, for all finite
primes~$\ell$. It says that, for each~$ \ell $,
\begin{equation} \label{BKconj1}
\ord_{\ell}\left(\frac{L(E_u,2)}{R_u}\right)=\ord_{\ell}\left(\frac{\prod_{p\leq\infty}\mathrm{Tam}^0_{p,\omega}(T_{\ell}(2))\,\# H^1_f(\Q, (V_{\ell}/T_{\ell}))}{\#H^0(\Q, (V_{\ell}/T_{\ell})(2))\,\#H^0(\Q, (V_{\ell}/T_{\ell}))}\right),
\end{equation}
which contains some as yet undefined terms.
We shall discuss~$ \mathrm{Tam}^0_{p,\omega} $ for~$ p \ne 2 $ a
finite prime where~$ E_u $ has good reduction, or $ p = \infty $,
later in this section. For~$ p $ an odd prime where~$ E_u $ has bad reduction
we do so in Section~\ref{tam-away}, for~$ p = 2 $ in Section~\ref{2Tam2},
and we treat~$ H^1_f(\Q, (V_{\ell}/T_{\ell})) $ in Section~\ref{selmer}.
But first, we weaken our assumptions.

Since we are only looking at the $\ell$-part of the Bloch-Kato conjecture for each~$ \ell $
separately, we can weaken the assumption that~\eqref{MMT-def} defines a~$ \Z $-lattice.
For this, we observe that~\eqref{MMTell-def} defines a $\Z_{(\ell)}$-module,
where $\Z_{(\ell)}$ is the localisation (not the completion) of $\Z$ at the prime ideal $(\ell)$.
Then in~\eqref{Ru-def} we tensor the right-hand side with~$ \Z_{(\ell)} $,
and in the left-hand side replace~$H^2_{\MMM}(E_u, \Q(2))_{\Z, T}$ by~$H^2_{\MMM}(E_u, \Q(2))_{\Z, T_{\ell}}$.
Assuming~$H^2_{\MMM}(E_u, \Q(2))_{\Z, T_{\ell}}$ is a~$\Z_{(\ell)}$-lattice in the $\Q$-vector space $H^{2}_{\MMM}(E_u, \Q(2))_{\Z}$,
we replace~$ R_u $ by~$ R_{u,\ell}$
in~$ \R^\times $, which is well-defined up to multiplication
by~$\Z_{(\ell)}^{\times}$.
If~\eqref{MMT-def} is indeed a~$\Z$-lattice
then~$H^2_{\MMM}(E_u, \Q(2))_{\Z, T_{\ell}} = H^2_{\MMM}(E_u, \Q(2))_{\Z, T} \otimes_\Z \Z_{(\ell)} $
is a~$\Z_{(\ell)}$-lattice, and~$ \ord_\ell( R_u) = \ord_\ell(R_{u,\ell})$.
So we have weakened our assumption for considering~\eqref{BKconj1}, but we
still assume~$\reg: (H^{2}_{\MMM}(E_u, \Q(2))_{\Z})\otimes\R\simeq H^{2}_{\DDD}(E_{u,\R}, \R(2))$
is an isomorphism, and that~$\frac{L(E_u, 2)}{R_{u,\ell}}$ is a rational number.

Note that if~$H^2_{\MMM}(E_u, \Q(2))_{\Z} $ has finite~$ \Q $-dimension,
then~$H^2_{\MMM}(E_u, \Q(2))_{\Z, T_{\ell}}$ is not a lattice
precisely when some non-zero element
in~$H^2_{\MMM}(E_u, \Q(2))_{\Z, T_{\ell}}$ has infinitely $\ell$-divisible image in~$H^1(\Q, T_{\ell}(2))_\tf$ under~$ \reg_{\ell,\Z} $.
Because we are assuming that the~$ \Q $-dimension is~1, we obtain
a lattice if and only if~$ H^2_{\MMM}(E_u, \Q(2))_{\Z, T_{\ell}} \simeq \Z_{(\ell)} $,
and no non-zero element has infinitely~$ \ell $-divisible image.
We shall return to this topic of (in)divisibility in Sections~\ref{K2elements} and~\ref{2-BK}.

\smallskip

We now return to~\eqref{BKconj1}, and say more about the terms on its right-hand side.

The precise definition of (the $\ell$-part of) the Bloch-Kato Selmer group,
which is denoted by~$H^1_f(\Q, (V_{\ell}/T_{\ell}))$ and is assumed to be finite, need not concern us until Section~\ref{selmer}.
The subscript $\omega$ denotes a $\Z$-line in $\det (H^1_{\dR}(E_{u,\Q})/\Fil^2)=\det H^1_{\dR}(E_{u,\Q})$, for which we have to choose~$\det_{\Z_p}H^1_{\dR}(\mathcal{E}_{u,\Z})$ by our earlier use of this choice in the definition of $\DDD_{1,2}$ and $R$. The Tamagawa factor $\mathrm{Tam}^0_{p,\omega}(T_{\ell}(2))$ actually does not depend on~$\omega$ unless $p=\ell$, and in that case we shall return to its definition in Section~\ref{2Tam2}. 

When $p\neq\ell$ is a finite prime, we may omit the $\omega$-subscript, and take as a working definition
\begin{equation} \label{Tam-def}
\mathrm{Tam}^0_p(T_{\ell}(2))=\frac{\# H^0(\Q_p, (V_{\ell}/T_{\ell})(2))}{\# H^0(\Q_p, (V_{\ell}^{I_p}/T_{\ell}^{I_p})(2))}
,
\end{equation}
where $I_p\leq \Gal(\Qbar_p/\Q_p)$ is the inertia subgroup, cf.\cite[p.373]{BK}, \cite[\S 11.5]{F}. For a finite prime $p\neq\ell$ of good reduction, $\mathrm{Tam}^0_p(T_{\ell}(2))=1$. Note also that the denominator $\# H^0(\Q_p, (V_{\ell}^{I_p}/T_{\ell}^{I_p})(2))$ is the $\ell$-part of $L_p(E_u, 2)^{-1}$. 

For $p=\infty$, a good working definition is
$$\mathrm{Tam}^0_{\infty}(T_{\ell}(2)):=\#\left(\frac{((V_{\ell}/T_{\ell})(2))^+}{ (V_{\ell}(2)^+/T_{\ell}(2)^+)}\right).$$

Comparing (\ref{BKconj1}) with \cite[(5.15.1)]{BK}, note that if $M=T_{\ell}(2)$ or $V_{\ell}(2)$ then $M^*(1)=T_{\ell}$ or $V_{\ell}$ respectively. Their $L_S(V,0)$ is our $L_S(E_u, 2)$, where $S$ is a finite set of places containing $\infty$ and all primes of bad reduction, and the subscript indicates omission of Euler factors at those primes. The $\ell$-parts of their $\mu_{p,\omega}(A(\QQ_p))$ and our $\mathrm{Tam}^0_{p,\omega}(T_{\ell}(2))L_p(E_u,2)^{-1}$ match. Thus their $\frac{L_S(V,0)}{\prod_{p\in S\setminus\infty}\mu_{p,\omega}(A(\QQ_p))}$ and our $\frac{L(E_u, 2)}{\prod_{p<\infty}\mathrm{Tam}^0_{p,\omega}(T_{\ell}(2))}$ are the same, up to $\ell$-units, as are their $\mu_{\infty,\omega}(A(\R)/A(\Q))$ and our $\frac{R_u \mathrm{Tam}^0_{\infty}(T_{\ell}(2))}{\#H^0(\Q, (V_{\ell}/T_{\ell})(2))}$. Finally, our $\#H^0(\Q, (V_{\ell}/T_{\ell}))$ is the $\ell$-part of their $\#(H^0(\Q, M^*\otimes\Q/\ZZ\,(1)))$, and the $\ell$-part of their $\#(\Sh(M))$ is our $\#H^1_f(\Q, (V_{\ell}/T_{\ell}))$. (Note that $\Sh(M)$ and $\Sh(M^*(1))$ have the same order, by a theorem of Flach \cite{Flach90}.)

We can re-write some of the factors appearing above, in terms of $\ell$-power torsion points on $E_u$. By Poincar\'e duality in $\ell$-adic cohomology, $\Hom_{\Q_{\ell}}(V_{\ell}, \Q_{\ell})=:V_{\ell}^*\simeq V_{\ell}(1)$, with $\Hom_{\Z_{\ell}}(T_{\ell}, \Z_{\ell})=:T_{\ell}^*\simeq T_{\ell}(1)$, respecting the natural Galois actions.
Using the natural isomorphism~$ \het^1(E_{u,\Qbar}, \mu_{\ell^m} ) \simeq E_u(\Qbar)[\ell^m] $
of~\cite[Corollary~III.4.18]{Mil80},
we find that $T_{\ell}\simeq \Tate_{\ell}(E_u)(-1)$ and $V_{\ell}/T_{\ell}\simeq E_u[\ell^{\infty}](-1)$, where $\Tate_{\ell}(E_u):=\varprojlim_n E_u[\ell^n]$ and $E_u[\ell^{\infty}]:=\cup_n E[\ell^n]$. Hence we may re-write the Bloch-Kato conjecture~\eqref{BKconj1} as 
\begin{equation}\label{BKconj}
\ord_{\ell}\left(\frac{L(E_u,2)}{R_{u, \ell}}\right)=\ord_{\ell}\left(\frac{\prod_{p\leq\infty}\mathrm{Tam}^0_{p,\omega}(T_{\ell}(2))\,\# H^1_f(\Q, E_u[\ell^{\infty}](-1))}{\#H^0(\Q, E_u[\ell^{\infty}](1))\,\#H^0(\Q, E_u[\ell^{\infty}](-1))}\right),
\end{equation}
where for finite $p\neq\ell$,
\begin{equation}\label{qTamell}\mathrm{Tam}^0_p(T_{\ell}(2))=\frac{\# H^0(\Q_p, E_u[\ell^{\infty}](1))}{\# H^0(\Q_p, (\VTate{\ell}(E_u)^{I_p}/\Tate_{\ell}(E_u)^{I_p})(1))},\end{equation}
with $\VTate{\ell}(E_u):=\Tate_{\ell}(E_u)\otimes\Q_{\ell}$.

\section{Global torsion and 2-parts of Tamagawa factors away from 2} \label{tam-away}

From now on, we mostly consider the $\ell=2$ part of the Bloch-Kato conjecture.
 
\begin{prop}\label{2tors} Suppose, for $u\in\Q$, that~$ |u^2-1| $ is not a square in $\Q$.
Then $H^0(\Q, E_u[2^{\infty}](1))$ and $H^0(\Q, E_u[2^{\infty}](-1))$
are non-cyclic of order~4.
\end{prop}

\begin{proof}
For any integer~$ r $, we view~$ E_u[2^\infty](r) $ as~$ E_u[2^\infty] $ with the natural
action of~$ \sigma $ in~$ G_\Q $ combined with multiplication by~$ \epsilon_2(\sigma)^r $,
with~$ \epsilon_2 $ the 2-adic cyclotomic character.
Then~$ H^0(\Q, E_u[2](r)) $ is non-cyclic of order~4
because the 2-torsion of~$ E_u $ is rational and~$ \epsilon(\sigma) \equiv 1 $
modulo~2 always.

We now show there are no more elements in~$ H^0(\Q, E_u[4](r)) $ for~$ r = \pm 1 $,
which implies the proposition.
For this, we consider the points of order~4, as given in Proposition~\ref{torsion}(4).
Those in  the first coset there are not invariant because~$ G_\Q $ acts on
them non-trivially by multiplication by~$ \pm1 $ due to the twist.
Those in the second and third cosets are not invariant
(in fact, for any twist~$ r $) because $ |u^2-1| $ is not a square,
so their first coordinates are not invariant, and~$ \epsilon_2(\sigma) $
acts as multiplication by~$ \pm 1 $ on the points, which does
not change the first coordinates.
\end{proof}

\begin{prop}\label{tams} Suppose that $u$ is an integer, with $|u|>1$,
and~$ p $ is an odd prime number.
\begin{enumerate} 
\item If $\ord_p(u)$ is odd, then $$\mathrm{Tam}^0_p(T_2(2))=\begin{cases} 4 \text{ \,\,\,\,if $p\equiv 1\pmod{4}$};\\ 2 \text{\,\,\,\, if $p\equiv 3\pmod{4}$}.\end{cases}$$
\item If $\ord_p(u^2-1)$ is odd, then $\mathrm{Tam}^0_p(T_2(2))=2$.
\item If~$ \ord_p(u(u^2-1)) > 0 $ then those are minimal values for~$ \mathrm{Tam}^0_p(T_2(2)) $.
\end{enumerate}
\end{prop}

\begin{proof} 
Recalling from Proposition~\ref{minimal} that $E_u$ has multiplicative reduction at any
odd prime~$p$ dividing~$ u(u^2-1) $, we have that~$\VTate{2}(E_u)^{I_p}$ is a $1$-dimensional $\Q_2$-vector space,
and~$\VTate{2}(E_u)^{I_p}/\Tate_2(E_u)^{I_p}$ is naturally identified with the group~$\mu_{2^{\infty}}$ of $2$-power roots of unity under the Tate isomorphism $E_u(\Qbar_p)\simeq \Qbar_p^{\times}/Q^{\Z}$, where~$Q$ in~$\Q_p^\times$ with~$ |Q|_p < 1 $ is the Tate parameter.
If~$ p \mid (u^2-1) $ and~$ p \equiv 3 $ modulo~4, then the reduction
is non-split, and
the natural action of~$ G_{\Q_p} $ on the right-hand side is twisted by the
non-trivial quadratic character of~$ \Gal(\Q_p(i)/\Q_p) $
(see the proof of \cite[Corollary~V.5.4]{S2}).
This character is trivial on~$ I_p $ because $\Q_p(i)/\Q_p$ is unramified.

(1)
From~$\Delta=Q\prod_{n=1}^{\infty}(1-Q^n)^{24}$,
$\Delta=\left(\frac{u}{4}\right)^4(u^2-1)^2$,
and~$\ord_p(u)$ being odd, we find~$2^2\parallel\ord_p(Q)=\ord_p(\Delta)$.
Then~$ 1 - u^2 $ is a square in~$ \Z_p $ because it is in~$ 1 + p \Z_p $, so~$ Q $ admits a fourth root~$ Q^{1/4} $ in~$ \Q_p $
and~$E_u[2^{\infty}]^{I_p}\simeq \frac{\langle Q^{1/4}\rangle}{Q^{\Z}}\oplus \mu_{2^{\infty}}$.
On the second summand, $\Frob_p$ acts as $p$, but in the twist $E_u[2^{\infty}](1)$ as $p^2$.
The order of the subgroup of invariants gives
the denominator in~\eqref{qTamell}, and the fraction itself equals the
order of the subgroup of invariants in~$ \frac{ \langle Q^{1/4} \rangle }{ Q^{\Z} } (1) $.
Because~$ \Frob_p $ acts trivially on~$ \sqrt{-1} $
if~$ p \equiv 1 $ modulo~4,  and by multiplication by~$ -1 $
if~$ p \equiv 3 $ modulo~4, the order of the invariants
is~4 in the former case and~2 in the latter.

(2)
Arguing as in~(1), we see that~$ Q $ has a square root~$ Q^{1/2} $
in~$ \Q_p $, and that~\eqref{qTamell} is equal to the order of
the subgroup of~$ G_{\Q_p} $-invariants in~$ \langle Q^{1/2} \rangle / Q^{\Z} (1) $,
which is~2.

(3) 
The denominator in~\eqref{qTamell} is still given by the order
of the subgroup of~$ G_{\Q_p} $-invariants of~$ \mu_{2^\infty} (1) $.
In~(1), $ Q $ again admits a fourth root~$ Q^{1/4} $ in~$ \Q_p $, 
and the subgroup of~$ G_{\Q_p} $-invariants of~$ \frac{ \langle Q^{1/4} \rangle }{ Q^{\Z} } (1) $
contributes to the resulting fraction in~\eqref{qTamell}.
Similar remarks apply to~(2).
\end{proof}

\begin{prop}\label{taminf} For any $u\in \Q-\{0,\pm 1\}$, $\Tam_{\infty}^0(T_2(2))=2$.
\end{prop}
\begin{proof} Recall that $$\mathrm{Tam}^0_{\infty}(T_2(2)):=\#\left(\frac{((V_2/T_2)(2))^+}{ (V_2(2)^+/T_2(2)^+)}\right)=\# \left( \frac{E_u[2^\infty](1)^+}{\VTate{2}(1)^+ / \Tate_2(1)^+} \right).$$ A Tate twist by $r\in\Z$ multiplies the action of the generating complex conjugation in $\Gal(\C/\R)$ by $(-1)^r$. Since all the $2$-torsion of $E_u$ is defined over $\R$, the period lattice of $E_u(\C)$ has rectangular fundamental region, and $H_1(E_u(\C), \Z)=\Z\gamma^+\oplus\Z\gamma^-$, where $\gamma^{\pm}$ are eigenvectors for complex conjugation. Then $T_2(2)\simeq (\Z_2\gamma^+\oplus\Z_2\gamma^-)(1)$ and $(V_2(2)^+/T_2(2)^+)\simeq (\Q_2/\Z_2)\gamma^-(1)$. The quotient $\frac{((V_2/T_2)(2))^+}{ (V_2(2)^+/T_2(2)^+)}$ is generated by $\frac{1}{2}\gamma^+(1)$ (representing the $2$-torsion point $(0,0)$).
\end{proof}

\section{The 2-part of the Tamagawa factor at 2}\label{2Tam2}

Recall from Proposition~\ref{minimal}(3) that a global minimal Weierstrass equation for~$E_u$~is
$${y'}^2+x'y'={x'}^3+\frac{u^2}{4}{x'}^2+\frac{u^2}{16}x',$$ where we are assuming (for an integer $u$ with $|u|>1$) that $4\parallel u$, and we have good, ordinary reduction at $2$. We think of this as an equation for an integral model $\mathcal{E}_{u,\Z_2}$, with generic fibre $E_{u,\Q_2}$ and special, regular, fibre $\mathcal{E}_{u,\FF_2}$.
With notation as at the beginning of Section~\ref{regBeil} we
then have~$ 1 + 2 - a_2 = \# \mathcal{E}_u(\FF_2)=4 $, so that~$ a_2 = -1 $,
and~$ L_2(E_u, s) = (1+ 2^{-s}+2^{1-2s})^{-1} $.

In this section we let $V:=V_2=H^1_{\text{\'et}}(E_{u,\Qbar}, \Q_{2})$, and think of $V_2(2)$ as the twisted dual $V^*(1)$. (Recall from Section~4 that $V_{\ell}^*\simeq V_{\ell}(1)$ by Poincar\'e duality. This is also connected with the fact that the functional equation of $L(E_u, s)$ relates $s=0$ and $s=2$.)

We view $V$ purely as a $2$-adic representation of $G_{\Q_2}=\Gal(\Qbar_2/\Q_2)$. Let
\begin{equation} \label{Dcris-def}
\mathbf{D}(V):=D_{\cris}(V):=(B_{\cris}\otimes V)^{G_{\Q_2}}
,
\end{equation}
where $B_{\cris}$ is Fontaine's ring. We know that
$$\mathbf{D}(V)\simeq H^1_{\cris}(\mathcal{E}_{u,\FF_2})\otimes_{\Z_2}\Q_2\simeq H^1_{\dR}(E_{u,\Q_2}).$$
The first isomorphism is an instance of the Fontaine-Messing $p$-adic comparison theorem \cite[Theorem B]{FM}, the second an instance of \cite[3.4.2]{I}. Their composite is an isomorphism of filtered $\phi$-modules over $\Q_2$, where $\phi$ is a Frobenius operator.  The characteristic polynomial of $\phi$ on $\mathbf{D}(V)$ is therefore $X^2-a_2X+2=X^2+X+2$.

Since $\dim_{\Q_2}(\mathbf{D}(V))=\dim_{\Q_2}(V)$, $V$ is a crystalline representation of $G_{\Q_2}$. It is further an ordinary representation of $G_{\Q_2}$, an invariant line in $V=\Tate_2(E_u)(-1)\otimes_{\Z_2}\Q_2$ being obtained from the kernel of reduction from $\Tate_2(E_u)$ to the rank~1 $\Tate_2(\mathcal{E}_{u,\FF_2})$. 
That this kernel is a~$ \Z_2 $-module of rank~1 is a direct consequence
of \cite[Chapter~4, \S2, Theorem~2(ii)]{Fr}.

Following \cite[Appendix C]{PR}, let $t_V:=\mathbf{D}(V)/\Fil^0\mathbf{D}(V)=\{0\}$.
(When making definitions below, we shall temporarily suppress our knowledge of this.) Then by the short exact sequence in \cite[C.2.3]{PR} we have $(t_{V^*(1)})^*\simeq \mathbf{D}(V)$, hence $t_{V^*(1)}\simeq \mathbf{D}(V)$ by Poincar\'e duality \cite[3.5.2]{I}. Define
$$H^1_f(\Q_2, V):=\ker(H^1(\Q_2, V)\rightarrow H^1(\Q_2, V\otimes B_{\cris})).$$

Since $L_2(E_u, 0)^{-1}\neq 0$, the Bloch-Kato exponential map $\exp: t_V\rightarrow H^1_f(\Q_2, V)$ is an isomorphism, as in \cite[11.5]{F}. Choosing $\omega_1\in{\det}_{\Q_2}t_V$
induces a measure $\mu_{\omega_1}$ on $H^1_f(\Q_2, V)$, such that if $\Lambda$ is a $\Z_2$-lattice in the $\Q_2$-vector space $H^1_f(\Q_2, V)$ with ${\det}_{\Z_2}(\Lambda)=\Z_2((\det\exp)(\omega_1))$ then $\mu_{\omega_1}(\Lambda)=1$. It also gives a measure on $H^1_f(\Q_2, \mathbf{T})$, defined to be the pre-image of $H^1_f(\Q_2, V)$ under the natural map from $H^1(\Q_2, \mathbf{T})$ to $H^1(\Q_2, V)$. Here, $\mathbf{T}$ is a $G_{\Q_2}$-invariant $\Z_2$-lattice in $V$, which we take to be $\mathbf{T}=T_{2}:=H^1_{\text{\'et}}(E_{u,\Qbar}, \Z_{2})$. According to \cite[11.1]{F}, $H^1_f(\Q_2, \mathbf{T})$ is a finitely-generated $\Z_2$-module. Defining $\Tam_{2,\omega_1}(\mathbf{T})=\mu_{\omega_1}(H^1_f(\Q_2, \mathbf{T}))$, we may then define $\Tam^0_{2,\omega_1}(\mathbf{T})$ by the relation $$\Tam_{2,\omega_1}(\mathbf{T})=\Tam^0_{2,\omega_1}(\mathbf{T})|L_2(E_u, 0)|_2$$
from \cite[11.5]{F}.

Since $L_2(E_u, 2)^{-1}\neq 0$, we may similarly define $\Tam^0_{2,\omega_2}(\mathbf{T}^*(1))$, where $\omega_2\in {\det}_{\Q_2}t_{V^*(1)}$, noting that $\mathbf{T}^*(1)=T_2(2)$ inside $V^*(1)=V_2(2)$. 
Recalling that $t_V=\{0\}$ and $t_{V^*(1)}\simeq \mathbf{D}(V)\simeq H^1_{\dR}(E_{u,\Q_2})$, we choose $\omega_1=1$ and $\Z_2\omega_2=\det_{\Z_2}H^1_{\dR}(\mathcal{E}_{u,\Z_2})$.

\begin{prop}\label{2tam2left}
$\Tam_{2,\omega_1}^0(T_2)=2$. (Recall that throughout this section, $u$ is an integer with $|u|>1$ and $4\mid\mid u$.)
\end{prop}

\begin{proof} First we look at the composition factors for the action of $G_{\Q_2}$ on $V$. On $\Tate_2(\mathcal{E}_{u,\FF_2})$, $G_{\Q_2}$ acts via an unramified character $\lambda_{c}: \Frob_p\mapsto c$, the unit root of Frobenius.
Because~$ c^2 + c + 2 = 0 $, it is in~$ 5 + 8 \Z_2 $.
On the kernel of reduction from $\Tate_2(E_u)$ to $\Tate_2(\mathcal{E}_{u,\FF_2})$, it must act via $\epsilon_2\lambda_{c}^{-1}$, so that on the determinant of $\Tate_2(E_u)$ it is $\epsilon_2$, the $2$-adic cyclotomic character. Taking the $(-1)$-twist, it follows that the composition factors for $V$ are $\epsilon_2^{-1}\lambda_{c}$ and $\lambda_{c}^{-1}$, hence that $H^0(G_{\Q_2}, V)=\{0\}$. Since $t_V=\{0\}$, $H^1_f(\Q_2, V)$ is trivial (by the Bloch-Kato exponential isomorphism), so $$H^1_f(\Q_2, \mathbf{T})=H^1(\Q_2, \mathbf{T})_{\tor}.$$ Then using the long exact sequence associated to $ 0 \to \mathbf{T} \to V \to V/\mathbf{T} \to 0 $, and~$H^0(G_{\Q_2}, V)=\{0\}$, we get 
$$H^1(\Q_2, \mathbf{T})_{\tor}\simeq H^0(\Q_2, V/\mathbf{T})=H^0(\Q_2, E[2^{\infty}](-1)),$$ and $$\Tam_{2,\omega_1}(\mathbf{T})=\mu_{\omega_1}(H^1_f(\Q_2, \mathbf{T}))=\#H^1(\Q_2, \mathbf{T})_{\tor}=\#H^0(\Q_2, E[2^{\infty}](-1)).$$

That~$ c $ is in~$ 5 + 8 \Z_2 $ shows~$\lambda_{c}^{-1}$ is trivial modulo~4 but not modulo~8, and~$\epsilon_2^{-1}\lambda_{c}$ is trivial modulo~2 but not modulo~4,
which implies $\#H^0(\Q_2, E[2^{\infty}](-1)) \le 8$.
The subgroup~$E[2](-1)$ of~$ E[2^\infty](-1) $ is in~$ H^0(\Q_2, E[2^{\infty}](-1)) $
by part~(1) of Proposition~\ref{torsion}.
From part~(4) we know that~$ R = (-1 + \sqrt{1-u^2}, i(u^2-1 + \sqrt{1-u^2}) ) $
has order~4. Because~$ 1 - u^2 $ is congruent to~1 modulo~8, it is a square in~$ \Z_2 $,
and~$ R $ is defined over~$ \Q_2(i) $. So the action 
of~$ G_{\Q_2} $ on~$ \langle R \rangle (-1) $ factors through~$ \Gal(\Q_2(i)/\Q_2) $.
It is easily seen to be trivial, and we conclude that~$ H^0(G_{\Q_2}, E[2^\infty])$ has
order~8.

Now~$\Tam_{2,\omega_1}(\mathbf{T})=8$ and~$L_2(E_u, 0)^{-1}=4$, so that~$\Tam^0_{2,\omega_1}(\mathbf{T})=\frac{8}{4}=2$.
\end{proof}

The original conjecture of Bloch and Kato \cite{BK} applies to $L$-values such as our $L(E_u, 2)$, where the $L$-function is evaluated at an integer point in the domain of convergence of the Dirichlet series. Using the functional equation, this would also predict the leading term at partner points on the left, $s=0$ in our case. Fontaine and Perrin-Riou \cite{FP} formulated a uniform conjecture that applies to all integer points, independent of each other, equivalent to that of Bloch and Kato in the domain of convergence. In \cite[Appendix C]{PR}, they address the question of whether the predictions of their conjecture, applied to partner points, are compatible with the functional equation. In \cite[C.3.3]{PR}, they show that this would follow from a conjecture of Deligne on determinants of motives (known in our case) and (for all primes) a new conjecture $C_{EP}(V)$, stated in \cite[C.2.9]{PR}, on the ratio $\frac{\Tam^0_{p, \omega_1}(\mathbf{T})}{\Tam^0_{p,\omega_2}(\mathbf{T}^*(1))}$.
As explained in \cite[C.2.10]{PR}, $C_{EP}(V)$ was already known in the case of ordinary representations (in which we find ourselves), thanks to work of Perrin-Riou \cite{PR2}. Having just found $\Tam^0_{2,\omega_1}(T_2)$, we can now deduce the value of $\Tam^0_{2,\omega_2}(T_2(2))$, using $C_{EP}(V)$.
\begin{prop}\label{2tam2}
$\Tam_{2,\omega_2}^0(T_2(2))=2$. (Recall that throughout this section, $u$ is an integer with $|u|>1$ and $4\mid\mid u$.)
\end{prop}
Note that in \S \ref{BlKa} we called this $\Tam_{2,\omega}^0(T_2(2))$. Our use here of the notation $\omega_1$ and $\omega_2$ is in harmony with \cite{PR}.
\begin{proof} It suffices to show that $\frac{\Tam^0_{p, \omega_1}(\mathbf{T})}{\Tam^0_{p,\omega_2}(\mathbf{T}^*(1))}=1$. Substituting the proposition in \cite[C.2.6]{PR} into $C_{EP}(V)$ in \cite[C.2.9]{PR} produces 
$$\Z_2\,\frac{\Tam^0_{p, \omega_1}(\mathbf{T})}{\Tam^0_{p,\omega_2}(\mathbf{T}^*(1))}=\Z_2\left(\prod_j \Gamma^*(-j)^{-\dim\mathrm{gr}^j\mathbf{D}(V)}\right)\eta_V(\omega),$$
where $\omega:=\omega_2\otimes\omega_{\mathbf{T}}$, with $\Z_2\omega_{\mathbf{T}}:={\det}_{\Z_2}(\mathbf{T})$, and $\eta_V(\omega)\in 2^{\Z}$ is as in \cite[C.2.8]{PR} (with more detail below). 

Also $$\Gamma^*(-j):=\begin{cases} (j-1)! & \text{ if } j>0;\\(-1)^j((-j)!)^{-1} & \text{ if } j\leq 0.\end{cases}$$
For us, $\prod_j \Gamma^*(-j)^{-\dim\mathrm{gr}^j\mathbf{D}(V)}=1^{-1}1^{-1}=1$ (a product over $j=0,1$, using $\mathbf{D}(V)\simeq H^1_{\dR}(E_{u,\Q_2})$), so it suffices to prove that $\eta_V(\omega)=1$.

Let $t$ be a generator of $\Z_2(1)$ inside $B_{\cris}$. Let $\mathbf{C}$ be the completion of an algebraic closure of $\Q_2$, and note that $\Fil^iB_{\cris}/\Fil^{i+1}B_{\cris}\simeq \mathbf{C}(i)$, for any $i\in\Z$. The Fontaine-Messing comparison isomorphism
$$B_{\cris}\otimes(H^*_{\text{\'et}}(E_{u,\overline{\Q}_2}, \Q_2))\simeq B_{\cris}\otimes(H^*_{\cris}(\mathcal{E}_{u,\FF_2})\otimes_{\Z_2}\Q_2)$$
(part of \cite[Theorem B]{FM}, given that $E_u$ is ``admissible'' by \cite[Theorem A]{FM}) is compatible with products and filtrations. Identifying $H^1_{\cris}(\mathcal{E}_{u,\FF_2})$ with $H^1_{\dR}(E_{u,\Q_2})$, choose a basis $\{w_0, w_1\}$ with $w_i\in\Fil^iH^1_{\dR}(E_{u,\Q_2})$. In the image of any element of $H^*_{\text{\'et}}(E_{u,\overline{\Q}_2}, \Q_2)$, the coefficient of $w_1$ must lie in $\Fil^{-1}B_{\cris}$. Thus by restricting the comparison map to $H^*_{\text{\'et}}(E_{u,\overline{\Q}_2}, \Q_2)$, and in the determinant of $H^1$ projecting coefficients to $\Fil^{-1}B_{\cris}/\Fil^0B_{\cris}\simeq \mathbf{C}(-1)$, we get a commutative diagram
$$\begin{CD}\det_{\Q_2}(V) @>>> \det_{\Q_2}\mathbf{D}(V)\otimes\mathbf{C}(-1) \\ @VVV @VVV\\\Q_2(-1) @>>>\mathbf{C}(-1).\end{CD}$$
The vertical maps are cup products composed with trace maps. 
The top row, tensored with $ \det_{\Q_2}\mathbf{D}(V)^*$, factors through the $\Q_2$-linear injection $$\xi_V: \det_{\Q_2}\mathbf{D}(V)^*\otimes \det_{\Q_2}(V)\rightarrow \Qbar_2 t^{-1}$$ of \cite[C.2.7]{PR},
and following loc. cit. we set~$\tilde{\xi}_V:=t \xi_V$. (For us, $t_H(V)=-1$.)
Down the right-hand side, $\Z_2^{\times}\omega_2^{-1}t^{-1}\mapsto \Z_2^{\times}t^{-1}$, by Poincar\'e duality in crystalline cohomology with $\Z_2$-coefficients \cite[3.5.4]{I}, while down the left-hand
side $\omega_{\mathbf{T}}\mapsto t^{-1}\Z_2$, by Poincar\'e duality in \'etale cohomology with $\Z_2$-coefficients.  Along the bottom row, $\Z_2^{\times}t^{-1}\mapsto\Z_2^{\times}t^{-1}$, since by \cite[III, 6.3]{FM} the \'etale cycle class of a point maps to the crystalline cycle class of a point. It follows that along the top, $\omega_{\mathbf{T}}\mapsto \Z_2^{\times}\omega_2^{-1}t^{-1}$.

It follows immediately from this (reading it as $\omega=\omega_2\otimes\omega_{\mathbf{T}}\mapsto t^{-1}\Z_2^{\times}$) that $\tilde{\xi}_V(\omega)$ in \cite[C.2.7]{PR} is a unit. In the notation of \cite[C.2.8]{PR},
$\eta_V(\omega)$ is the $2$-adic absolute value of $\frac{\tilde{\xi}_V(\omega)}{|d_K|^{\dim(V)/2}\epsilon(V,\psi_{0,K},\mu_{0,K})},$ but for us $K=\Q_2$ has trivial discriminant, and the $\epsilon$-factor is a unit because $2$, as a prime of good reduction, does not divide the conductor. Hence $\eta_V(\omega)=1$, as required.
\end{proof}
\begin{remar}\label{BK4.1iii} The closest together $i\leq 0$ and $j\geq 1$ such that $\Fil^i\mathbf{D}(V(2))=\mathbf{D}(V(2))$ and $\Fil^j\mathbf{D}(V(2))=\{0\}$ are $i=-2, j=1$. With $p=2$, they fail to satisfy the condition $(*): j-i<p$ from \cite[Theorem 4.1(iii)]{BK}, which therefore fails to show that $\Tam^0_{2,\omega_2}(T_2(2))=1$.
\end{remar}

\section{The 2-part of the Bloch-Kato Selmer group}\label{selmer}

In this section $V:=V_{\ell}=H^1_{\text{\'et}}(E_{u,\Qbar}, \Q_{\ell})$ and $T=T_{\ell}:=H^1_{\text{\'et}}(E_{u,\Qbar}, \Z_{\ell})$. 
We turn our attention now to the task of constructing elements in the group $H^1_f(\Q, V/T)=H^1_f(\Q, E_u[\ell^{\infty}](-1))$
when $\ell=2$, as its order appears in the numerator of the right-hand
side of (\ref{BKconj}).
In fact, we shall describe its 2-torsion completely in Corollary~\ref{2torscor}.

The definition of~$ H_f^1(\Q, V/T) $, for general~$ \ell $, is as follows (see \cite[(3.7)]{BK}).
First we define (including for $p=\infty$)
$$H^1_f(\Q_p, V):=\begin{cases} \ker(H^1(\Q_p, V)\rightarrow H^1(I_p, V)) & p\neq\ell;\\
\ker(H^1(\Q_p, V)\rightarrow H^1(\Q_p, V\otimes B_{\cris}) & p=\ell,\end{cases}$$
where $I_p \subseteq G_{\Q_p}$ is the inertia subgroup.
Then we let~$H^1_f(\Q_p, V/T)$ be the image of~$H^1_f(\Q_p, V)$ under the natural map from $H^1(\Q_p, V)$ to $H^1(\Q_p, V/T)$.
We also set
$$H^1_f(\Q, V/T):=\cap_p \, \res_p^{-1} (H^1_f(\Q_p, V/T)),$$ where~$ \res_p : H^1(\Q, V/T) \to H^1(\Q_p, V/T) $
is the restriction map.

It follows from the finite generation of $H^1_f(\Q, T)$ and
the finiteness of~$ \Sh $, as discussed in~11.1 and~11.2 of~\cite{F},
that~$ H_f^1(\Q, V/T) $ is isomorphic to a finite direct sum of
copies of~$ \Q_{\ell} / \Z_{\ell} $ and cyclic groups of~$ \ell $-power
order.

For~$ p \ne \ell , \infty$ the conditions above can be interpreted using the inflation and restriction
maps in the commutative diagram
\begin{equation}
\begin{split} \label{infCD}
\xymatrix{
0 \ar[r] & H^1_f(\Q_p, V) \ar[r]\ar[d] &  H^1(\Q_p ,  V) \ar[r]\ar[d] & H^1(I_p ,  V) \ar[d]
\\
0 \ar[r] & H^1(\F_p , (V/T)^{I_p} ) \ar[r] &  H^1(\Q_p ,  V/T) \ar[r] & H^1(I_p ,  V/T) 
}
\end{split}
\end{equation}
with exact rows, where we used that~$ G_{\Q_p} /I_p \simeq G_{\F_p} $,
and identified~$H^1_f(\Q_p, V)$ with~$ H^1(\F_p ,  V^{I_p} ) $.
The leftmost vertical map is induced
by the composition of the surjection~$ V^{I_p} \to V^{I_p} / T^{I_p} $
and the injection~$ V^{I_p} / T^{I_p} \to (V / T)^{I_p} $.
Therefore, for an element of~$ H^1(\Q_p, V/T) $ to be in~$ H_f^1(\Q_p, V/T) $
it has to become trivial in~$ H^1(I_p, V/T) $, so that it comes
from an element in~$ H^1(\F_p , (V/T)^{I_p} ) $, and this second
element has to become trivial when replacing the coefficients
with the cokernel of the injection~$ V^{I_p} / T^{I_p} \to (V/T)^{I_p} $.
We shall refer to these as the \emph{row} and \emph{column} condition,
respectively. Combining the corollary on~\cite[p.261]{Ta} with~II.3.3a in \cite{Se}
shows that~$ G_{\F_p} $ has cohomological dimension~1, 
so that~$ H^2(\F_p, T^{I_p}) = 0 $, hence~$H^1(\F_p, V^{I_{p}})$ surjects onto~$H^1(\F_p, V^{I_p} / T^{I_p} )$.
Therefore the two conditions we formulated are necessary and sufficient.
We shall use these in the proof of Theorem~\ref{f-conditions},
so that we traverse the diagram~\eqref{infCD} in a different way
compared to that in the definition of~$ H_f^1(\Q, V/T) $.

We note in passing that the Tate twist is irrelevant for the action of~$ I_p $ if~$ p \ne \ell , \infty $.
From the discussion on \cite[p.382]{S2}, which applies equally
well to~$ \Tate_{\ell} $ as to~$ \VTate{\ell} $,
we then see that~$ V^{I_p} / T^{I_p} $
is the subgroup of~$ (V/T)^{I_p} = ( E_u[\ell^\infty](-1))^{I_p} $ consisting
of points that hit the 0-component of the regular minimal
model of~$ E_u $ at~$ p $.

\smallskip

We now take~$ \ell = 2 $ for the remainder of this section.

\begin{lem} \label{2kernel}
For~$u \ne 0, \pm1$ in $\Q$, the 2-torsion in~$  H^1(\Q , E_u[2^\infty](-1)) $ is equal to the
image of the map~$ H^1(\Q , E_u[2](-1)) \to H^1(\Q , E_u[2^\infty](-1)) $
obtained by extension of the coefficients.
If~$ |u^2-1| $ is not a square in $\Q$ then 
this image is the quotient of~$ H^1(\Q , E_u[2](-1)) $
by a subgroup of order~4,
which is obtained from~$ H^1(\Gal(\Q(\sqrt{u^2-1}, i) / \Q) , \langle (0,0) \rangle (-1) ) $
by inflation and extension of the coefficients.
\end{lem}

\begin{proof}
From the long exact sequence in $G_{\Q}$-cohomology associated
to the short exact sequence
\begin{equation} \label{sos2mult}
0\rightarrow E_u[2](-1)\rightarrow E_u[2^{\infty}](-1)\xrightarrow{2} E_u[2^{\infty}](-1)\rightarrow  0
,
\end{equation}
we get the description of the 2-torsion as the stated image, and see
that the kernel of the natural map~$ H^1(\Q , E_u[2](-1)) \to H^1(\Q , E_u[2^\infty](-1)) $
is the image of~$ H^0(\Q,  E_u[2^{\infty}](-1)) $ under the connecting homomorphism.
If~$ |u^2-1| $ is not a square, then using Proposition~\ref{2tors}, we have $ H^0(\Q,  E_u[2^{\infty}](-1))=E_u[2](-1)$, and obtain the following description of its image. Not writing the Tate twist at the level of elements,
we have~$ 2 Q = (0, 0) $ in~$ E_u[2](-1) $ for~$ Q = (u, u(u+1)) $, and
the action of~$ G_\Q $ on~$ Q $ factors through~$ \Gal(\Q(i)/\Q) = \langle \sigma \rangle $.
The image of~$ (0,0) $ under the connecting homomorphism maps~$ \sigma $
to~$ \sigma(Q) - Q = (0,0) $ because~$ \sigma(Q) = - Q $ due
to the Tate twist.
With~$ R = (-u^2 + u \sqrt{u^2-1} , i ( u (u^2 - 1) - u^2 \sqrt{u^2-1})) $,
we have~$ 2 R = (-u^2,0) $ in~$ E_u[2](-1) $. The action of~$ G_\Q $
on~$ R $ factorises through~$ \Gal(\Q(\sqrt{u^2-1},i) / \Q) $
but, in fact, also through~$ \Gal(\Q(\sqrt{u^2-1}) / \Q) = \langle \tau \rangle $ due to the
Tate twist. Then~$ \tau(R) = R + (0,0) $
because~$ (0,0) $,~$ R $, and $  - \tau(R) = (-u^2 - u \sqrt{u^2-1} , - i ( u (u^2 - 1) + u^2 \sqrt{u^2-1})) $
as points of~$ E_u[4] $ all lie on the line~$y=i\sqrt{u^2-1} \, x$.
It follows that the image of~$ (-u^2,0) $ under the connecting homomorphism maps~$ \tau $ to~$ (0,0) $.
\end{proof}

In order to describe the 2-torsion of~$ H_f^1(\Q, E_u[2^\infty](-1) ) $
in Corollary~\ref{2torscor} below, we employ an explicit
description of~$ H^1(\Q, E_u[2](-1) ) $, and use this to determine,
for every~$ p $, when the resulting
elements in~$ H^1(\Q, E_u[2^\infty](-1) ) $ under~$ \res_p $
map to~$ H_f^1(\Q_p, E_u[2^\infty](-1) ) $. 
We do this under the assumption that~$ u $ is an integer
congruent to~4 modulo~8, so that we have the result of Proposition~\ref{minimal};
in particular, the curve~$ E_u $ has good, ordinary reduction
at the prime~2.
In order to avoid more complicated statements in parts~(3) and~(4)
we further assume that every prime number that occurs in the factorisation
of~$ u(u^2-1)/4 $ has odd exponent. Then the condition
of Lemma~\ref{2kernel} is also satisfied.

For~$ D $ in~$ \Q^\times $, Kummer theory gives a homomorphism~$ G_\Q \to \{\pm1\} $.
Identifying the latter group with~$ \langle (0,0) \rangle (-1) $
in the only possible way, we obtain an element~$ g = g_D $ in~$ H^1(\Q, \langle (0,0) \rangle (-1) ) $.
Similarly, for~$ D' $ in~$ \Q^\times $, using an identification
with~$ \langle (-1,0) \rangle (-1) $ we obtain an element~$ h = h_{D'} $ in~$ H^1(\Q, \langle (-1,0) \rangle (-1) ) $.
With~$ E_u[2] = \langle (0,0) \rangle \oplus \langle (-1,0) \rangle $,
we then have~$ g+h $ in~$ H^1(\Q, E_u[2] (-1) ) $.
This way we identify~$ H^1(\Q, E_u[2] (-1) ) $ with~$ \Q^\times/2 \times \Q^\times/2 $.

\begin{thm} \label{f-conditions}
Suppose that~$ u $ is an integer that is
congruent to~4 modulo~8, and that
every prime number in the factorisation
of~$ u(u^2-1)/4 $ has odd exponent.
For~$ D $ and~$ D' $ in~$ \Q^\times $, let~$ g + h $ in~$ H^1(\Q, E_u[2](-1)) $
be the element defined above, and 
let~$ [g+h] $ denote its image in~$ H^1(\Q, E_u[2^\infty](-1) ) $
under extension of the coefficients.
Then~$ [g+h] $ is in~$ \res_p^{-1}(H_f^1(\Q_p, E_u[2^\infty](-1))) $ if and
only if the following hold.
\begin{enumerate}
\item
$ D' > 0 $ for~$ p = \infty $.

\item
$ \ord_p(D) $ and $ \ord_p(D') $ are even, 
if~$ p $ does not divide~$ u (u^2-1) $.

\item
$ D' $ is a square in~$ \Q_p $,
if~$ p $ divides~$ u^2-1 $.

\item
If~$ p $ divides~$ u/4 $ then~$ \ord_p(D) $ is even,~$ \ord_p(D') $ is
even
if~$ p \equiv 3 $ modulo~4, and~$ 2^{ord_p(D')} D $ is a square in~$ \Q_p $ if~$ p \equiv 1 $
modulo~4.

\item
$ D $ or~$ -D $ is a square in~$ \Q_2 $, as well as~$ D' $
 or~$ -3 D' $, if~$  p = 2 $.
\end{enumerate}
\end{thm}

\begin{proof}
As a general strategy in this proof, we shall often use that~$ \res_p([g+h] )$ in~$ H^1(\Q_p, E_u[2^\infty](-1) ) $
comes from the image of~$ g+h $ in~$ H^1(\Q_p, E_u[2](-1) ) $, so that
we can compute when~$ \res_p([g+h]) $ is trivial by comparing
that image with the image of~$ H^0(\Q_p, E_u[2](-1) ) / 2 $ under the
connecting homomorphism in the long exact sequence of~$ G_{\Q_p} $-cohomology
associated to~\eqref{sos2mult}.

(1) With all points in~$ E_u[2] $ rational,
the Weierstrass parametrisation of~$ E_u $ shows that its torsion
under the action of~$  I_\infty = G_{\R} = \Gal(\C/\R) = \langle \sigma \rangle $
is isomorphic to~$ (\Q/\Z)^+ \oplus (\Q/\Z)^- $, where~$ \sigma $
acts trivially on the first copy of~$ \Q/\Z $ and by multiplication by~$ -1 $ on
the second (cf.\ the proof of Proposition~\ref{taminf}). The first copy contains the point~$P = (u, u(u+1))$,
hence also $2P=(0,0)$. Then~$ (E_u[2^\infty](-1))^{G_\R} = \langle 2 P \rangle(-1) \oplus (\Q_2/\Z_2)^-(-1)$.
From the long exact sequence of~$ G_{\R} $-cohomology associated
to~\eqref{sos2mult} we see that the kernel of~$ H^1(\R, E_u[2](-1)) \to H^1(\R, E_u[2^\infty](-1)) $
is generated by the element that maps~$ \sigma $
to~$ \epsilon_2(\sigma) \sigma(P) - P = -2 P = (0,0) $ in~$ E_u[2](-1) $,
where~$ \epsilon_2 $ is again the 2-adic cyclotomic character.
Therefore~$ \res_\infty([g]) $ is always trivial in~$ H^1(\R, E_u[2^\infty](-1)) $,
but for~$ \res_\infty([h]) $ to be trivial we need that~$ h $ gives the trivial
element in~$ H^1(\R, \langle (-1,0) \rangle (-1) ) $, i.e.,
that~$ D' > 0 $. 
If that is the case, then~$ \res_\infty([g+h]) $ is trivial,
hence is in~$ H_f^1(\R, E_u[2^\infty] (-1) ) $.
So~$ D' > 0 $ is necessary and sufficient for~$ \res_\infty([g+h]) $
to be in~$ H_f^1(\R, E_u[2^\infty](-1) ) $.

(2)
Here~$ p $ is an odd prime number where~$ E_u $ has good reduction.
Then~$ I_p $ acts trivially on~$ E_u[2^\infty](-1) $, so that
from the long exact sequence of~$ I_p $-cohomology 
associated to~\eqref{sos2mult},
we see that~$ H^1(I_p, E_u[2](-1)) $ injects into~$ H^1(I_p, E_u[2^\infty](-1) ) $.
Therefore~$ \res_p([g+h]) $ restricts to the trivial element
in~$ H^1(I_p, E_u[2^\infty](-1) ) $ if and only
if both~$ g $ and~$ h $ restrict to the trivial homomorphism
on~$ I_p $, i.e., $ \ord_p(D) $ and~$ \ord_p(D') $ are even.
With the row condition satisfied, the column condition is trivial
because~$ I_p $ acts trivially on~$ V $ and~$ T $.
So the condition that we found is necessary and sufficient for~$ \res_p([g+h]) $
to be in~$ H_f^1(\Q_p, E_u[2^\infty](-1) ) $.

(3) By Proposition~\ref{minimal}(2) the curve has split multiplicative reduction
if~$ p \equiv 1 $ modulo~4, and non-split multiplicative reduction
if~$ p \equiv 3 $ modulo~4.
We consider the parametrisation $t:\Qbar_{p}^{\times}/Q^{\Z}\simeq E_u(\Qbar_{p})$
obtained from the Tate curve,
where~$ Q $ in~$ \Q_p^\times $ satisfies~$ |Q|_p < 1 $, and
the natural action of~$ G_{\Q_p} $ on the left-hand side is twisted by the
non-trivial quadratic character of~$ \Gal(\Q_p(i)/\Q_p) $ if~$ p \equiv 3 $
modulo~4 (see the proof of \cite[Corollary~V.5.4]{S2}). Because $\Q_p(i)/\Q_p$ is unramified, this
quadratic twist is trivial on~$ I_p $.
Since~$ E_u $ has~$ j $-invariant
\begin{equation*}
2^8 (1-u^2+u^4)^3 u^{-4} (1-u^2)^{-2} = \frac{1 + 744 Q + 196884 Q^2 + \dots}{Q}
\,,
\end{equation*}
we find
\begin{equation*}
Q = \frac{1}{2^8} u^4 (1-u^2)^2 \frac{1 + 744 Q + 196884 Q^2 + \dots}{(1-u^2+u^4)^3}
\,.
\end{equation*}
We have~$ |u^2-1|_p < 1 $ because~$ p $ divides~$ u^2-1 $, and with~$ 1 + p \Z_p $ admitting
unique square roots, we see that~$ Q = \widetilde Q^2 $ for~$ \widetilde Q = w (1-u^2)/16 $ with~$ w $ in~$ 1 + p \Z_p $.
Since~$ \ord_p(u^2-1) $ is odd, it follows that~$ (E_u[2^\infty](-1))^{I_p} =t( \mu_{2^\infty} \cdot \langle \widetilde Q \rangle) (-1) $.

The short exact sequence~\eqref{sos2mult} over~$ \Qbar_p $
gives rise to a long exact sequence of~$ I_p $-cohomology,
and the kernel of~$ H^1(I_p, E_u[2](-1) ) \to H^1(I_p, E_u[2^\infty](-1) ) $
is generated by the image of~$ (E_u[2^\infty](-1))^{I_p} / 2 \simeq \Z/2\Z $ under the
connecting homomorphism. 
The non-trivial class in~$ (E_u[2^\infty](-1))^{I_p} / 2 $ is that of~$t( \widetilde Q) $, and it gives the element of~$ H^1(I_p, E_u[2](-1) ) $
that maps~$ \sigma $ in~$ I_p $ to $ t(\sigma(\widetilde Q^{1/2}) \widetilde Q^{-1/2} ) $,
which is~$ t(1) $ if~$ \sigma $ is in~$ G_{\Q_p^\unr(\sqrt p)} \subseteq I_p $,
and~$ t(-1) $ if~$ \sigma $ is not in~$ G_{\Q_p^\unr(\sqrt p)} $.
Since $\pm 1$ are units in $\Z_p$, under~$ t $ they map to points of~$ E_u[2] $
with nonsingular reduction modulo~$p$ \cite[p.432]{S2}, i.e., to~$O$ and $(0,0)$ (cf.~the proof of Proposition~\ref{minimal}(2)).
Because~$ \Q_p(\sqrt D) \subseteq \Q_p^\unr(\sqrt p) $, we conclude
that~$ g $ does not contribute to the image
of~$ \res_p([g+h]) $ in~$ H^1(I_p, E_u[2^\infty](-1)) $. So
this image is trivial if and only if~$ \Q_p(\sqrt{D'}) \subseteq \Q_p^\unr $,
i.e., if and only  if~$ \ord_p(D') $ is even.

Assume this to be the case, so the row condition is fulfilled.
Using Lemma~\ref{2kernel} to replace~$ D $ with~$ (u^2-1) D $
if necessary, we may assume that~$ \ord_p(D) $ is also even,
so that~$ \sqrt{D} $ and~$ \sqrt{D'} $ are in~$ \Q_p^\unr $.
Then~$ \res_p([g+h]) $ in~$ H^1(\F_p , (E_u[2^\infty](-1))^{I_p} ) $ is 
obtained from the element in~$ H^1(\F_p , (E_u[2](-1))^{I_p} ) $
induced by~$ g+h $.
(Recall that we identify $ G_{\F_p} $ with~$ \Gal(\Q_p^\unr / \Q_p) $.)
The subgroup~$ t( \mu_{2^\infty}) (-1) $
in our earlier description of~$ (E_u[2^\infty](-1))^{I_p}  $ as~$t( \mu_{2^\infty} \cdot \langle \widetilde Q \rangle) (-1) $
corresponds to~$ V^{I_p} / T^{I_p} $, and because~$ (0,0) = t(-1) $, $ g $ always satisfies the column condition.
Because the inclusion of~$ \langle (-1,0) \rangle (-1) $
into~$ ( E_u[2^\infty](-1) )^{I_p} $ results in an isomorphism of that subgroup
with the cokernel of~$ V^{I_p} / T^{I_p} \to (V/T)^{I_p} $,
the column condition is that~$ h $ gives the
trivial element in~$ H^1(\F_p , \langle (-1,0) \rangle (-1) ) $,
i.e., that~$ D' $ is a square in~$ \Q_p $.
Therefore~$ \res_p([g+h]) $ is in~$ H_f^1(\Q_p, E_u[2^\infty](-1) ) $
if and only if~$ D' $ is a square in~$ \Q_p $.

(4)
The curve has split multiplicative reduction at~$ p $ by
Proposition~\ref{minimal}(2).
Using again the Tate curve, we see from the~$ j $-invariant
that here~$ Q = \widetilde Q^4 $ for~$ \widetilde Q = w u/4 $ with~$ w $ in~$ 1 + p \Z_p $.
Since~$ \ord_p(u) $ is odd, we have that~$ (E_u[2^\infty](-1))^{I_p}=t ( \mu_{2^\infty} \cdot \langle \widetilde Q \rangle) (-1) $ on the Tate curve.
The kernel of~$ H^1(I_p, E_u[2](-1) ) \to H^1(I_p, E_u[2^\infty](-1) ) $
in the long exact sequence of~$ I_p $-cohomology associated to~\eqref{sos2mult}
is generated by the image under the connecting homomorphism of
the class of~$t( \widetilde Q) $ (with Tate twist) in~$ (E_u[2^\infty](-1))^{I_p} / 2 $.
This maps~$ \sigma $ in~$ I_p $ to $t( \sigma( \widetilde Q ^{1/2}) \widetilde Q^{-1/2}) = t(\pm 1)$.
(Again we do not write the Tate twist at the level of elements, because~$ I_p $ acts trivially through the Tate twist.)
This time the points in~$ E_u[2](-1) $ of nonsingular reduction are~$ O $ and~$ (-1,0) $.
Since~$ \ord_p(u) $ is odd, we have~$ \Q_p^\unr( \widetilde Q^{1/2} ) = \Q_p^\unr(\sqrt{p}) $,
so that~$ \res_p([h]) $ always becomes trivial in~$ H^1(I_p, E_u[2^\infty](-1) ) $.
Therefore~$ \res_p([g+h]) $ has trivial image in~$ H^1(I_p, E_u[2^\infty](-1)) $
if and only if~$ D $ is a square in~$ \Q_p^\unr $, i.e., if~$ \ord_p(D) $ is
even.

Assuming this to be the case, the row condition is
fulfilled, and~$ \res_p([g+h]) $ comes from an element of~$ H^1(\F_p , (E_u[2^\infty](-1))^{I_p} ) $.
For the column condition, we then need this element to become trivial when we quotient out the coefficient
group~$ (E_u[2^\infty](-1))^{I_p}=t(\mu_{2^\infty} \cdot \langle \widetilde Q \rangle) (-1) $ by its subgroup~$ V^{I_p} / T^{I_p}=t(\mu_{2^\infty})(-1) $.
The resulting quotient is isomorphic to~$ \langle \widetilde Q \rangle / \langle Q \rangle (-1) $. So the column condition is equivalent to~$ [g+h] $ giving rise to the
trivial element of~$ H^1(\F_p , \langle \widetilde Q \rangle / \langle Q \rangle (-1) ) $ in this way.
Below we make this explicit using the 
long exact sequence of~$ G_{\F_p} $-cohomology associated to the short exact
sequence
\begin{equation} \label{Qsos}
 0 \to \langle \widetilde Q^2 \rangle / \langle Q \rangle (-1) \to \langle \widetilde Q \rangle / \langle Q \rangle (-1) \xrightarrow{2} \langle \widetilde Q^2 \rangle / \langle Q \rangle (-1) \to 0.
\end{equation}
Note that, under the projection from $ (E_u[2^\infty](-1))^{I_p}$ to $t(\langle \widetilde Q \rangle / \langle Q \rangle)(-1)$, the points $(-1,0)$ and $(0,0)$ map to the trivial element and the class of $\widetilde Q^2$, respectively.

First we assume that~$ \ord_p(D') $ is also even, so that~$ \sqrt{D} $ and~$ \sqrt{D'} $
are in~$ \Q_p^\unr $, hence~$ g $ and~$ h $ are
trivial on~$ I_p $, and~$ g+h $ defines an element of~$ H^1(\F_p, E_u[2](-1) ) $.
From the long exact sequence of~$ G_{\F_p} $-cohomology associated
to~\eqref{Qsos}, and considering how~$ G_{\F_p} $ acts on~$ \mu_4 $ in the twist inside $H^0(\F_p , \langle \widetilde Q \rangle / \langle Q \rangle (-1) ) $,
we find that~$ H^1(\F_p , \langle \widetilde Q^2 \rangle / \langle Q \rangle (-1) ) \to H^1(\F_p , \langle \widetilde Q \rangle / \langle Q \rangle (-1) ) $
is injective if~$ p \equiv 1$ modulo~4. Recalling what was just said about the images of~$ (-1,0) $
and~$ (0,0) $  in~$ t( \langle \widetilde Q \rangle / \langle Q \rangle)(-1) $,
we see that for such~$ p $, the
column condition is equivalent to the map induced by~$ g $
on~$ G_{\F_p} $ being trivial, i.e., to~$ D $
being a square in~$ \Q_p $. But for~$ p \equiv 3$ modulo~4
the column condition is always satisfied  because the connecting
homomorphism~$ H^0(\F_p , \langle \widetilde Q^2 \rangle / \langle Q \rangle (-1) ) \to H^1(\F_p ,  \langle \widetilde Q^2 \rangle / \langle Q \rangle (-1) ) $
is an isomorphism.
(Alternatively, here~$ u^2-1 $ is not a square in~$ \Q_p $,
and by Lemma~\ref{2kernel} we can replace~$ D $ with~$ (u^2-1) D $
if necessary and assume~$ D $ is a square in~$ \Q_p $.)

Now assume that~$ \ord_p(D') $ is odd. Then~$ h $ is non-trivial on~$ I_p $,
and we replace it with~$ h - h' $, where~$ h' $ maps~$ \sigma $ in~$ G_{\Q_p} $ to~$ t( \sigma(\widetilde Q^{1/2})^{\epsilon_2(\sigma)^{-1}} \widetilde Q^{-1/2} ) $
in~$ E_u[4](-1) $, and we wrote the Tate twist explicitly. We have~$ [g+h] = [g+h-h'] $ in~$ H^1(\Q_p, E_u[2^\infty](-1) ) $, since $h'$ is a coboundary.
Because~$ \epsilon_2 $ is trivial on~$ I_p$, 
$ D' Q ^{-1} $ is a square in~$ \Q_p^\unr $, and~$  t(-1) = (-1,0) $,
$ h $ and~$ h' $ have the same restriction to~$ I_p $, hence
$ g+h-h' $ is trivial on~$ I_p $, 
and gives an element of~$ H^1(\F_p, ( E_u[4](-1) )^{I_p} ) $.

Extending the coefficients to~$ (E_u[2^\infty](-1))^{I_p} = (V/T)^{I_p} $,
quotienting out these by~$ V^{I_p} / T^{I_p} $, and identifying the quotient
with the middle term in~\eqref{Qsos} as before, we have to consider
the cocycle~$c: G_{\F_p} \to \langle \widetilde Q \rangle / \langle Q \rangle (-1) $ obtained from~$ g+h-h' $.
We shall do this by explicitly identifying~$ G_{\F_p} $ with~$ G_p/I_p = \Gal(\Q_p^\unr / \Q_p) $.
If~$ \tilde\sigma $ is a lift to~$ \Gal(\Q_p^\unr / \Q_p) $ of the Frobenius in~$ G_{\F_p} $,
then~$c(\tilde \sigma) $ is equal to the class of the product
\begin{equation*}
\begin{Bmatrix*}[l]
1, & \text{ if $ D $ is a square in~$ \Q_p $}
\\
\widetilde Q^2 , & \text{ if $ D $ is not a square in~$ \Q_p $}
\end{Bmatrix*}
\cdot
\begin{Bmatrix*}[l]
1 , & \text{ if $ p \equiv 1 $ modulo~8}
\\
\widetilde Q^{-1} , & \text{ if $ p \equiv 3 $ modulo~8}
\\
\widetilde Q^{-2} , & \text{ if $ p \equiv 5 $ modulo~8}
\\
\widetilde Q^{-3} , & \text{ if $ p \equiv 7 $ modulo~8}
\end{Bmatrix*}
.
\end{equation*}
The first term is the contribution of~$ g $, the second
that of~$ h - h' $. For the latter, note that the values of~$ h $ in the Tate curve
are in~$ \langle -1 \rangle (-1) $, which is in the kernel of our
map to~$ \langle \widetilde Q \rangle / \langle Q \rangle (-1) $, and
the contribution of~$ -h' $ is the image of
\begin{equation*}
(\widetilde Q^{1/2})^{-\epsilon_2(\tilde\sigma)^{-1}} \widetilde Q^{1/2} 
=
\widetilde Q^{(1-p)/2}
\end{equation*}
because the sign in~$ \tilde \sigma(\widetilde Q ^{1/2}) = \pm \widetilde Q ^{1/2} $
similarly does not matter, $ \widetilde Q^{1/2} $ is 8-torsion
in the Tate curve, and~$ p^{-1} \equiv p $ modulo~8.

We now consider the long exact sequence of~$ G_{\F_p} $-cohomology
obtained from~\eqref{Qsos}.
The image of $c$ under the map~$ H^1(\F_p, \langle \widetilde Q \rangle / \langle Q \rangle (-1) ) \xrightarrow{2} H^1(\F_p, \langle \widetilde Q^2 \rangle / \langle Q \rangle (-1) ) $,
is non-zero if~$ p \equiv 3 $ or~7 modulo~8, i.e., we do not
obtain the trivial homomorphism~$ G_{\F_p} \to \langle \widetilde Q^2 \rangle / \langle Q \rangle (-1) ) $;
hence the column condition for such~$ p $  is not satisfied with~$ \ord_p(D') $ odd.
On the other hand, for~$ p \equiv 1 $ or~5 modulo~8, taking~$ G_{\F_p} $-invariants
in~\eqref{Qsos} is again exact, so that the map~$ H^1(\F_p, \langle \widetilde Q^2 \rangle / \langle Q \rangle (-1) ) \to H^1(\F_p, \langle \widetilde Q \rangle / \langle Q \rangle (-1) ) $
obtained from~\eqref{Qsos} is injective.  The cocycle~$ c $
takes values in the subgroup~$ \langle \widetilde Q^2 \rangle / \langle Q \rangle (-1) $ of~$ \langle \widetilde Q \rangle / \langle Q \rangle (-1)  $
in this case, so the column condition is equivalent to it giving the trivial
class in~$ H^1(\F_p, \langle \widetilde Q^2 \rangle / \langle Q \rangle (-1) ) $,
which means it defines the trivial homomorphism~$ G_{\F_p} \to \langle \widetilde Q^2 \rangle / \langle Q \rangle (-1) $.
From the above calculation we now see that
this is equivalent to~$ D $ being a square in~$ \Q_p $
if~$ p \equiv 1 $ modulo~8, and a non-square if~$ p \equiv 5 $
modulo~8.
Because~$ 2 $ is a square in~$ \Q_p $ if and only if~$ p \equiv 1 $
or~7 modulo~8, we can formulate this as~$ 2 D $ being a square
in~$ \Q_p $ for~$ p \equiv 1 $ modulo~4, all under the assumption
that~$ \ord_p(D') $ is odd.

Finally, the conditions found depending on the parity of~$ \ord_p(D') $ can
be combined into that given in the proposition.

(5)
For~$ p = 2 $, we first make explicit~$ H_f^1(\Q_2, V) $. At
the beginning of Section~\ref{2Tam2} we saw that~$ \phi $ acting
on~$ \mathbf{D}(V) $, as defined in~\eqref{Dcris-def}, satisfies~$ \phi^2 + \phi + 2 = 0 $. Recall that 
$$\mathbf{D}(V):=D_{\cris}(V):=(B_{\cris}\otimes V)^{G_{\Q_2}}.$$ Around \cite[(1.8)]{BK} an element $t\in B_{\cris}$ is introduced, such that $\phi(t)=2t$ and $G_{\Q_2}$ acts on $t$ via the 2-adic cyclotomic character. (We already met this element in the proof of Proposition~\ref{2tam2}.) It follows that $x\mapsto t^{-2}x$ is a bijection from $\mathbf{D}(V)$ to $\mathbf{D}(V(2))$. Hence the action of $\phi$ on~$ \mathbf{D}(V(2)) $ is such that $(4\phi)^2+(4\phi)+2=0$, i.e.,~$ \phi^2 + \frac14 \phi + \frac18 = 0 $.

According to \cite[Corollary~3.8.4]{BK},$$ H_f^1(\Q_2, V(2))/ H_e^1(\Q_2, V(2))\simeq \mathbf{D}(V(2))/(1-\phi)\mathbf{D}(V(2)),$$
where $H_e^1(\Q_2, V(2))$ is defined like $H_f^1(\Q_2, V)$, but replacing~$ B_{\cris} $ with its
subring~$ B_{\cris}^{\phi=1} $ and~$ V $ with~$ V(2) $.
Since~$1$ is not a root of the above polynomial, it follows that~$ H_f^1(\Q_2, V(2))=H_e^1(\Q_2, V(2))$.
(Note that~$ \phi $ and~$ \mathbf{D}(V(2)) $ are denoted~$ f $ and~$ \textup{Crys}(V(2)) $ respectively 
in loc.\ cit.)

Since~$ V(2)^*(1) = V^*(-1) = V $,  we conclude using Proposition~3.8
of loc.\ cit.\ that~$ H_f^1(\Q_2, V) = H_g^1(\Q_2, V) $, as annihilators, with respect to the local Tate duality pairing, of  $ H_f^1(\Q_2, V(2))$ and $H_e^1(\Q_2, V(2))$, respectively.
Here $H_g^1(\Q_2, V) $ is defined by replacing~$ B_{\cris} $ with the
larger ring~$ B_{\dR} $ in the definition of~$ H_f^1(\Q_2, V) $.
Finally, according to \cite[Lemma~2]{Flach90},~$H_g^1(\Q_2, V) = \ker(H^1(\Q_2, V)\to H^1(I_2, V/F^1V))$, where the filtration of the ordinary representation $V$ is such that $I_2$ acts on the $j^{\mathrm{th}}$ graded piece as $\epsilon_2^j$. Recalling the composition factors of $ V $ found in the proof of Proposition~\ref{2tam2left},
the graded pieces are trivial except for~$j=0$ and~$-1$, so $F^1V=0$, and we conclude that
$$H_f^1(\Q_2, V) = \ker(H^1(\Q_2, V)\to H^1(I_2, V)) .$$
This matches the description for the other primes, and we can use the same approach,
based on~\eqref{infCD}.

If we let~$ K \subseteq E_u[2^\infty] $ be the kernel of the reduction
map at~2, then we have a short exact sequence
$ 0 \to K \to E_u[2^\infty] \to \mathcal{E}_{u,\F_2}[2^\infty] \to 0 $
that matches the composition factors of~$ V(1) $.
Here~$ I_2 $ acts trivially on the last term and through~$ \epsilon_2 $
on the first.
From Proposition~\ref{minimal}(3)
we see that~$ P = (u, u(u+1)) $ and~$ 2 P = (0,0) $ are not in~$ K $,
which implies that
\begin{equation} \label{EIdescription}
(E_u[2^\infty](-1))^{I_2} = K(-1) \oplus \langle (0,0) \rangle (-1) 
.
\end{equation}
From the long exact sequence of~$ I_2 $-cohomology associated to the
short exact sequence~\eqref{sos2mult},
we find that the kernel of~$ H^1(I_2, E_u[2](-1)) \to H^1(I_2, E_u[2^\infty](-1)) $
is generated by the image under the connecting homomorphism of~$ \langle (0,0) \rangle (-1) $,
which maps~$ \sigma $ in~$ I_2 $ to~$ \epsilon_2(\sigma)^{-1} \sigma(P) - P $.
This is trivial if~$ \sigma $ is in~$ G_{\Q_2^\unr(i)} \subseteq I_2 $, the generator of~$ \langle (0,0) \rangle (-1) $ if~$ \sigma $ is
not in~$ G_{\Q_2^\unr(i)} $.
So~$ \res_2([g+h)] $ maps to the trivial element of~$ H^1(I_2, E_u[2^\infty](-1) ) $
if and only if~$ D' $ and one of~$ D $ and~$ -D $ are squares in~$ \Q_2^\unr $.
Assuming this to be the case, using Lemma~\ref{2kernel} we may replace~$ D $ with~$ (u^2-1) D $ if necessary,
and assume~$ D $ and~$ D' $ are squares in~$ \Q_2^\unr $.
Then the row condition is fulfilled,~$ g+h $ is trivial on~$ I_2 $, and gives an element of~$ H^1(\F_2 , E_u[2](-1)) $.

From~\eqref{EIdescription} and the composition factors of~$ V $,
we obtain the short exact
sequence~$ 0 \to V^{I_2}/ T^{I_2} \to (V/T)^{I_2} \to \langle (0,0) \rangle (-1) \to 0 $,
because the first term is~$ K(-1) $ and the second~$ (E_u[2^\infty](-1))^{I_2} $.
So the column condition for~$ [g+h] $ is that
the map induced by~$ g $ on~$ G_{\F_2} $ is trivial, i.e., that~$ D $ is a square in~$ \Q_2 $.

Therefore $ \res_2([g+h]) $ is in $ H_f^1(\Q_2, E_u[2^\infty](-1)) $
if and only if~$ D' $ is a square in~$ \Q_2^\unr $
and~$ D $ or~$ -D $ is a square in~$ \Q_2 $.
The condition on~$ D' $ is equivalent to~$ \Q_2(\sqrt{D'}) \subseteq \Q_2(\sqrt{-3}) $,
the unramified quadratic extension of~$ \Q_2 $, so it can also
be formulated as~$ D' $ or~$ -3 D' $ being a square in~$ \Q_2 $.
\end{proof}

\begin{cor} \label{2torscor}
With notation and assumptions as in Theorem~\ref{f-conditions},
let~$ S $ be the set of prime divisors of~$ u^2-1 $, and~$ S' $
the set of prime divisors of~$ u $ that are congruent to~1 modulo~4.
Then the 2-torsion in~$ H_f^1(\Q, E_u[2^\infty](-1) ) $ is in bijection
with pairs~$ (D, D') $ of positive squarefree integers, where
the prime factors of~$ D $ are in~$ S $ and those of~$ D' $ in~$ S' $,
and which satisfy
\begin{itemize}
\item
$ D' $ is a square modulo~$ p $ for every~$ p $ in~$ S $;

\item
$ 2 ^{\ord_p(D')} D $ is a square modulo~$ p $ for every~$ p $ in~$ S' $;

\item
$ D \equiv 1 $ modulo~8.
\end{itemize}
\end{cor}

\begin{proof}
The 2-torsion in~$ H_f^1(\Q, E_u[2^\infty](-1) ) $ is the 2-torsion
of~$ H^1(\Q, E_u[2^\infty](-1) ) $ such that its image under~$ \res_p $
lies in~$ H_f^1(\Q_p, E_u[2^\infty](-1) ) $ for all~$ p $, including~$ p = \infty $.
By Lemma~\ref{2kernel}, the 2-torsion of~$ H^1(\Q, E_u[2^\infty](-1) ) $
is the image of~$ H^1(\Q, E_u[2](-1) ) $ under extension of the
coefficients to~$ E_u[2^\infty](-1) $.
We described~$ H^1(\Q, E_u[2](-1) ) $ using elements of~$ \Q^\times/2 \times \Q^\times/2 $,
and imposing the relations coming from Lemma~\ref{2kernel} then
describes its image. Representing the classes in~$ \Q^\times/2 \times \Q^\times/2$
by a pair~$ (D, D') $ of squarefree integers, we see from parts~(2), (4) and~(5) of Theorem~\ref{f-conditions}
that the only prime divisors of~$ D $ are those in~$ S $,
and using parts~(2), (3),(4) and~(5) that those of~$ D' $ are in~$ S' $.
Imposing the condition in part~(5) that~$ D $ or~$ -D $ is a square in~$ \Q_2 $,
so that~$ D \equiv \pm 1 $ modulo~8,
we can use Lemma~\ref{2kernel} to modify~$ D $ using~$ -1 $ and~$ u^2-1 $,
and normalise the pair~$ (D, D') $ uniquely
such that~$ D > 0 $ and~$ D \equiv 1 $ modulo~8 because~$ u^2 -1 \equiv -1 $
modulo~8. Then the only condition left on~$ D $ is that~$ 2^{\ord_p(D')} D $
is a square in~$ \Q_p $ for~$ p $ in~$ S' $, as in Theorem~\ref{f-conditions}(4).
Part~(1) of the proposition requires that~$ D' > 0 $, and the condition in~(5) on~$ D' $
is fulfilled because its prime factors are~1 or~5 modulo~8.
So the only remaining condition is that in part~(3) of the proposition,
but because~$ D' $ has no factors in~$ S $, and all primes in~$ S $
are odd, this simplifies to~$ D' $ being a square modulo each
prime~$ p $ in~$ S $.
\end{proof}

\begin{remar} \label{selmertrivial}
There does not seem to be a result in the literature that implies~$ H_f^1(\Q, E_u[2^\infty](-1)) $
is finite. But it is, of course, trivial if and only
if its 2-torsion is trivial, which we can sometimes verify in
examples (see Remark~\ref{2torsrem} and Section~\ref{num-sec}
for examples).
\end{remar}

\begin{remar} \label{2torsrem}
If~$ s = \#S $ and~$ s' = \#S' $, and we identify~$ D $ and~$ D' $
with their classes in~$ \Q^\times/2 $, then the corollary imposes~$ s + s' + 2 $
linear conditions on an~$ \F_2 $-vector of dimension~$ s + s' $.
Thus one might
expect the 2-torsion in~$ H_f^1(\Q, E_u[2^\infty](-1)) $, and
hence the group itself, to be trivial in general.
This applies, for example, for~$ u=4 $, where~$ S = \{3, 5\} $ and~$ S' = \emptyset $,
and the corollary gives us only~$ (D, D') = (1,1) $.
But the situation is not symmetric in~$ D $ and~$ D' $. If we
take~$ D'=1 $ then we impose~$ s' + 2 $ conditions on an~$ \F_2 $-vector
space of dimension~$ s $. Because~$ u^2-1 = (u-1)(u+1) $ already,
in practice one often has~$ s > s'+2 $, resulting in non-trivial
2-torsion.

But we find not only pairs~$ (D, D')  $ with~$ D' = 1 $. For example, for~$ u = 292 $ we have~$ S = \{3, 97, 293\} $ and~$ S' = \{73\} $,
which results in~$ (D,D') = (1,1) $, $ (97,1) $, $ (1, 73) $ and~$ (97, 73) $.
For~$ u = 1020 $ we have~$ S = \{1019, 1021\} $ and~$ S' = \{5, 17\} $, which
gives~$ (D, D') = (1,1) $ or~$ (1, 17) $.
For~$ u = 1060 $ we have~$ S = \{3, 353, 1061 \} $ and~$ S' = \{5, 53\} $,
and find that~$ (D,D') = (1,1) $ or~$ (353, 5 \cdot 53) $.

For more examples of the order of the 2-torsion in~$ H_f^1(\Q, E_u[2^\infty](-1) ) $,
we refer to Section~\ref{num-sec}, and especially to the tables there.
\end{remar}

\begin{remar}
In the proof of Proposition~\ref{2tam2left} we saw~$ H_f^1(\Q_2, V) = 0 $,
so the condition in Theorem~\ref{f-conditions}(5) can also be
formulated as~$ [g+h] $ having trivial image in~$ H^1(\Q_2, E_u[2^\infty](-1)) $.
This can, in fact, be used directly by considering the image
of~$ H^0(\Q_2 ,  E_u[2^\infty](-1) ) / 2 $
under the connecting homomorphism in the long exact sequence
of~$ G_{\Q_2} $-cohomology. But the proof we used
for that part of the proposition
is far less computational, and bears a greater similarity to the
proofs of the other parts of the proposition.
\end{remar}

\section{\texorpdfstring{$K_2$}{K2} and the regulator}\label{K2reg}

Motivic cohomology, Deligne cohomology, and regulator maps from one to the other, may be defined in a more general setting than what we saw in Section~\ref{regBeil}. In particular, for $X/\Q$ a smooth, quasi-projective variety, there exist maps
$$\mathrm{reg}:H^{\cdot}_{\MMM}(X, \Q(*))\rightarrow H^{\cdot}_{\DDD}(X_{\R}, \R(*)),$$
which are functorial and respect products \cite[\S 2.6]{DS}. We have $H^1_{\MMM}(X, \Q(1))\simeq \OOO^*(X)\otimes\Q$, where $\OOO^*(X)$ is the group of units of $X$, i.e., invertible global sections of the structure sheaf of $X$. 

Let $\AAA^{\cdot}$ be the de Rham complex of real-valued $C^{\infty}$-forms on $X(\C)$, and let $\pi_k:\C\rightarrow (2\pi i)^k\R$ be the projection taking the real or imaginary part according as $k$ is even or odd, respectively. Then by \cite[(2.5.1)]{DS}, $H^p_{\DDD}(X_{\R}, \R(p))$ can be described as
$$\frac{\left\{\phi\in H^0(X(\C), \AAA^{p-1}\otimes(2\pi i)^{p-1}\R)\mid d\phi=\pi_{p-1}(\omega),\,\omega\in H^0(\overline{X}(\C),\Omega^p_{\overline{X}}\langle D\rangle)\right\}}{dH^0(X(\C), \AAA^{p-2}\otimes(2\pi i)^{p-1}\R)},$$
where $\overline{X}$ is a nonsingular projective variety in which $X$ is an open subvariety, with $D:=\overline{X}-X$ a divisor with normal crossings.

In Section~\ref{regBeil}, our choice of rational and integral structures on $H^{m+1}_{\DDD}(X_{\R}, \R(r))$ was guided by the isomorphism $$H^{m+1}_{\DDD}(X_{\R}, \R(r))\simeq \frac{H^m_{\dR}(X_{\R})/\Fil^r H^m_{\dR}(X_{\R})}{H^m_B(X(\C), \R(2\pi i)^r)^+}.$$
The equivalence of this with the alternative description
\begin{equation*}
 H^{m+1}_{\DDD}(X_{\R}, \R(r))\simeq \frac{H^m_B(X(\C), \R(2\pi i)^{r-1})^+}{\Fil^r H^m_{\dR}(X_{\R})}
\end{equation*}
is an elementary consequence of 
$$H^m_{\dR}(X_{\R})\simeq H^m_B(X(\C), \R)^+\oplus iH^m_B(X(\C), \R)^-.$$
The second description gives another approach to rational and integral structures.
We explore the relation between these integral structures in the case $H^2_{\DDD}(E_{u,\R}, \R(2))$ of direct interest to us.

\begin{lem}\label{perioddet}
Assume that~$ u $ is an integer congruent to~4 modulo~8. Then under the period isomorphism
$${\det}_{\C}H^1_{\dR}(E_{u,\C})\xrightarrow{\sim}{\det}_{\C}H^1_B(E_u(\C), \C),$$
we have $${\det}_{\Z}H^1_{\dR}(\mathcal{E}_u/\Z)\xrightarrow{\sim}(2\pi i){\det}_{\Z}H^1_B(E_u(\C), \Z).$$
\end{lem}
\begin{proof} Recall from Proposition~\ref{minimal}(3) that a global minimal Weierstrass equation is $${y'}^2+x'y'={x'}^3+\frac{u^2}{4}{x'}^2+\frac{u^2}{16}x',$$ where $y'=\frac{y-x}{8}$ and $x'=\frac{x}{4}$. It follows from $H^1_{\dR}(\mathcal{E}_u/\Z_2)\simeq H^0(\mathcal{E}_u/\Z_2, \Omega^1(2O))$ \cite[A.1.2.3]{K} that
\begin{equation*}
\left\{\frac{dx'}{2y'+x'},\,x'\,\frac{dx'}{2y'+x'}\right\}=\left\{\frac{dx}{y},\,\frac{1}{4}x\,\frac{dx}{y}\right\}
\end{equation*}
is a $\Z$-basis for $H^1_{\dR}(\mathcal{E}_u/\Z)$.

Consider $E_u/\C$ in the form $Y^2=4X^3-g_2X-g_3$, with $(X,Y)=(\wp(z), \wp'(z))$ for $z\in\C/\Lambda$. Note that $\frac{dX}{Y}=dz,\,X\,\frac{dX}{Y}=\wp(z)\,dz$. Let
$$\omega^{\pm}:=\int_{\gamma^{\pm}}\,\frac{dX}{Y},\,\,\eta^{\pm}:=\int_{\gamma^{\pm}}\,X\,\frac{dX}{Y},$$ where $H_1(E_u(\C),\Z)=\Z\gamma^+\oplus\Z\gamma^-$, so $\Lambda=\Z\omega^+\oplus\Z\omega^-$. (We may choose the signs of $\gamma^{\pm}$ in such a way that $\omega^+, \omega^-/i\in\R_{>0}$.)
By Legendre's period relation \cite[Exercise 6.4(d)]{S1}, we
now have~$\omega^-\eta^+-\omega^+\eta^-=2\pi i$.

Because~$Y^2=4X^3-g_2X-g_3$, we have $ (4Y)^2=(4X)^3-4g_2(4X)-16g_3$. Changing this to $y^2=x(x+1)(x+u^2)$ would transform the differentials as 
$$\frac{dx}{y}=\lambda\,\frac{d(4X)}{4Y},\,\,x\,\frac{dx}{y}=\lambda^{-1}\,(4X)\,\frac{d(4X)}{4Y},$$ for some $\lambda\in\C^{\times}$. Therefore integrating instead $\left\{\frac{dx}{y},\,\frac{1}{4}x\,\frac{dx}{y}\right\}$ would have produced $\lambda\omega^{\pm}$ and $\lambda^{-1}\eta^{\pm}$, with no impact on the determinant. This proves the lemma.
\end{proof}
As noted above, we have the alternative descriptions
\begin{equation} \label{reg22target}
H^2_{\DDD}(E_{u,\R}, \R(2))\simeq \frac{H^1_{\dR}(E_{u,\R})}{(2\pi i)^2H^1_B(E_u(\C), \R)^+}\simeq (2\pi i) H^1_B(E_u(\C), \R)^-
.
\end{equation}
Note that it is $1$-dimensional, hence equal to its determinant.
Recall from Section~\ref{BlKa} the integral line $$\DDD_{1,2} :=\det (H^1_{\dR}(\mathcal{E}_u/\Z))\otimes{\det}^{\vee}(H^1_B(E_u(\C), \Z(2\pi i)^2)^+$$ in $\det H^2_{\DDD}(E_{u,\R}, \R(2))=H^2_{\DDD}(E_{u,\R}, \R(2))$, where we are looking at the first isomorphism. 
\begin{prop}\label{Dint} Assume that $u\equiv 4\pmod{8}$ is an integer. Looking at the second isomorphism, $$\DDD_{1,2}=\frac{1}{2\pi i}\, H^1_B(E_u(\C), \Z)^-.$$
\end{prop} 
\begin{proof} By Lemma~\ref{perioddet} we have 
\begin{alignat*}{1}
{\det}_{\Z}H^1_{\dR}(\mathcal{E}_u/\Z)& \xrightarrow{\sim}(2\pi i){\det}_{\Z}H^1_B(E_u(\C), \Z) \\ & \simeq \frac{1}{2\pi i}H^1_B(E_u(\C), \Z)^-\otimes H^1_B(E_u(\C), (2\pi i)^2\Z)^+
.
\end{alignat*}
Hence $$\DDD_{1,2}=\det (H^1_{\dR}(\mathcal{E}_u/\Z))\otimes{\det}^{\vee}(H^1_B(E_u(\C), (2\pi i)^2\Z)^+\simeq \frac{1}{2\pi i}\, H^1_B(E_u(\C), \Z)^-,$$ as required.
\end{proof}

\begin{remar} \label{rem8.6}
We shall determine~$ R_u $ in~\eqref{Ru-def} in Section~\ref{hyper-sec}
by other means for the elements that we shall construct in Section~\ref{K2elements},
but for the convenience of the reader we recall here a more classical description
(cf.~\cite[p.251/252]{dewi88} or~\cite[\S3]{DdJZ}).

By \cite[(2.6.1)]{DS}, if $g\in\OOO^*(X)$, considered as an element of $H^1_{\MMM}(X, \Q(1))$, then $\reg(g)\in H^1_{\DDD}(X_{\R}, \R(1))$ is represented by the $C^{\infty}~~$ $0$-form $\phi_g=\log |g|$. Note that 
$$dg=\pi_0(\omega_g),\,\,\,\omega_g=d\log g=d\log |g|+i\,d\arg(g).$$
Given $f,g\in\OOO^*(X)$, we may produce an element~$f\cup g$ in~$ H^2_{\MMM}(X, \Q(2))$,
and by the compatibility of regulators and cup products, we have~$\reg(f \cup g)=[\phi_f]\cup[\phi_g]$.
According to \cite[7 lines below (2.5.1)]{DS}, the cup product on the right is represented by the $C^{\infty}~~$ $1$-form $\phi_f\wedge\pi_1\omega_g-\phi_g\wedge\pi_1\omega_f$. This is
\begin{equation*}
\eta(f,g):= i(\log |f|\,d\arg(g)-\log |g|\,d\arg(f))
.
\end{equation*}
As mentioned in the introduction, we shall view~$ H^2_{\MMM}(E_u, \Q(2)) $ 
as~$ K_2^T(E_u) \otimes_\Z \Q $ (see the beginning of Section~\ref{2-BK}). Here an element~$ \alpha $ of~$ K_2^T(E_u) $
is a sum of such~$ f \cup g $ that is in the kernel of the tame symbol (see~\eqref{K2locseq},
where~$ f \cup g = \symb f g $ in~$ K_2(\Q(E_u)) $). That~$ \alpha $
is in the kernel of the tame symbol implies that the sum of the~$ \eta(f,g) $ gives a well-defined integration map from $H_1(E_u(\C), \Z)$ to $i\R$. This is explained in the paragraph straddling pages 342 and 343 of \cite{DdJZ}. Hence it represents a class~$ \eta_\alpha $ in~$ i H_B^1(E_u(\C) \setminus D , \R)^-$ coming from
its subspace~$ i H_B^1(E_u(\C), \R)^-$, where $D$ is the union of the sets of zeros and poles of all the $f$ and $g$. Up to sign, $ R_u = \frac{1}{2 \pi i} \int_{\gamma^-} \eta_\alpha $, with~$ \gamma^- $
a generator of~$ H_B^1(E_u(\C), \Z)^- $ (cf.~\cite[(3.4)]{DdJZ},
where factors~$ i $ were cancelled).
\end{remar}

\section{\texorpdfstring{$\ell$}{l}-adic regulator maps} \label{ellreg}

In this section and Section~\ref{K2elements} we discuss in detail
a map~$ \reg_{\ell} $, which in Section~\ref{2-BK} will give
rise to the map~$ \reg_{\ell}^{\Q_\ell} $ mentioned at the beginning of Section~\ref{BlKa}.
In Section~\ref{2-BK} we shall need information on the (in)divisibility of the image
under~$ \reg_{\ell} $ of certain elements, which we shall prove
in Theorem~\ref{new-index-theorem}.
Because the proof of this result is subtle
and depends on computation with torsion,
we work with the original
Chern class map on~$ K_2(E_u) $ in order to construct~$ \reg_{\ell} $,
and relate it to the map on~$ H^2_{\MMM}(E_u, \Q(2)) $ only in Section~\ref{2-BK}.

We discuss the set-up of \cite[p.370]{BK} for the elliptic curves
defined by~\eqref{Eueq}.
For any prime number~$\ell$, and~$ m \ge 1 $, we have a group homomorphism 
$$\mathrm{ch}_{\ell,m}: K_2(E_u)\rightarrow \het^2(E_u, (\Z/\ell^m\Z)(2))=\het^2(E_u, \mu_{\ell^m}^{\otimes 2}).$$
This is Soul\'e's Chern class map  \cite{G, Sh, So}, modified as in \cite[(2.2), Lemma 2.3]{W} when~$\ell=2$.
In other places (e.g., \cite{So, Ta, W}) this map is denoted~$\mathrm{ch}_{2,2}$,
where the subscripts refer to the degree and the twist in the codomain.
Because those are fixed in our context, we choose indices relating
to the coefficients instead.

There is a Hochschild-Serre spectral sequence
\begin{equation} \label{HS-ss}
E_{2,m}^{p,q}=H^p(\Q, \het^q(E_{u,\Qbar}, (\Z/\ell^m\Z)(2)))\implies \het^{p+q}(E_u, (\Z/\ell^m\Z)(2))
\end{equation}
for~$ E_{u, \Qbar} / E_u $ (see \cite[Remark~III.2.21(b)]{Mil80}). It gives rise to a filtration 
$$ \het^2(E_u, (\Z/\ell^m\Z)(2))=\Fil_m^0\supseteq \Fil_m^1\supseteq \Fil_m^2\supseteq \Fil_m^3=\{0\},$$
with $\Fil_m^p/\Fil_m^{p+1}$ a subquotient of $H^p(\Q, \het^{2-p}(E_{u,\Qbar}, (\Z/\ell^m\Z)(2)))$ for $p=0,1,2$.
For~$ p = 0 $ and~1 those subquotients are simply submodules
because all incoming higher differentials start at trivial groups.

The spectral sequence and the Chern class map are compatible with the natural maps as~$ m $ varies.
Taking inverse limits over~$ m $ we get a filtration
\begin{equation*}
\het^2(E_u, \Z_\ell(2))=\Fil^0\supseteq \Fil^1\supseteq \Fil^2\supseteq \Fil^3=\{0\}
.
\end{equation*}
Note that~$ \Fil_m^0 / \Fil_m^1 $ is a submodule of~$H^0(\Q, \het^2(E_{u,\Qbar}, (\Z/\ell^m\Z)(2)))$,
which is isomorphic to~$ H^0(\Q, (\Z/\ell^m\Z)(1))$.
This last group is trivial if~$ \ell $ is odd, because
it is isomorphic to the subgroup of $ \Z/\ell^m \Z $ annhilated by $ \overline{2} - \overline{1} $,
and for $ \ell = 2 $ it has order~2 because it is isomorphic to
the subgroup annihilated by $ \overline{3} - \overline{1} $.
Taking inverse limits over~$ m $ then shows that~$ \Fil^0 = \Fil^1 $.
Composing the inverse limit~$ \ch_{\ell} $ of the~$ \ch_{\ell, m}$ with the
projection~$ \pi_{\ell} : \het^2(E_u, \Z_\ell(2)) \to\Fil^0/\Fil^2 = \Fil^1/\Fil^2 $, we obtain a group homomorphism
\begin{equation*}
\pi_\ell \circ \ch_{\ell} : K_2(E_u) \to H^1(\Q, \het^1(E_{u,\Qbar}, \Z_\ell(2)))
.
\end{equation*}
In order to justify the target, note
that
\begin{equation*}
\Fil_m^1 / \Fil_m^2 = E_{\infty,m}^{1,1} \subseteq E_{2,m}^{1,1} = H^1(\Q, \het^1(E_{u,\Qbar}, (\Z/\ell^m\Z)(2)))
\end{equation*}
and that~$\underset{m}\varprojlim\,H^1(\Q, \het^1(E_{u,\Qbar}, (\Z/\ell^m\Z)(2)))$ 
identifies with~$  H^1(\Q, \het^1(E_{u,\Qbar}, \Z_\ell(2))) $
by the corollary on~\cite[p.261]{Ta}.

We recall the exact localisation sequence
\begin{equation} \label{K2locseq}
 \dots \to \oplus_P K_2(\Q(P)) \to K_2(E_u) \to K_2(\Q(E_u)) \overset{T}{\to} \oplus_P \, \Q(P)^\times \to \dots
\,,
\end{equation}
where $ P $ runs through the closed points of~$ E_u $,
and the~$ P $-component~$ T_P $ of the tame symbol~$ T $ is given by
mapping~$ \symb f g $ to~$ (-1)^{\ord_P(f) \ord_P(g)} \frac{f^{\ord_P(g)}}{g^{\ord_P(f)}} (P) $.
We recall
that we have the \emph{product formula}
\begin{equation} \label{product-formula}
\prod_P \Nm_{\Q(P)/\Q} (T_P(\alpha)) = 1
\end{equation}
in~$ \Q^\times $ for any~$ \alpha $ in~$ K_2(\Q(E_u)) $.
Writing~$ K_2^T(E_u) $ for the kernel of~$ T $,
we have that the kernel of the surjection~$ K_2(E_u) \to K_2^T(E_u) $ is torsion
because~$ K_2 $ of any number field is torsion (see \cite[p.155]{Wei04}).
Hence there is an induced map~$ \reg_{\ell} $ that
fits into a commutative diagram
\begin{equation}
\begin{split} \label{K2TregCD}
\xymatrix{
 K_2(E_u) \ar[r] \ar[d]_-{\mathrm{ch}_{\ell}} & K_2^T(E_u) \ar[d]^{\reg_{\ell}}
\\
\het^2(E_u, \Z_\ell(2)) \ar[r]^-{\pi_\ell} & H^1(\Q, \het^1(E_{u,\Qbar}, \Z_\ell(2)))_\tf
\,.
}
\end{split}
\end{equation}

Because~$ E_u $ has rational points, we can strengthen the above.
Note that below, the \'etale cohomology group~$ \het^2(\Spec(\Q), \Z_\ell(2)) = \varprojlim_m H^2(\Q , (\Z/\ell^m\Z) (2) ) $
plays a role, not the continuous Galois cohomology group~$ H^2(\Q, \Z_\ell(2) ) ) $.

\begin{prop} \label{Het2sos}
The pullback~$ \het^2(\Spec(\Q), \Z_{\ell}(2)) \to \het^2(E_u, \Z_{\ell}(2)) $
along the structure map~$ E_u \to \Q $ is injective. We have a short exact sequence
\begin{equation*}
0 \to \het^2(\Spec(\Q) , \Z_{\ell}(2)) \to \het^2(E_u, \Z_{\ell}(2)) \buildrel{\pi_{\ell}}\over{\to} H^1(\Q, \het^1(E_{u,\Qbar}, \Z_{\ell}(2)) ) \to 0
\,,
\end{equation*}
which can be split by pullback to any rational point of~$ E_u $.
\end{prop}

\begin{proof}
Pulling back~$  \het^2(E_u, \Z_{\ell}(2)) $ to any fixed~$ \Q $-rational point of~$ E_u $ shows that
the pullback along the structure map is injective. The same argument
also shows that in the spectral sequence~\eqref{HS-ss} all differentials with target~$ E_{s,m}^{p,0} $
for~$ s \ge 2 $ are trivial.
In particular, we have~$ E_{\infty,m}^{1,1} = E_{2,m}^{1,1} = H^1(\Q, \het^1(E_{u, \Qbar} , (\Z/\ell^m\Z)(2))) $
as well as~$ \Fil_m^2 = E_{\infty,m}^{2,0} = E_{2,m}^{2,0} = H^2(\Q, (\Z/\ell^m\Z)(2) ) $.
Taking inverse limits over the resulting short exact sequences~$ 0 \to \Fil_m^2 \to \Fil_m^1 \to E_{2,m}^{1,1} \to 0 $
is exact because pulling back to the fixed rational point provides
compatible splittings of the injections~$ \Fil_m^2 \to \Fil_m^0 $.
This results in the exact sequence in the proposition because we saw before that~$ \Fil^1 = \Fil^0 = \het^2(E_u, \Z_\ell(2)) $,
obtained a description of the inverse limit of the~$ E_{2,m}^{1,1} $,
and the Hochschild-Serre spectral sequences for~$ \Qbar/\Q $ provide compatible
isomorphisms~$ H^2(\Q, (\Z/\ell^m\Z)(2)) \simeq \het^2(\Spec(\Q), (\Z/\ell^m\Z)(2)) $.
\end{proof}

\section{Explicit elements of~\texorpdfstring{$ K_2^T(E_u) $}{K2TEu}} \label{K2elements}

Let~$ C $ be a regular, projective curve over a number field~$ k $,
and~$ \CCC $ a regular, flat and proper model of~$ C $ over
the ring of algebraic integers of~$ k $.
If~$ D $ is an irreducible curve in~$ \CCC $,
then by slight abuse of notation, we shall write~$ T_D $ for the tame
symbol associated with the generic point of~$ D $ in~$ \CCC $.
It is given by mapping~$ \symb f g $ in~$ K_2(k(C)) $ to~$ (-1)^{\ord_D(f) \ord_D(g)} \frac{f^{\ord_D(g)}}{g^{\ord_D(f)}} (D) $
in~$ \F(D)^\times $, with~$ \F(D) $ the function field of~$ D $.
If~$ D $ is not contained in a closed fibre of~$ \CCC $
then its generic point is a (closed) point~$ P $ of the generic
fibre~$ C $, $ \F(D) = k(P) $, and~$ T_D $ coincides with~$ T_P $
as defined right after~\eqref{K2locseq}.
So if we let
\begin{equation*}
\INT C = \cap_D (\ker(T_D))
\end{equation*}
then we have~$ \INT C \subseteq K_2^T(C) $.
We recall from~\cite[Proposition~4.1]{Li-dJ} and its proof
that this subgroup of~$ K_2^T(C) $ is independent
of the choice of~$ \CCC $, and, in fact, is
the image of~$ K_2(\CCC) $ in $ K_2(k(C)) $ under localisation from~$ \CCC $ to its generic point
(just as~$ K_2^T(C) $ is the image of~$ K_2(C) $ in~$ K_2(k(C)) $
under the localisation in~\eqref{K2locseq}).

For~$ E_u $ as in~\eqref{Eueq},
we set~$ v = \frac{x+u^2}{y} $,
$ w =  \frac{u - x v}{u+xv} = \frac{uy - x (x+u^2)}{uy+x(x+u^2)} = \frac{u(x+1) - y}{u(x+1)+y}$,
and~$ h = \frac{u(x+1) + y}{x+u} $.
Then~$ x = \frac u v \frac{1-w}{1+w} $
and~$ y = \frac{x+u^2}{v} $, so~$ \Q(v,w) = \Q(x,y) $.
From
\begin{alignat*}{1}
v - v^{-1} & = \frac{x+u^2}{y} - \frac{y}{x+u^2}
=
\frac{(x+u^2) - x (x+1)}{ y}
=
\frac{ u^2 - x^2 }{ y}~~\text{and}
\\
w - w^{-1} & =
\frac{u(x+1) - y}{u(x+1)+y} - \frac{u(x+1)+y}{u(x+1) - y}
=
\frac{-4uy} {u^2(x+1) - x (x+u^2)}
=
\frac{-4uy} {u^2 - x^2}
\end{alignat*}
we see that the (for~$ u \ne 0, \pm1 $ irreducible) polynomial 
\begin{equation} \label{VWZeq}
(V^2 - Z^2) (W^2- Z^2) + 4 u V W Z^2 = 0
\end{equation}
defines a (singular) model of~$ E_u $ in~$ \PP_\Q^2 $ with homogeneous
coordinates~$ [V, W, Z] $, where~$ v = V/Z $ and~$ w = W/Z $.

We can now define the element~$ \alpha_u $ mentioned in the
introduction. In order to simplify the notation, we suppress
the subscript~$ u $.

\begin{prop} \label{tameINT}
Let~$ E_u $ be as in~\eqref{Eueq}.

\begin{enumerate}
\item
The elements~$ \symb -1 x $, $ \symb -1 x+1 $
and~$ \alpha = \symb v w + \symb -1 h $
of~$ K_2(\Q(E_u)) $ are in~$ K_2^T(E_u) $.

\item
If $ 4 u $ is an integer then $ 2 \alpha $ is in~$ \INT E_u $,
and if~$ 4 u $ is not an integer, then $ m \alpha $ is not in~$ \INT E_u $ for
any integer~$ m \ne 0 $.
\end{enumerate}
\end{prop}

\begin{proof}
(1)
Because~$ (x) = 2 [(0,0)] - 2 [O] $ and~$ (x+1) = 2 [(-1,0)] - 2 [O] $,
it is clear that~$ \symb -1 x $ and~$ \symb -1 x+1 $ are in~$ K_2^T(E_u) $.

In order to show the same for~$ \alpha $, note that we have the divisors
\begin{alignat*}{1}
(v) & =  [(-u^2, 0)]  - [(0,0)] - [(-1,0)] + [O]
\\
(y+u(x+1)) & = [(-u,u^2-u )] + [(-1,0)] + [(u, - u^2-u)] - 3 [O]
\\
(y-u(x+1)) & = [(-u,-u^2+u )] + [(-1,0)] + [(u, u^2+u)] - 3 [O]
\\
(x+u) & = [(-u,- u^2 + u )] + [(-u, u^2 - u )] - 2 [O]
\,,
\end{alignat*}
where points with different notation are distinct because
of our assumptions on~$ u $.
Then
\begin{alignat*}{1}
(w) & = [(-u,-u^2+u )] + [(u, u^2+u)] - [(-u,u^2-u )]  - [(u, - u^2-u)]
\\
(h) & = [(-1,0)] + [(u, - u^2-u)] - [(-u,- u^2 + u )] - [O]
\,.
\end{alignat*}
Recall from just after~\eqref{K2locseq} that,
for a point~$ P $ on the curve,
the~$ P $-component~$ T_P $ of the tame symbol~$ T $
maps~$ \symb f g $ in~$ K_2(\Q(E_u)) $ to~$ (-1)^{\ord_P(f) \ord_P(g)} \frac{f^{\ord_P(g)}}{g^{\ord_P(f)}} (P) $.
We compute this for~$ P $ occurring in the divisors
of~$ v $, $ w $ and~$ h $.
The divisors of~$ v $ and~$ w $ have no common points, so that~$ T_P(\symb v w ) $
is easy to calculate, as is~$ T_P(\symb -1 h ) $.

If~$ P = (-u^2,0) $ or~$ (0,0) $, then~$ w(P) = 1 $, so that~$ T_P(\alpha) = T_P(\symb v w ) = 1 $.
If~$ P = (-1,0) $ then~$ T_P(\symb v w ) = w(P)^{-1} = -1 = T_P(\symb -1 h ) $, and~$ T_P(\alpha) = 1 $.
If~$ P = (-u, -u^2+u) $ or~$ (u, -u^2-u) $ then~$ v(P) = -1 $,
from which we obtain that~$ T_P(\symb v w ) = -1 = T_P(\symb -1 h ) $, and~$ T_P(\alpha) = 1 $.
If~$ P = (u, u^2+u) $ or~$ (-u,u^2-u) $, then~$ v(P) = 1 $, and~$ T_P(\alpha) = T_P(\symb v w ) = 1 $.
Finally, we have~$ T_O(\alpha) = 1 $ by the product formula~\eqref{product-formula}.
(Alternatively, writing~$ w = \frac{u(x/y+1/y)-1}{u(x/y+1/y)+1} $
we see that~$ w(O) = -1 $, so that~$ T_O(\symb v w ) = - 1 = T_O(\symb -1 h ) $.) 

(2)
First we assume that~$ 4 u $ is an integer.
In order to verify that $ 2 \alpha = 2 \symb v w $ is in~$ \INT E_u $,
we follow the set-up of the proofs of \cite[Theorem~8.3(1)]{DdJZ}
and \cite[Theorem~4.2]{Li-dJ}.
Because of our assumption, \eqref{VWZeq} defines a subscheme~$ \DDD $ of~$ \PP_\Z^2 $.
Applying~\cite[Corollary~8.3.51]{Liu} to the normalisation of~$ \DDD $,
we see that there exists a regular, flat and proper model~$ \EE / \Z $ of~$ E_u / \Q $
with a morphism~$ \EE \to \DDD $ that on the generic
fibres is the normalisation map.
For an irreducible curve~$ C $ in~$ \EE $, we now show~$ T_C(2\alpha) = 1 $
by considering three cases.

If~$ C $ meets the generic
fibre~$ E_u $ then we know that~$ T_C(\alpha) = 1 $ because
this tame symbol corresponds to that of the generic point of~$ C $,
which is a closed point of~$ E_u $.
So we may assume that~$ C $ does not meet the generic fibre, hence
is contained in a fibre~$ \EE_p $ of~$ \EE $ above a prime number~$ p $.

If the image of~$ C $ in~$ \DDD $ is a curve~$ C' $, then~$ C' $
is in~$ \DDD_p $. But the poles and zeroes of~$ v = V/Z $ and~$ w = W/Z $
on~$ \DDD $ are contained in its subscheme~$ \YYY $ defined by~$ VWZ = 0 $ in~\eqref{VWZeq},
which meets~$ \DDD_p \subset \PP_{\F_p}^2 $ in the points~$ [0, \pm 1, 1] $, $ [\pm 1, 0, 1] $, $ [1,0,0] $, 
or~$ [0,1,0] $. Hence~$ v $ and~$ w $ are generically defined on~$ C' $
and~$ C $, which implies that~$ T_C(2\alpha) = 1 $.

Finally, suppose that the image of~$ C' $ in~$ \DDD $ is a point~$ P $ in~$ \DDD_p $.
If~$ P $ is not in~$ \YYY $ then both~$ v = V/Z $ and~$ w = W/Z $ are regular with non-zero
values at~$ P $, so both are constant as functions on~$ C $,
hence~$ T_C(2\alpha) = 1 $.
If~$ P = [0, \pm 1, 1] $ then we write~$ 2 \alpha = \symb V/Z (W/Z)^2 $
and note that~$ (W/Z)^2 $ is regular with value~1 at~$ P $.
This implies that~$ T_C(2 \alpha) = 1 $ because~$ (W/Z)^2 $ on~$ \EE $
is a function that restricts to~1 on~$ C $.
For~$ P = [1,0,0] $ we see from~\eqref{VWZeq} that~$ (1-Z^2/V^2) ( (W/Z)^2 - 1 )= 4 u W/V $
on~$ \DDD $, and because~$ 4 u $ is an integer this implies that~$ (W/Z)^2 $ is regular at~$ P $ with value~1,
so that~$ T_C(2 \alpha) = 1 $ also in this case.
If~$ P  = [\pm 1, 0, 1] $ or~$ [0,1,0] $ then~$ T_C(2 \alpha) = 1 $
by symmetry.

If~$ 4 u $ has denominator~$ d > 1 $ then we multiply~\eqref{VWZeq} by~$ d $,
and the resulting equation also defines a subscheme~$ \DDD $ of~$ \PP_\Z^2 $.
There exists again a regular, flat and proper model~$ \EE / \Z $ of~$ E_u / \Q $
with a morphism~$ \EE \to \DDD $ that on the generic
fibres is the normalisation map. Let~$ p $ be a prime number
dividing~$ d $. Because the fibre~$ \DDD_p $ of~$ \DDD $ above~$ p $
is defined by~$ V W Z^2 = 0 $ in~$ \PPP_{\F_p}^2 $, we see that,
if we let~$ \eta $ in~$ \DDD $ be the generic point of the component~$ C' $ defined
by~$ V = 0 $ in~$ \DDD_p $, then the local ring~$ \OOO_{\DDD,\eta} $
is normal, hence a discrete valuation ring.
Therefore the corresponding tame symbol~$ T_{C'} $ is defined,
and~$ T_{C'}{ \symb v w } = Z/W $,
which has infinite order in~$ \F_p(C')^\times $.
There is a irreducible component~$ C $ of~$ \EE_p $ surjecting
onto~$ C' $, and then~$ T_C( \symb v w ) $ has infinite order in~$ \F_p(C)^\times $ as
well.
\end{proof}

\begin{remar} \label{tamenotINTrem}
(1)
On~$ E_u = E_{-u} $ we have the identities~$ v_{-u} = v_u $, $ w_u w_{-u} = 1 $,
as well as~$ h_u h_{-u} = x+1 $, so that~$ \alpha_u + \alpha_{-u} = \symb -1 x+1 $.

(2)
In Theorem~\ref{new-index-theorem} we shall prove a statement
that implies~$ \alpha $ is not in~$ \INT E_u $ if~$ 4 u $ is an integer.
\end{remar}

Our next goal in this section is to prove certain indivisibility
statements (see Theorem~\ref{new-index-theorem}). These statements
are modulo torsion, and in their proofs we control such torsion by
means of the commutative diagram~\eqref{regCD} and Proposition~\ref{H1G-torsion},
which describes the torsion of the group in the top right corner
of the diagram.

For that diagram we shall use the pullback of elements from $ K_2(E_u) $
or $ K_2^T(E_u) $ to (closed) points. We shall discuss this for~$ E_u $, but the method works for any regular curve over a field.

Fix a closed point~$ R $ of~$ E_u $. Then the pullback~$ K_2(E_u) \to K_2(\Q(R)) $
factorises through the localisation~$ K_2(E_u) \to K_2( \OOO_{E_u, R} ) $. Because~$ E_u $
is regular, 
the natural map~$  K_2( \OOO_{E_u, R} ) \to \ker(T_R) $, with
the latter in~$ K_2(\Q(E_u)) $, is an isomorphism
(see (V.9.3.1), Theorem~V.9.6 and Conjecture~V.9.3 in\cite{Kbook}).

In order to make the pullback to~$ K_2(\Q(R)) $ explicit on an element of~$ \ker(T_R) $,
we first note that every such element can be written
as a sum of symbols~$ \symb f_1 f_2 $ with~$ f_1 $ and~$ f_2 $ regular at~$ R $
with non-zero values,
and a term~$ \symb {\pi} g $ with~$ \pi $ a uniformiser at~$ R $,
and $ g $ regular at~$ R $ with~$ g(R) = 1 $.
Using~$ \symb {\pi} g = \symb {\pi} (1-\pi)g $
if necessary, we may assume~$ g = 1 - \tilde g \pi  $ with~$ \tilde g(R) \ne 0 $.
Then~$ \symb {\pi} g = \symb (1-g)^{-1}{\pi} g = \{ g, \tilde g \} $,
in which both~$ g $ and~$ \tilde g $ are regular at~$ R $ with
non-zero values. 

We therefore only have to consider elements~$ \symb f_1 f_2 $ with~$ f_1 $
and~$ f_2 $ regular at~$ R $ with non-zero values.
Because $ K_1(\OOO_{E_u, R}) = \OOO_{E_u, R}^\times $
 (see \cite[Lemma~III.1.4]{Kbook}) and pullback is compatible with
cup products, the pullback of~$ \symb f_1 f_2 $ is~$ \symb f_1(R) f_2(R) $,
as was to be expected. As a useful tool, we observe that our calculation above
shows that if~$ f_1 $ and~$ f_2 $ are in~$ \Q(E_u)^\times $, and~$ f_2 $ is regular
at~$ R $ with~$ f_2(R) = 1 $, then this element pulls back to~0
in $ K_2(\Q(R)) $.

In particular, this gives an explicit pullback~$ i_R^* : K_2^T(E_u) \to K_2(\Q(R)) $
that induces the pullback~$ K_2(E_u) \to K_2(\Q(R)) $. We compute
it for the elements of~$ K_2^T(E_u) $ in Proposition~\ref{tameINT}(1), with~$ R $
equal to~$ P = (0,0) $ or~$ Q = (-u^2, 0) $.
From~$ h(P) = w(P) = 1 $ we obtain~$ i_P^*( \alpha) = 0 $.
Similarly, with~$ h(Q) = u+1 $ and~$ w(Q) = 1 $, we get~$ i_Q^*(\alpha) = \symb -1 1+u $.
The pullbacks of~$ \symb -1 x $ and~$\symb -1 x+1 $ are easily computed, using
\begin{equation*}
 \symb -1 x = \symb -1 (x+1)^{-1}(x+u^2)^{-1}y^2 = \symb -1 (x+1)^{-1}(x+u^2)^{-1} 
\end{equation*}
to mitigate the problem that $x(P)=0$.
All pullbacks are listed in Table~\ref{pullback-table}.

\begin{table}[t]
\caption{\label{pullback-table}Pullbacks of elements of~$ K_2^T(E_u) $ at some rational points.}
\begin{tabular}{|c|c|c|}
\hline
\vphantom{$ p_p{p_p} $}\vphantom{$ b^{b^b} $}
 & $ P = (0,0) $ & $ Q = (-u^2, 0) $ 
\\
\hline\vphantom{$ b^{b^b} $}
$ \symb -1 x $ & $ 0 $ & $ \symb -1 -1 $ 
\\
$ \symb -1 x+1 $ & $ 0 $ & $ \symb -1 1-u^2 $
\\
\vphantom{$ p_{p_{p_p}} $}
$ \alpha = \symb v w + \symb  -1 h $ & 0 & $ \symb -1 1+u $
\\
\hline
\end{tabular}
\end{table}

We now combine these pullback maps with the map~$ \ch_{\ell} $ in~\eqref{K2TregCD}
for~$ \ell = 2 $.
By the construction and functoriality of~$ \ch_2 $, 
and because the pullbacks factor through the surjection~$ K_2(E_u) \to K_2^T(E_u) $,
for every~$ m \ge 1 $ we have a commutative diagram
\begin{equation}
\begin{split}
\label{regCD}
\xymatrix{
K_2(E_u) \ar[r]^-{\mathrm{ch}_2} \ar[d] & \het^2(E_u, \Z_2(2)) \ar[r]^-{\pi_2} \ar[dd]^-{i_Q^* - i_P^*} & H^1(\Q, \het^1(E_{u, \Qbar}, \Z_2(2))) \ar@{.>}[ldd]
\\
K_2^T(E_u) \ar[d]^-{i_Q^* - i_P^*} & &
\\
K_2(\Q) \ar[d] \ar[r]^-{\mathrm{ch}_2} & \het^2(\Spec(\Q), \Z_2(2)) \ar[d]
\\
K_2(\Q)/2^m \ar[r]^-{\mathrm{ch}_{2,m}}_-{\simeq } & \het^2(\Spec(\Q),  \mu_{2^m}^{\otimes 2})
\,.
}
\end{split}
\end{equation}
Here~$ \pi_2 $ is as in~\eqref{K2TregCD},
and the map in the lowest row is an isomorphism by~\cite[Theorem~11.5]{MS}.
The first map in the second column annihilates the
first term in the exact sequence in Proposition~\ref{Het2sos}
because that was pulled back from~$ \Q $, so induces
the indicated map~$ H^1(\Q, \het^1(E_{u, \Qbar}, \Z_2(2))) \to \het^2(\Spec(\Q), \Z_2(2)) $.
Composing this map with~$ \het^2(\Spec(\Q), \Z_2(2)) \to \het^2(\Spec(\Q),  \mu_{2^m}^{\otimes 2}) \simeq K_2(\Q)/2^m $
gives a~map
\begin{equation} \label{psim-def}
\psi_m : H^1(\Q, \het^1(E_{u, \Qbar}, \Z_2(2))) \to K_2(\Q)/2^m
.
\end{equation}

For the last ingredient of our proofs, we recall the structure
of~$ K_2(\Q) $ (see \cite[III.6.5.1]{Kbook}). We have~$ K_2(\Q) \simeq \{\pm 1\} \times \oplus_p \F_p^\times $
where~$ p $ runs through the prime numbers, and the map is~$ T_\infty \times \prod_p T_p $,
with~$ T_p $ the tame symbol for~$ p $ defined by~$ T_p(\symb a b ) = (-1)^{v_p(a) v_p(b)} a^{v_p(b)} b^{-v_p(a)} $
modulo~$ p $ in~$ \F_p^\times $, and~$ T_\infty $ defined
by~$ T_\infty( \symb a b ) = -1 $ if~$ a, b < 0 $ and~1 otherwise.
Here the kernel of~$ \prod_{p<\infty} T_p $ is $ K_2(\Z) = \{ 0, \symb -1 -1 \} $,
which has order~2.

\begin{prop} \label{H1G-torsion}
Let~$ E_u $ be as in~\eqref{Eueq}, with~$ |u^2-1| $ not a square in~$ \Q^\times $.

(1)
Then~$ H^1(\Q, \het^1(E_{u, \Qbar}, \Z_2(2)))_\tor $ is non-cyclic of
order~4.

(2)
Assume that also~$ 2 |u^2-1| $ is not a square in~$ \Q^\times $.
If we lift the elements~$ \symb -1 x $ and~$ \symb -1 x+1 $
of Proposition~\ref{tameINT}(1)
to~$ K_2(E_u) $, then the images~$ t_1 $ and~$ t_2 $ of these lifts under the composition
\begin{equation*}
K_2(E_u) \buildrel{\ch_2} \over{\longrightarrow} \het^2(E_u, \Z_2(2)) \buildrel{\pi_2}\over{\longrightarrow} H^1(\Q, \het^1(E_{u, \Qbar} , \Z_2(2) ) ) 
,
\end{equation*}
with~$ \ch_2 $ and~$ \pi_2 $ as in~\eqref{regCD},
are independent of the chosen lifts, and form an~$ \F_2 $-basis
of~$ H^1(\Q, \het^1(E_{u,\Qbar}, \Z_2(2)))_\tor $.
Moreover, these images~$ t_1 $ and~$ t_2 $ are mapped by~$ \psi_m $ to the classes of~$ \symb -1 -1 $ and~$ \symb -1 1-u^2 $
in~$ K_2(\Q) / 2^m $, for any~$ m \ge 1 $.
\end{prop}

\begin{proof}
(1)
For each $m>0$ we have an exact sequence 
$$0\rightarrow \het^1(E_{u,\Qbar}, \Z_2(2))\xrightarrow{2^m} \het^1(E_{u,\Qbar}, \Z_2(2))\rightarrow \het^1(E_{u,\Qbar}, (\Z/2^m\Z)(2))\rightarrow 0.$$
From the resulting long exact sequence in Galois cohomology, and using the fact that $H^0(\Q, \het^1(E_{u,\Qbar}, \Z_2(2)))$ is trivial, we see that~$H^1(\Q,\het^1(E_{u,\Qbar}, \Z_2(2)))[2^m] $ is isomorphic to $H^0(\Q, \het^1(E_{u,\Qbar}, (\Z/2^m\Z)(2)))\simeq H^0(\Q, E_u[2^m](1))$,
where we used the natural isomorphism~$ \het^1(E_{u,\Qbar}, \mu_{2^m} ) \simeq E_u(\Qbar)[2^m] $
of \cite[Corollary~III.4.18]{Mil80}.
The statement now follows from Proposition~\ref{2tors}.

(2)
With~$ P = (0,0) $ and~$ Q = (-u^2, 0) $ as before, using
Table~\ref{pullback-table}, in~\eqref{regCD} we find~$ (i_Q^* - i_P^*) (\symb -1 x ) = \symb -1 -1 $
and~$ (i_Q^* - i_P^*) (\symb -1 x+1 ) = \symb -1 1-u^2 $ in~$ K_2(\Q) $.
These two elements generate a non-cyclic subgroup of order~4 because~$ T_\infty(\symb -1 -1 ) = -1 $,
$ T_p(\symb -1 -1 ) = 1 $, and $ T_p(\symb -1 1-u^2 ) = -1 $
for an odd prime number~$ p $ with~$ \ord_p(u^2-1) $
odd, which exists by our assumptions.

For every~$ m \ge 1 $, the class of $ \symb -1 -1 $ in $ K_2(\Q) / 2^m $
is non-trivial. The classes of $ \symb -1 -1 $ and~$ \symb -1 1-u^2 $
generate a non-cyclic group of order~4 in $ K_2(\Q) / 2^m $ for
any~$ m \ge \ord_2(p-1) $ with~$ p $ as before,
because~$ -1 $ is not a~$ 2^m $th power in~$ \F_p^\times $.

Now let~$ \symb -1 x ^\sim $ and~$ \symb -1 x+1 ^\sim $ in~$ K_2(E_u) $
lift~$ \symb -1 x $ and~$ \symb -1 x+1 $, respectively.
As noted around~\eqref{K2locseq}, the kernel of~$ K_2(E_u) \to K_2^T(E_u) $ is torsion
so these lifts are in~$ K_2(E_u)_{\tor} $, and have
images~$ t_1 $ and~$ t_2 $ in~$ H^1(\Q, \het^1(E_{u, \Qbar}, \Z_2(2)))_{\tor} $ under~$ \pi_2 \circ \ch_2 $.
Applying the map~$ \psi_m $ of~\eqref{psim-def}, with~$ m $ as before, to~$ t_1 $
and~$ t_2 $, we see from~(1), the commutativity
of~\eqref{regCD}, and our calculations above, that~$ t_1 $ and~$ t_2 $
form an~$ \F_2 $-basis of~$ H^1(\Q, \het^1(E_{u, \Qbar}, \Z_2(2)))_{\tor} $,
and that~$ \psi_m $ induces an isomorphism between this subgroup
and the subgroup of~$ K_2(\Q) / 2^m $ generated by the classes of~$ \symb -1 -1 $ and $ \symb -1 1-u^2 $.
In particular, the images~$ t_1 $ and~$ t_2 $
are independent of the choice of the lifts.
\end{proof}

The next result is the main goal of this section. It concerns certain
indivisibilities in $ K $-groups or in the target of the map~$ \reg_2 $ in~\eqref{K2TregCD}.

\begin{thm} \label{new-index-theorem}
Let~$ E_u $ and~$ \alpha $ be as in Proposition~\ref{tameINT}(1),
but assume~$ \ord_p( 1-u ) $ is odd for some odd prime number~$ p $,
and similarly for~$ 1+u $ and~$ 1-u^2 $ in place of~$ 1-u $ (with
possibly different~$ p $). Then the following hold.

(1)
If~$ m \ge \ord_2(p-1) $ for some odd prime number~$ p $ with~$ \ord_p(u-1) $
odd, and also~$ m \ge \ord_2(q-1) $ for some odd prime number~$ q $ with~$ \ord_q(u+1) $
odd, then~$ \reg_2(\alpha) $ is not divisible by~$ 2^m $
in~$ H^1(\Q, \het^1(E_{u, \Qbar} , \Z_2(2)))_\tf $,
and the class of~$ \alpha $ in~$ K_2^T(E_u)_\tf $ cannot be divided by~$ 2^m $.

(2)
Assume also that~$  4 u  $ is an integer, so that~$ 2 \alpha $
is in~$ \INT E_u $ by Proposition~\ref{tameINT}(2).
Then~$ 2 \alpha $ is not in~$ 2 \INT E_u + \INT E_u {}_{, \tor} \subseteq \INT E_u $.
\end{thm}

\begin{cor} \label{new-index-cor}
Suppose that in Theorem~\ref{new-index-theorem} the conditions are satisfied
for a positive integer~$ u $ that is divisible by~4.
Let~$ m_u $ be the minimal value
of~$ \ord_2(q-1) $ where~$ q $ runs through the prime divisors of~$ u+1 $.
Then~$ \reg_2(\alpha) $ is not divisible by~$ 2^{m_u} $ in~$ H^1(\Q, \het^1(E_{u, \Qbar} , \Z_2(2)))_\tf $,
hence the class of~$ \alpha $ is not divisible by~$ 2^{m_u} $ in~$ K_2^T(E_u)_\tf $.
\end{cor}

The corollary follows immediately from Theorem~\ref{new-index-theorem}(1)
by noting that there must be a prime number~$ p \equiv 3$ modulo~4
with~$ \ord_p(u-1) $ odd, and~$\ord_2(p-1) = 1$.

\begin{remar} \label{m=1or2}
If~$ u $ is a positive integer congruent to~4 modulo~8,
then~$  u- 1 $, $ u+1 $ and~$ u^2-1 $ are not congruent
to~1 modulo~8 so they cannot be squares, and
the conditions of the theorem and corollary are always satisfied.
Then~$ m_u = 1 $ or~2 because~$ u+1 $ is not a product of prime numbers that
are congruent to~1 modulo~8.
\end{remar}

\begin{proof}[Proof of Theorem~\ref{new-index-theorem}]
(1)
Suppose~$ \reg_2(\alpha) $ is divisible by~$ 2^m $ for some~$ m \ge 1 $.
From~\eqref{K2TregCD} we then have in~\eqref{regCD}, with~$ \tilde\alpha $ any lift of~$ \alpha $ to~$ K_2(E_u) $,
that~$ \pi_2 \circ \ch_2(\tilde\alpha) = 2^m s + t $ for some~$ s $ and~$ t $
in~$ H^1(\Q, \het^1(E_{u, \Qbar} , \Z_2(2))) $ with~$ t $ torsion.
Applying~$ \psi_m $ as in~\eqref{psim-def}, and using Proposition~\ref{H1G-torsion}(2)
together with Table~\ref{pullback-table}, we see from the commutativity
of~\eqref{regCD} that~$ \symb -1 1+u $ is in~$ \langle \symb -1 -1 , \symb -1 1-u^2 \rangle + 2^m K_2(\Q) $
inside~$ K_2(\Q) $.
Equivalently (as one sees by applying $T_{\infty}$),~$ \symb -1 |1+u| $ is
in $ \langle  \symb -1 |1-u^2| \rangle + 2^m K_2(\Q) $,
so that~$ \symb -1 |1-u| $ or~$ \symb -1 |1+u| $ is in~$ 2^m K_2(\Q) $.
Now take~$ m $ as in the statement of the theorem.
Then~$ T_p( \symb -1 |1-u| ) = -1 $ in~$ \F_p^\times $ is not a~$ 2^m $th
power, nor is~$ T_q( \symb -1 |1+u| ) = -1 $ in~$ \F_q^\times $, so both are impossible.
That the class of~$ \alpha $ in~$ K_2^T(E_u)_\tf $ is not divisible by~$ 2^m $
now follows by applying~$ \reg_2 $.

(2)
Suppose that $ 2 \alpha = 2 \beta +  \gamma $ with~$ \beta $
and~$ \gamma $ in~$ \INT E_u $ and~$ \gamma $ torsion. Then~$ \gamma $
is 2-divisible in~$ K_2^T(E_u) $, so in that group we have~$ \alpha = \beta + \gamma' $
with~$ \gamma' $ in~$ K_2^T(E_u)_{\tor} $.
For any~$ m \ge 1 $, in~\eqref{regCD} we know from Table~\ref{pullback-table} that
the image of~$ \alpha $ in~$ K_2(\Q) / 2^m $ is the class of~$ \symb -1 1+u $.
On the other hand, as we explained at the beginning of this section,
$ \beta $ comes from an element of
$ K_2 $ of a regular model over~$ \Z $ of~$ E_u $, so $ i_P^*(\beta) $
and~$ i_Q^* (\beta) $ are in~$ K_2(\Z) = \{ 0, \symb -1 -1 \} \subseteq K_2(\Q) $
because they come from the sections of the model corresponding
to~$ P $ and~$ Q $.

By the commutativity of~\eqref{regCD}, we can compute the class of~$ (i_P^*-i_Q^*)(\gamma') $ in~$ K_2(\Q) / 2^m $ by first lifting~$ \gamma' $
to~$ \tilde\gamma' $ in~$ K_2(E_u) $, and then computing the image of~$ \tilde\gamma' $ under~$ \psi_m \circ \pi_2 \circ \ch_2 $.
Here~$ \tilde\gamma' $ is a torsion element because
the kernel of the map~$ K_2(E_u) \to K_2^T(E_u) $ is torsion,
as noted around~\eqref{K2locseq}, so
we see from Proposition~\ref{H1G-torsion}(2) and Table~\ref{pullback-table}
that its image lies in the subgroup of~$ K_2(\Q) / 2^m $ generated by the classes of~$ \symb -1 -1 $
and~$ \symb -1 1-u^2 $.
We then have in~$ K_2(\Q) $ that~$ \symb -1 1+u $ 
is in $ \langle \symb -1 -1 , \symb -1 1-u^2 \rangle + 2^m K_2(\Q) $,
which, as in the proof of~(1), is equivalent to~$ \symb -1 |1+u| $
or~$ \symb -1 |1-u| $ being in~$ 2^m K_2(\Q) $.
Both are impossible as in the proof of~(1), because there must
be an odd prime number~$ p $ with~$ \ord_p(1-u) $ odd,
and an odd prime number~$ q $ with~$ \ord_p(1+u) $ odd,
and for~$ m $ sufficiently large, $ T_p( \symb -1 |1-u|) = -1 $
is not a~$ 2^m $th power in~$ \F_p^\times $, nor is~$ T_q( \symb -1 |1+u|) = -1 $
a~$ 2^m $th power in~$ \F_q^\times $.
\end{proof}

\begin{remar}
(1) In particular, in Proposition~\ref{new-index-theorem}(1)~$ \reg_2(\alpha) $ is non-trivial.

(2)
Clearly~$ 2\alpha $ is always divisible by~2 in $ K_2^T(E_u) $ but its class
in~$ \INT E_u {}_{,\tf} $ is not divisible by~2 in that group under
the conditions of Proposition~\ref{new-index-theorem}(2).
This difference does not appear for
more classical $ K $-groups; for example, the notions of 2-divisibility in the units of
the ring of integers of a number field, and in the units of the number field itself, coincide.
\end{remar}

\section{A hypergeometric formula for the regulator} \label{hyper-sec}

Boyd \cite{Bo} investigated numerically the relations between Mahler measures of $2$-variable Laurent polynomials associated with reflexive lattice polygons, and leading terms at $s=0$ of derivatives of associated $L$-functions. His work was revisited by Rodriguez Villegas \cite{V}, in the context of $K_2$ of elliptic curves and Beilinson's conjecture.  One of Boyd's examples \cite[\S 9, Example (b)]{V} is the family of Laurent polynomials $X+Y+X^{-1}+Y^{-1}-k$.
For~$ k $ in~$ \Q $ with~$ k \ne 0, \pm 4 $ this (set equal to $0$) defines the function field of an elliptic curve $C_k$
over~$ \Q $.
We now relate it to the family of elliptic curves in~\eqref{Eueq}.

\begin{prop} \label{boyd-isogeny}
On the curve $E_u$ as in~\eqref{Eueq}, let~$ X = -v w $ and $ Y = v / w $, where~$ v = V/Z $
and~$ w = W/Z $ are as around~\eqref{VWZeq}.

(1)
Then~$X+Y+X^{-1}+Y^{-1}=4u$, the subfield~$ \Q(X,Y) \subseteq \Q(E_u) $
is the subfield of~$ \Q(E_u) = \Q(v,w) $ invariant under translation by~$ (-u^2, 0) $,
corresponding to an isogeny~$ \phi : E_u \to C_{4u} $ with kernel generated
by~$ (-u^2, 0) $.

(2)
We have that~$ \symb X Y $ is in~$ K_2^T(C_{4u}) $, and under~$ \phi $
pulls back to~$ - 2 \alpha $ in~$ K_2^T(E_u) $.
\end{prop}

\begin{proof}
(1)
We can rewrite~\eqref{VWZeq} as~$ - (v-1/v) (w-1/w) = 4u $, and
multiplying out gives the stated identity.
From the divisors computed in the proof of Proposition~\ref{tameINT}(2)
one sees that the divisors of~$ v $ and~$ w $ are invariant under
translation by the point~$ (-u^2, 0) $ of order~2, hence translating~$ v $
and~$ w $ gives at most a sign change.
By the rewritten version of~\eqref{VWZeq} above, the signs for~$ v $ and~$ w $ are the same because~$ u \ne 0 $,
hence~$ X $ and~$ Y $ are invariant under the translation. Using
that the
translation does not give the identity on~$ \Q(E_u) $,
that~$ \Q(E_u) = \Q(v,w) $ as seen just before~\eqref{VWZeq},
as well as that~$ \Q(v,w) / \Q(X,Y) $ is of degree at most~2 because~$ v^2 = - X Y $
and~$ w = - v^{-1} X $, the result follows.

(2)
We have in~$ K_2^T(E_u) $ that~$ \phi^*(\symb X Y ) = \symb -vw v/w  $ is equal
to
\begin{equation*}
\symb -vw v/w  + \symb -v/w v/w = \symb v^2 v/w = 2 \symb v 1/w + 2 \symb -v v = - 2 \alpha 
.
\end{equation*}
It follows from this that~$ \symb X Y $ is in~$ K_2^T(C_{4u}) $ because~$ \phi : E_u \to C_{4u} $ is unramified
and~$ \Q(\phi(P))^\times \to \Q(P)^\times $ for~$ P $ in~$ E_u $ is injective.
\end{proof}

\begin{remar}
(1)
Substituting $x'=x+u^2$ to translate $(-u^2,0)$ to $(0,0)$, then using standard formulas as in \cite[X.4.8]{S1} to obtain the $2$-isogenous curve, we find that a Weierstrass form for $C_{4u}$ is
$$(Y')^2=(X')^3+2(2u^2-1)(X')^2+X'.$$
If~$u$ is an integer with $4\mid\mid u$, then after a further change of variables to achieve minimality at $2$,
the minimal discriminant is $(u/4)^2(u^2-1)$, the square root of that of $E_u$, and the reduction at bad primes is exactly as for $E_u$ (cf. Proposition~\ref{minimal}(4)). But in contrast to Proposition \ref{torsion}, the only rational point of order $2$ is $(0,0)$, the other points of order $2$ being $(-(2u^2-1)\pm 2u\sqrt{u^2-1},0)$.

(2)
Arguing along the lines of the proof of Proposition~\ref{tameINT}(2),
one can show that~$ \symb X Y $ is in~$ \INT C_k $ if~$ k \ne 0, \pm 4 $ is an integer,
but the argument is more involved.

(3)
Because the kernel of the pullback~$ K_2(\Q(C_{4u})) \to K_2(\Q(E_u)) $
is 2-torsion by the projection formula
(see~\cite[Corollary~V~3.~7.32]{Kbook})
it follows from Theorems~1.8 and~3.5 of~\cite{sus89} that this map
is, in fact, injective.
\end{remar}

In order to make the 2-part of the Bloch-Kato conjecture 
explicit in Section~\ref{2-BK}, we need to determine~$ \reg(\alpha) $
in~$ (2\pi i) H^1_B(E_u(\C), \R)^- $
(see~\eqref{regdef} with~$ m+1 = 2 = r $ and~\eqref{reg22target}).
For this, we want to compute~$ 2\pi i\int_{\gamma^-}\reg(2 \alpha)$,
where~$ \gamma^- $ is a generator of~$ H_1(E_u(\C), \Z)^- $.
We want to exploit that~$ -2 \alpha = \phi^*(\symb X Y ) $ with~$ \symb X Y $
in~$ K_2^T(C_{4u}) $.

\begin{lem} \label{reg-lemma}
Assume $u>1$. If~$ \gamma_C^- $ is a generator of~$H_1(C_{4u}(\C), \Z)^-$,
then, up to sign, we have
\begin{equation*} 2\pi i\int_{\gamma^-}\reg(\alpha) = 2\pi i \int_{\gamma^-_{C}}\reg(\symb X Y )
.
\end{equation*}
\end{lem}

\begin{proof}
Integrating~$ \frac{dx}{y} $ over the loop~$ \gamma^- $ in~$ E_u(\C) $ where~$ x $ is real and~$ x \le -u^2 $ gives a purely imaginary period~$ \omega_1 $.
Similarly, integrating over the loop~$ \gamma^+ $ where~$ x $ is real and~$ x \ge 0 $ gives
a real period~$ \omega_2 $.
With~$ \Lambda = \Z \omega_1 + \Z \omega_2 $ in~$ \C $, we get an identification
of~$ E_u(\C) $ with~$ \C/\Lambda $, compatible
with complex conjugation.
Because~$ (-u^2,0) $ corresponds to~$ \frac 12 \omega_1 $,
$ C_{4u} $ has lattice~$ \Z \frac12\omega_1 + \Z \omega_2 $.
Then the line segment~$  0 \to \frac12\omega_1 $ gives a generator~$ \gamma_C^- $
of~$H_1(C_{4u}(\C), \Z)^-$, to which~$ \gamma^- $ maps 2-1.
The result follows because~$ \phi^*(\reg(\symb X Y )) = \reg(-2\alpha) $.
\end{proof}

The following proposition is essentially an instance of the theorem in \cite[\S 10]{V} (combined with the material in \cite[\S~11,12]{V} on expansions), but with $\gamma_C^+$ corrected to $\gamma_C^-$, and the missing details for the proof supplied here.

\begin{prop} \label{reg-prop}
Let~$ \gamma_C^- $ be a generator of~$ H_1(C_{4u}(\C), \Z)^- $
where~$ u $ is in~$ \Q $ and~$ u > 1 $.
Define\footnote{One can check that the series converges for all $|u|\geq 1$.} $F(u):=\log(4u)-\sum_{n=1}^{\infty}\left({2n\atop n}\right)^2\frac{(4u)^{-2n}}{2n}$.
Then, up to sign,
$$\frac{1}{2\pi i}\int_{\gamma^-_{C}}\reg(\symb X Y )=F(u).$$
\end{prop}

\begin{proof}
We abbreviate $C_{4u}$ to $C$ in this proof.
Through composition of the natural map~$ K_2^T(C) \to K_2^T(C) \otimes_\Z \Q = H^2_{\mathcal{M}}(C, \Q(2))$
with~\eqref{regdef} we get a regulator map
$$ \reg: K_2^T(C) \rightarrow H^2_{\mathcal{D}}(C_{\R}, \R(2))$$
(cf.~the map~$ \reg_{\ell} $ in~\eqref{K2TregCD}).
There is also an integral regulator map
(see~\cite{Bl3} and~\cite{KL})
$$R: \mathrm{CH}^2(C, 2)\rightarrow H^2_{\mathcal{D}}(C_{\C}, \ZZ(2))\simeq \frac{H^1_{\dR}(C_{\C})}{H^1_B(C(\C), \ZZ(2\pi i)^2)}.$$ Composing $R$ with the projection $\pi_{\RR}$ to $H^2_{\mathcal{D}}(C_{\CC},\RR(2))\simeq \frac{H^1_{\dR}(C_{\R})}{H^1_B(C(\C),\R(2\pi i)^2)}$, the image lies in the $+$ part.
Combining the identification~$ CH^2(F,2) = K^M_2(F) $ of \cite[Theorem~1]{Tot92},
the localisation sequences for (cubical) higher Chow groups (see \cite{Park23}), 
and the fact that~$ CH^1(F,2) = \{0\} $, with the compatibility of the tame symbol and the localisation map
(see \cite[p.25]{Ke}), we obtain an identification $K_2^T(C) = \mathrm{CH}^2(C, 2) $.
We then have a diagram
\begin{equation}
\begin{split}
\label{eq11.5}
\xymatrix{\mathrm{CH}^2(C,2) \ar [r]^{R\mspace{30mu}} \ar @{=} [d] & H^2_{\mathcal{D}}(C_{\CC},\ZZ(2)) \ar @{->>} [r]^{\pi_{\RR}} & H^2_{\mathcal{D}}(C_{\CC},\RR(2)) \\ K^T_2(C) \ar [rr]^-{\reg} && H^2_{\mathcal{D}}(C_{\RR},\RR(2)) , \ar @{^(->} [u]}
\end{split}
\end{equation}
which we shall see below to be commutative.

Let $\xi\in\mathrm{CH}^2(C,2)$ denote the element to which $\{X,Y\}$ extends.  To compute its image under $R$, we think
of $C$ as living inside the toric variety $\PP$ obtained by blowing up $\PP^1\times \PP^1$ at the~4 torus fixed points
(and meeting the 4 exceptional divisors each with multiplicity~1).  Writing 
$$\varphi(X,Y):=X+X^{-1}+Y+Y^{-1}$$ 
and $U:=\PP\setminus C$, the symbol $\{4u-\varphi,X,Y\}\in K_3^M(\QQ(X,Y))$ extends to a cycle $\Xi$  in~$ \CH^3(U,3)$ with $\mathrm{Res}_C(\Xi)=\xi$.
In order to see this, again using the triviality of $ CH^1(F,2) $ for
any field~$ F $, it is enough to check the vanishing of its residues
in $ K_2^M $ of the function fields of the
curves in~$ U $. These are obviously trivial except for two types
of curves.  The first is represented by $ X=0 $, with function field $ \Q(Y) $. 
Here we rewrite the symbol as~$ \{-4uX+X\varphi,X,Y\} $.
Since $ -4uX+X\varphi $ evaluates to 1 at $ X=0 $, the tame symbol vanishes.
The second type is a blowup component with function field $ \Q(X) $, where $ Y=Xv $
for a coordinate~$ v $ in the blowup.
Now we multiply the first entry by $ -Y = - X v $ and notice that it
evaluates to~1 for~$ v = 0 $.

The next part of the argument is analytic and invokes the computation of $R$ via the explicit morphism of complexes in \cite{KLM}, as corrected in \cite{KL}.  Its compatibility with Bloch's cycle class map and (under composition with $\pi_{\RR}$) with the real regulator
(which is~$ \reg $ in our case) are checked in \cite[\S7]{KLM} and \cite[\S3.1]{Ke}, respectively.  (See also \cite{Ke2}.)
This shows that the diagram~\eqref{eq11.5} is commutative.

In particular, the image of~$ R $ of a symbol~$\{f,g\}$
on a suitable Zariski open part is represented
by $\log(f)\frac{dg}{g}-2\pi i\log(g)\delta_{T_f}$, where we take $\log(f)$ to have argument in $(-\pi,\pi)$ with branch cut along $T_f:=f^{-1}(\R_-)$, oriented from $-\infty$ to $0$.  Taking $i\mathrm{Im}$ of this current computes $\pi_{\RR}(R\{f,g\})$, and subtracting $d[i\arg(f)\log|g|]$ from this recovers exactly $\eta(f,g)$
as in Remark~\ref{rem8.6}, 
which, together with the identification~\eqref{reg22target},
makes explicit the commutativity of~\eqref{eq11.5}.

What follows works for $u\in\CC$ with $\mathrm{Re}(u)>1$.  The 3-chain $\Gamma=\{|X|\leq 1=|Y|\}$ has boundary $\mathbb{T}=\{|X|=|Y|=1\}\cong S^1\times S^1\subset U$.  Since $|\varphi(X,Y)|\leq 4$ on~$\mathbb{T}$, we have $\mathbb{T}\cap C=\emptyset$, making $\gamma:=\Gamma\cap C\in H_1(C,\ZZ)$. In fact, $\gamma$ is a primitive vanishing cycle for the semistable degeneration $u\to \infty$ (with $\mathrm{I}_8$ singular fiber the toric boundary of $\PP$).
Since $\mathbb{T}=\mathrm{Tube}(\gamma)$ under $\mathrm{Tube}\colon H_1(C_{4u},\ZZ)\to H_2(U,\ZZ)$, adjointness of Tube and~$ 2 \pi i $Res
gives
\begin{align*}
\int_{\gamma}R(\xi)&=\int_{\gamma}R(\mathrm{Res}_C(\Xi))=\int_{\gamma}\mathrm{Res}_C(R(\Xi))=\frac{1}{2\pi i}\int_{\mathbb{T}}R(\Xi) \\
&= \frac{1}{2\pi i}\int_{\mathbb{T}}R\{4u-\varphi,X,Y\}.
\end{align*}
Now
$R\{f,g,h\}=\log(f)\tfrac{dg}{g}\wedge\tfrac{dh}{h}-2\pi i\delta_{f^{-1}(\R_-)} R\{g,h\}$, but here $f=4u-\varphi$ does
not take negative real values on $\mathbb{T}$ under the assumption on $u$.  So writing $[-]_0$ for the
constant term in a Laurent polynomial, the above
equals
\begin{align*}
\frac{1}{2\pi i}\int_{\mathbb{T}}\log(4u-\varphi)\frac{dX}{X}\wedge\frac{dY}{Y}
&=\frac{1}{2\pi i}\int_{\mathbb{T}}\left\{\log(4u)-\sum_{m>0}\frac{\varphi^m}{m(4u)^m}\right\}\frac{dX}{X}\wedge\frac{dY}{Y}\\
&=2\pi i\left(\log(4u)-\sum_{m>0}\frac{[\varphi^m]_0}{m(4u)^{m}}\right),
\end{align*}
where the series expansion works since $|\varphi(X,Y)/4u|<1$ on $\mathbb{T}$. 
We compute that $[\varphi^m]_0=0$ for $m$ odd and $[\varphi^{2n}]_0=\sum_{k+\ell=n}\frac{(2n)!}{k!^2\ell!^2}=\binom{2n}{n}^2$, and conclude that $\int_{\gamma}R(\{X, Y\})=2\pi i F(u)$.  Taking
imaginary parts gives $\int_{\gamma}\reg(\{X, Y\})=2\pi i\mathrm{Re}(F(u))$, and $F(u)$ is already real
for $u >1$.

To complete the proof, we need to know that $\gamma$ equals~$\pm\gamma_C^-$ when $ u \in \R_{u>1} $; since it is primitive, it suffices to check that it is anti-invariant under complex conjugation on points for $u\in \R_{>1}$.  For $u$ in this range, notice that $\omega:=\text{Res}_C \left(\frac{dX/X\wedge dY/Y}{1-\varphi/4u}\right)$ is $\RR$-de Rham, and --- once again invoking the adjointness of $\mathrm{Tube}$ and $2\pi i\mathrm{Res}$ --- that
\begin{align*}
\int_{\gamma}\omega&=\frac{1}{2\pi i}\int_{\mathbb{T}}\frac{dX/X\wedge dY/Y}{1-\varphi/4u} = \frac{1}{2\pi i}\int_{\mathbb{T}}\sum_{m\geq 0}\frac{\varphi^m}{(4u)^m}\frac{dX}{X}\wedge\frac{dY}{Y}\\ &=2\pi i\sum_{n\geq 0}\binom{2n}{n}^2(4u)^{-2n}.
\end{align*}
Since the last expression takes purely imaginary values on $\R_{>1}$, $\gamma$ is anti-invariant and we are done.
\end{proof}

\begin{cor} \label{regnonzero}
If~$ u > 1$ in Proposition~\ref{reg-prop}, then
$\frac{1}{2\pi i}\int_{\gamma^-_{C}}\reg( \symb X Y )\ne 0$.
\end{cor}

\begin{proof}
Differentiation shows~$ F(u) $ is increasing on~$ (1,\infty) $. At $u=1$, the
curve~$C_4$  is parametrised by its normalisation (a projective
line) as
\[
X(z):=\frac{(1-\frac{1}{z})(1+\frac{i}{z})}{(1+\frac{1}{z})(1-\frac{i}{z})},\;\;\;\;\;Y(t)=\frac{(1-z)(1-\frac{z}{i})}{(1+z)(1+\frac{z}{i})},
\]
whence by \cite[Prop.~6.3]{DK} we have $$F(1)=\frac{4}{\pi}D_2(i)=\frac{4}{\pi}L(\Q,\chi_{-4},2)=\frac{4}{\pi}\sum_{k\geq 0}\frac{(-1)^k}{(2k+1)^2}>0,$$
where~$D_2$ is the Bloch-Wigner function and $\chi_{-4}$ the non-trivial Dirichlet character
with conductor~4.
\end{proof}

\begin{remar} \label{remark:u=4}
It is immediate from \cite[Cor.~4.4]{DK} that on all of $(1,\infty)$
we have $F(u)=m(4u):=\mathsf{m}(4u+\varphi(X,Y))$, where $\mathsf{m}$
denotes the logarithmic Mahler measure.  By invoking the Main Theorem and
equation~(40) of \cite{RZ}, we then obtain~$F(4)=m(16)=11\,m(1)=-\frac{165}{(2 \pi i)^2}L(\mathsf{E}_{15},2)= - \frac{165}{(2 \pi i)^2}L(E_4,2)$,
where $\mathsf{E}_{15}$ in the notation of loc.\ cit.\ is our
curve~$ C_1 $, which is an elliptic curve of conductor 15, and
we used that isogenous elliptic curves have the same
$L$-function as well as that there is only one isogeny class of elliptic
curves over~$ \Q $ with conductor~15 \cite{LMFDB}.
\end{remar}

\section{The 2-part of the Bloch-Kato conjecture for~\texorpdfstring{$ K_2(E_u) $}{K2(Eu)}} \label{2-BK}

Because of the exact sequence~\eqref{K2locseq}, and~$ K_2 $ of
any field being pure of weight~2,
we may identify~$ K_2(E_u) \otimes_\Z \Q = K_2^T(E_u) \otimes_\Z \Q = \textup{gr}^2( K_2(E_u) \otimes_\Z \Q) $
with~$ H^2_{\MMM}(E_u, \Q(2)) $,
as mentioned in Section~\ref{regBeil}.
This identifies~$ \INT E_u \otimes_\Z \Q $ with~$ H^2_{\MMM}(E_u, \Q(2))_{\Z}$.
In line with Beilinson's conjecture, we shall assume that~$ H^2_{\MMM}(E_u, \Q(2))_{\Z}$
has dimension~1. Below we shall also use this to simplify the notation
by identifying free rank~1 modules with their determinants.
We also let~$ \ell $ denote a prime number.

With this notation, the map~$ \reg_{\ell}^{\Q_\ell} : H^2_{\MMM}(E_u, \Q(2)) \to  H^1(\Q, \het^1(E_{u,\Qbar}, \Q_\ell(2))) $
mentioned just below~\eqref{MMTell-def} is induced by the map~$ \reg_{\ell} $ in~\eqref{K2TregCD}.
Its restriction to~$ H^2_{\MMM}(E_u, \Q(2))_{\Z} $, denoted~$ \reg_{\ell,\Z} $,
was already introduced in Section~\ref{BlKa},
where~\eqref{MMTell-def} defined~$ H^2_{\MMM}(E_u, \Q(2))_{\Z, T_{\ell}} $ as~$ \reg_{\ell,\Z}^{-1}(H^1(\Q, T_{\ell}(2))_\tf) $,
a~$ \Z_{(\ell)} $-submodule in the 1-dimensional~$ \Q $-vector space~$ H^2_{\MMM}(E_u, \Q(2))_{\Z} $.
We fix~$ \beta_u\ne 0 $ in~$ H^2_{\MMM}(E_u, \Q(2))_{\Z, T_{\ell}} $,
and set~$ \iota_{u,\ell} = | H^2_{\MMM}(E_u, \Q(2))_{\Z, T_{\ell}} : \Z_{(\ell)} \beta_u | $.
Under our assumptions this is finite if and only if~$ \reg_{\ell}(\beta_u) $ is not
divisible by arbitrarily large powers of~$ \ell $.
If also~$ \Z \, \reg(\beta_u) = f(u) \DDD_{1,2} $
for some real number~$ f(u) \ne 0 $,
then in the~$ \Z_{(\ell)} $-version of~\eqref{Ru-def} discussed
after~\eqref{BKconj1} we have
\begin{equation*}
f(u) \DDD_{1,2} \otimes_\Z\Z_{(\ell)} = \Z_{(\ell)} \reg(\beta_u) = \iota_{u,\ell} R_{u,\ell} \DDD_{1,2} \otimes_\Z\Z_{(\ell)}
,
\end{equation*}
so that~\eqref{BKconj} becomes
$$\ord_{\ell}\left(\frac{\iota_{u,\ell}\,L(E_u,2)}{f(u)}\right)
=
\ord_{\ell}\left(\frac{\prod_{p\leq\infty}\mathrm{Tam}^0_{p,\omega}(T_{\ell}(2))\,\# H^1_f(\Q, E_u[\ell^{\infty}](-1))}{\#H^0(\Q, E_u[\ell^{\infty}](1))\,\#H^0(\Q, E_u[\ell^{\infty}](-1))}\right).$$
This is under the following conditions.
\begin{assumption} \label{assumptions}
(1)
$ H^2_{\MMM}(E_u, \Q(2))_{\Z} $ is 1-dimensional, generated by~$ \beta_u $.

(2)
$ f(u) $ is non-zero.

(3)
$L(E_u, 2)/R_{u,\ell}  $ is in~$ \Q^\times $.

(4)
$ H^1_f(\Q, E_u[\ell^{\infty}](-1)) $ is finite.

(5)
$ \iota_{u,\ell} $ is finite.
\end{assumption}

Now taking~$ u \ne 0, \pm1 $ such that~$ 4 u $ is an integer, we
let~$ \beta_u $ be~$ \alpha $ as in Proposition~\ref{tameINT}
so that~$ \alpha $ is in~$ K_2^T(E_u) $, and~$ 2 \alpha $ is in~$ \INT E_u $.
Under more conditions on~$ u $
we can verify~(2) in Assumption~\ref{assumptions}, as well as~(5) for~$ \ell = 2 $.
(Condition~(4) for~$ \ell=2 $ also holds for various~$ u $; see Remark~\ref{yet-another-label} and the tables in Section~\ref{num-sec}.)
In order to use our earlier results, we assume that~$ u > 1 $ is an integer such that~$4\parallel u$ and~$ \ord_p(u(u^2-1))$ is zero
or odd for every odd prime~$ p $.
With~$ F(u) $ as in Proposition~\ref{reg-prop}, we find using
Proposition~\ref{Dint}, Lemma~\ref{reg-lemma}, Proposition~\ref{reg-prop},
and Corollary~\ref{regnonzero} that~$ f(u) = \pm (2\pi i)^2F(u) \ne 0$,
and Corollary~\ref{new-index-cor} together with Remark~\ref{m=1or2}
shows~$ \iota_{u,2} = 1 $ or~2.
In fact, with notation as in Corollary~\ref{new-index-cor} together with Remark~\ref{m=1or2}, we have~$ \iota_{u,2} = 1 $ if~$ m_u = 1 $,
i.e., when $ u+1 $ has a prime divisor that is congruent to~3
modulo~4, and~$ \iota_{u,2} = 1 $ or~2 if that is not the case,
so that~$ m_u = 2 $.

\begin{remar} \label{uis4}
For~$ u = 4 $, we also know (3) by Remark~\ref{remark:u=4}. Moreover,
explicit calculation using Corollary~\ref{2torscor} and Remark~\ref{selmertrivial}
shows that~$ H_f^1(\Q, E_4[2^\infty](-1) ) $
is trivial, so that (4) also holds, and only~(1) remains unknown.
\end{remar}

Referring back to Section \ref{BlKa}, note that our use of the element $\alpha$, for which we know that~$ \iota_{u,2} = 1 $ or~2, proves that $ H^2_{\MMM}(E_u, \Q(2))_{\Z, T_{2}} $ is a $\Z_{(2)}$-line
in $ H^2_{\MMM}(E_u, \Q(2))_{\Z} $, assuming that the latter is $1$-dimensional.  By~\eqref{K2TregCD}, $\alpha$ is in fact an element of~$ H^2_{\MMM}(E_u, \Q(2))_{\Z, T}=\cap_{\ell}  H^2_{\MMM}(E_u, \Q(2))_{\Z, T_{\ell}} $, so we may take $\beta_u=\alpha$ simultaneously for all $\ell$. But we know neither that the~$\iota_{u,\ell}$ are all finite nor that they are almost all $1$, so we cannot use $\alpha$ to show that~$ H^2_{\MMM}(E_u, \Q(2))_{\Z, T} $
as defined in~\eqref{MMTell-def} is a free~$\Z$-module of rank $1$.

In order to make  the version of~\eqref{BKconj} given above more explicit by incorporating our results from previous sections,
we let~$\omega(n)$ denote the total number of distinct prime divisors of~$n$,
as usual.
We also let~$\omega_1(n)$ (respectively, $\omega_3(n)$) denote the number of
distinct prime divisors of $n$ congruent to~1 (respectively~3) modulo~4.
Observing from~\eqref{Tam-def} that~$ \mathrm{Tam}^0_q(T_{\ell}(2)) = 1 $
if the curve has good reduction at~$ q \ne 2 $, and
using now Propositions~\ref{2tors}, \ref{tams}, \ref{taminf},
and~\ref{2tam2}, we obtain the following.

\begin{prop}\label{BK2}
Let~$ u $ be a positive integer congruent to~4 modulo~8, and such that~$ \ord_p(u(u^2-1))$
for each odd prime number is zero or odd.
Suppose that~(1), (3) and~(4) in Assumption~\ref{assumptions} hold,
and define~$s_u$ as $\ord_2(\# H^1_f(\Q, E_u[2^{\infty}](-1)))$. Then the Bloch-Kato conjecture predicts that
$$\ord_2\left(\frac{L(E_u, 2)}{(2\pi i)^2 F(u)}\right)= 2 \omega_1(u)+\omega_3(u)+\omega(u^2-1)-2+s_u-\ord_2(\iota_{u,2}),$$
where $F(u)=\log(4u)-\sum_{n=1}^{\infty}\left({2n\atop n}\right)^2\frac{(4u)^{-2n}}{2n}$ and $\iota_{u,2}=|H^2_{\MMM}(E_u, \Q(2))_{\Z, T_2}: \Z_{(2)}\alpha |$.
\end{prop}

\begin{remar} \label{yet-another-label}
The 2-torsion subgroup of~$ H_f^1(\Q, E_u[2^\infty](-1) ) $, for~$ u $ as in Proposition~\ref{BK2},
is finite (and explicitly computable for a given~$ u $) by Corollary~\ref{2torscor}.
Let~$ s_u' \ge 0 $ be such that~$ 2^{s_u'} $ is its order.
Without assuming any conjectures we have~$ s_u = 0 $ if and
only if~$ s_u' = 0 $ by Remark~\ref{selmertrivial}. 
Obviously~$ s_u \ge s_u' $, and~$ s_u = s_u' $ if and only if~2
is an exponent for~$ H_f^1(\Q, E_u[2^\infty](-1) ) $.
\end{remar}

\section{Numerical experiments with \texorpdfstring{$L(E_u, 2)$}{L(Eu,2)}} \label{num-sec}

The examples contained in Tables~\ref{numericstable1}, \ref{numericstable2}, \ref{numericstable3}, and~\ref{extratable},
were computed using pari/GP~\cite{PARI}.
By choice, for the given~$ u $ the conditions in Proposition~\ref{BK2} are satisfied,
and the first three tables contain the results for all such~$ u $ with $ u \le 1100 $.
(The results for all such~$ u $ with~$ u\le 25000 $ are available
 at~\url{https://www.few.vu.nl/~jeu/} .)
Because of Proposition~\ref{Vu-prop}, the~$ u $ for which~$ u(u^2-1)/4 $
is \emph{not} squarefree have been indicated with~$ {}^{\natural} $.

\begin{table}[t]\caption{\label{numericstable1}Data for Proposition~\ref{BK2} with~$   1 \le u \le  348 $.}
\resizebox{320pt}{!}{
\begin{tabular}{|c|c|c|c|c|c|c|c|c|}
\hline \vph
 $u$ &  $u/4$ & $u-1$ & $u+1$ & $ - N_u q_u $ & $ \hat s_u - \ord_2(\iota_{u,2})$ & $ s_u' $ & $ m_u - 1 $
 \\ \hline \vphs
4 &  $ 1 $ &  $ 3 $ &  $ 5 $ &  $ 11^{-1} $ & 0 & 0 & 1\rlap{${}^+$}
 \\ \hline \vphs 
12 &  $ 3 $ &  $ 11 $ &  $ 13 $ &  $ 2 $ & 0 & 0 & 1\rlap{${}^+$}
 \\ \hline \vphs 
20 &  $ 5 $ &  $ 19 $ &  $ 3 \cdot 7 $ &  $ 2^{3} $ & 0 & 0 & 0\rlap{*}
 \\ \hline \vphs 
28\rlap{${}^\natural$} &  $ 7 $ &  $ 3^{3} $ &  $ 29 $ &  $ 2 $ & 0 & 0 & 1\rlap{${}^+$}
 \\ \hline \vphs 
52 &  $ 13 $ &  $ 3 \cdot 17 $ &  $ 53 $ &  $ 2^{5} \cdot 3 $ & 2 & 1 & 1
 \\ \hline \vphs 
60 &  $ 3 \cdot 5 $ &  $ 59 $ &  $ 61 $ &  $ 2^{3} \cdot 29 $ & 0 & 0 & 1\rlap{${}^+$}
 \\ \hline \vphs 
68 &  $ 17 $ &  $ 67 $ &  $ 3 \cdot 23 $ &  $ 2^{3} \cdot 3^{3} $ & 0 & 0 & 0\rlap{*}
 \\ \hline \vphs 
84 &  $ 3 \cdot 7 $ &  $ 83 $ &  $ 5 \cdot 17 $ &  $ 2^{5} \cdot 17 $ & 2 & 1 & 1
 \\ \hline \vphs 
92 &  $ 23 $ &  $ 7 \cdot 13 $ &  $ 3 \cdot 31 $ &  $ 2^{5} \cdot 3 \cdot 5 $ & 2 & 2 & 0\rlap{*}
 \\ \hline \vphs 
108\rlap{${}^\natural$} &  $ 3^{3} $ &  $ 107 $ &  $ 109 $ &  $ 2 \cdot 3 \cdot 19 $ & 0 & 0 & 1\rlap{${}^+$}
 \\ \hline \vphs 
124\rlap{${}^\natural$} &  $ 31 $ &  $ 3 \cdot 41 $ &  $ 5^{3} $ &  $ 2^{4} \cdot 3 $ & 2 & 1 & 1
 \\ \hline \vphs 
132 &  $ 3 \cdot 11 $ &  $ 131 $ &  $ 7 \cdot 19 $ &  $ 2^{6} \cdot 3^{3} $ & 3 & 1 & 0
 \\ \hline \vphs 
140 &  $ 5 \cdot 7 $ &  $ 139 $ &  $ 3 \cdot 47 $ &  $ 2^{4} \cdot 113 $ & 0 & 0 & 0\rlap{*}
 \\ \hline \vphs 
156 &  $ 3 \cdot 13 $ &  $ 5 \cdot 31 $ &  $ 157 $ &  $ 2^{4} \cdot 3^{2} \cdot 23 $ & 0 & 0 & 1\rlap{${}^+$}
 \\ \hline \vphs 
164 &  $ 41 $ &  $ 163 $ &  $ 3 \cdot 5 \cdot 11 $ &  $ 2^{10} \cdot 3 $ & 6 & 1 & 0
 \\ \hline \vphs 
188\rlap{${}^\natural$} &  $ 47 $ &  $ 11 \cdot 17 $ &  $ 3^{3} \cdot 7 $ &  $ 2^{5} \cdot 13 $ & 2 & 2 & 0\rlap{*}
 \\ \hline \vphs 
204 &  $ 3 \cdot 17 $ &  $ 7 \cdot 29 $ &  $ 5 \cdot 41 $ &  $ 2^{10} \cdot 7 $ & 5 & 1 & 1
 \\ \hline \vphs 
212 &  $ 53 $ &  $ 211 $ &  $ 3 \cdot 71 $ &  $ 2^{3} \cdot 3^{2} \cdot 73 $ & 0 & 0 & 0\rlap{*}
 \\ \hline \vphs 
220 &  $ 5 \cdot 11 $ &  $ 3 \cdot 73 $ &  $ 13 \cdot 17 $ &  $ 2^{9} \cdot 13 $ & 4 & 1 & 1
 \\ \hline \vphs 
228 &  $ 3 \cdot 19 $ &  $ 227 $ &  $ 229 $ &  $ 2^{2} \cdot 3 \cdot 5^{4} $ & 0 & 0 & 1\rlap{${}^+$}
 \\ \hline \vphs 
236 &  $ 59 $ &  $ 5 \cdot 47 $ &  $ 3 \cdot 79 $ &  $ 2^{8} \cdot 3 \cdot 11 $ & 5 & 2 & 0
 \\ \hline \vphs 
268 &  $ 67 $ &  $ 3 \cdot 89 $ &  $ 269 $ &  $ 2^{4} \cdot 3 \cdot 5 \cdot 43 $ & 2 & 1 & 1
 \\ \hline \vphs 
284 &  $ 71 $ &  $ 283 $ &  $ 3 \cdot 5 \cdot 19 $ &  $ 2^{5} \cdot 449 $ & 2 & 2 & 0\rlap{*}
 \\ \hline \vphs 
292 &  $ 73 $ &  $ 3 \cdot 97 $ &  $ 293 $ &  $ 2^{5} \cdot 419 $ & 2 & 2 & 1\rlap{${}^+$}
 \\ \hline \vphs 
308 &  $ 7 \cdot 11 $ &  $ 307 $ &  $ 3 \cdot 103 $ &  $ 2^{8} \cdot 3^{2} \cdot 7 $ & 5 & 1 & 0
 \\ \hline \vphs 
340 &  $ 5 \cdot 17 $ &  $ 3 \cdot 113 $ &  $ 11 \cdot 31 $ &  $ 2^{6} \cdot 3 \cdot 7 \cdot 17 $ & 0 & 0 & 0\rlap{*}
 \\ \hline \vphs 
348 &  $ 3 \cdot 29 $ &  $ 347 $ &  $ 349 $ &  $ 2^{3} \cdot 5 \cdot 7 \cdot 97 $ & 0 & 0 & 1\rlap{${}^+$}
 \\ \hline  
\end{tabular}%
}%
\end{table}

\begin{table}[t]\caption{\label{numericstable2}Data for Proposition~\ref{BK2} with~$ 356 \le u \le  732 $.}
\resizebox{320pt}{!}{
\begin{tabular}{|c|c|c|c|c|c|c|c|c|}
\hline \vph
 $u$ &  $u/4$ & $u-1$ & $u+1$ & $ - N_u q_u $ & $ \hat s_u - \ord_2(\iota_{u,2})$ & $ s_u' $ & $ m_u - 1 $
 \\ \hline \vphs
356 &  $ 89 $ &  $ 5 \cdot 71 $ &  $ 3 \cdot 7 \cdot 17 $ &  $ 2^{7} \cdot 3^{2} \cdot 23 $ & 2 & 2 & 0\rlap{*}
 \\ \hline \vphs 
372 &  $ 3 \cdot 31 $ &  $ 7 \cdot 53 $ &  $ 373 $ &  $ 2^{6} \cdot 7 \cdot 79 $ & 3 & 1 & 1
 \\ \hline \vphs 
380 &  $ 5 \cdot 19 $ &  $ 379 $ &  $ 3 \cdot 127 $ &  $ 2^{4} \cdot 3^{3} \cdot 7 \cdot 11 $ & 0 & 0 & 0\rlap{*}
 \\ \hline \vphs 
412 &  $ 103 $ &  $ 3 \cdot 137 $ &  $ 7 \cdot 59 $ &  $ 2^{7} \cdot 3^{2} \cdot 31 $ & 4 & 2 & 0
 \\ \hline \vphs 
420 &  $ 3 \cdot 5 \cdot 7 $ &  $ 419 $ &  $ 421 $ &  $ 2^{4} \cdot 3^{2} \cdot 13 \cdot 29 $ & 0 & 0 & 1\rlap{${}^+$}
 \\ \hline \vphs 
428 &  $ 107 $ &  $ 7 \cdot 61 $ &  $ 3 \cdot 11 \cdot 13 $ &  $ 2^{8} \cdot 3^{2} \cdot 17 $ & 4 & 3 & 0
 \\ \hline \vphs 
436 &  $ 109 $ &  $ 3 \cdot 5 \cdot 29 $ &  $ 19 \cdot 23 $ &  $ 2^{7} \cdot 359 $ & 2 & 2 & 0\rlap{*}
 \\ \hline \vphs 
444 &  $ 3 \cdot 37 $ &  $ 443 $ &  $ 5 \cdot 89 $ &  $ 2^{4} \cdot 5 \cdot 17 \cdot 43 $ & 0 & 0 & 1\rlap{${}^+$}
 \\ \hline \vphs 
452 &  $ 113 $ &  $ 11 \cdot 41 $ &  $ 3 \cdot 151 $ &  $ 2^{6} \cdot 3 \cdot 239 $ & 2 & 1 & 0
 \\ \hline \vphs 
460\rlap{${}^\natural$} &  $ 5 \cdot 23 $ &  $ 3^{3} \cdot 17 $ &  $ 461 $ &  $ 2^{4} \cdot 3^{2} \cdot 47 $ & 0 & 0 & 1\rlap{${}^+$}
 \\ \hline \vphs 
492 &  $ 3 \cdot 41 $ &  $ 491 $ &  $ 17 \cdot 29 $ &  $ 2^{4} \cdot 3 \cdot 5 \cdot 13 \cdot 23 $ & 0 & 0 & 1\rlap{${}^+$}
 \\ \hline \vphs 
500\rlap{${}^\natural$} &  $ 5^{3} $ &  $ 499 $ &  $ 3 \cdot 167 $ &  $ 2^{3} \cdot 3 \cdot 11^{2} $ & 0 & 0 & 0\rlap{*}
 \\ \hline \vphs 
516 &  $ 3 \cdot 43 $ &  $ 5 \cdot 103 $ &  $ 11 \cdot 47 $ &  $ 2^{8} \cdot 3 \cdot 5^{3} $ & 4 & 2 & 0
 \\ \hline \vphs 
556 &  $ 139 $ &  $ 3 \cdot 5 \cdot 37 $ &  $ 557 $ &  $ 2^{5} \cdot 5 \cdot 13 \cdot 47 $ & 2 & 2 & 1\rlap{${}^+$}
 \\ \hline \vphs 
564 &  $ 3 \cdot 47 $ &  $ 563 $ &  $ 5 \cdot 113 $ &  $ 2^{5} \cdot 3^{3} \cdot 5 \cdot 29 $ & 2 & 1 & 1
 \\ \hline \vphs 
572 &  $ 11 \cdot 13 $ &  $ 571 $ &  $ 3 \cdot 191 $ &  $ 2^{8} \cdot 359 $ & 4 & 1 & 0
 \\ \hline \vphs 
580 &  $ 5 \cdot 29 $ &  $ 3 \cdot 193 $ &  $ 7 \cdot 83 $ &  $ 2^{6} \cdot 1667 $ & 0 & 0 & 0\rlap{*}
 \\ \hline \vphs 
596 &  $ 149 $ &  $ 5 \cdot 7 \cdot 17 $ &  $ 3 \cdot 199 $ &  $ 2^{11} \cdot 5 \cdot 11 $ & 6 & 2 & 0
 \\ \hline \vphs 
620\rlap{${}^\natural$} &  $ 5 \cdot 31 $ &  $ 619 $ &  $ 3^{3} \cdot 23 $ &  $ 2^{4} \cdot 3^{2} \cdot 5 \cdot 19 $ & 0 & 0 & 0\rlap{*}
 \\ \hline \vphs 
628 &  $ 157 $ &  $ 3 \cdot 11 \cdot 19 $ &  $ 17 \cdot 37 $ &  $ 2^{9} \cdot 3 \cdot 7 \cdot 11 $ & 4 & 3 & 1
 \\ \hline \vphs 
644 &  $ 7 \cdot 23 $ &  $ 643 $ &  $ 3 \cdot 5 \cdot 43 $ &  $ 2^{9} \cdot 3^{3} \cdot 11 $ & 5 & 2 & 0
 \\ \hline \vphs 
652 &  $ 163 $ &  $ 3 \cdot 7 \cdot 31 $ &  $ 653 $ &  $ 2^{7} \cdot 3^{2} \cdot 113 $ & 4 & 2 & 1
 \\ \hline \vphs 
660 &  $ 3 \cdot 5 \cdot 11 $ &  $ 659 $ &  $ 661 $ &  $ 2^{4} \cdot 41 \cdot 307 $ & 0 & 0 & 1\rlap{${}^+$}
 \\ \hline \vphs 
668 &  $ 167 $ &  $ 23 \cdot 29 $ &  $ 3 \cdot 223 $ &  $ 2^{7} \cdot 3^{2} \cdot 5^{3} $ & 4 & 2 & 0
 \\ \hline \vphs 
708 &  $ 3 \cdot 59 $ &  $ 7 \cdot 101 $ &  $ 709 $ &  $ 2^{5} \cdot 3 \cdot 2113 $ & 2 & 1 & 1
 \\ \hline \vphs 
716 &  $ 179 $ &  $ 5 \cdot 11 \cdot 13 $ &  $ 3 \cdot 239 $ &  $ 2^{8} \cdot 3^{2} \cdot 89 $ & 4 & 3 & 0
 \\ \hline \vphs 
732 &  $ 3 \cdot 61 $ &  $ 17 \cdot 43 $ &  $ 733 $ &  $ 2^{4} \cdot 14479 $ & 0 & 0 & 1\rlap{${}^+$}
 \\ \hline  
\end{tabular}%
}%
\end{table}

\begin{table}[t]\caption{\label{numericstable3}Data for Proposition~\ref{BK2} with~$ 740 \le u \le 1092 $.}
\resizebox{320pt}{!}{
\begin{tabular}{|c|c|c|c|c|c|c|c|c|}
\hline \vph
 $u$ &  $u/4$ & $u-1$ & $u+1$ & $ - N_u q_u $ & $ \hat s_u - \ord_2(\iota_{u,2})$ & $ s_u' $ & $ m_u - 1 $
 \\ \hline \vphs
740 &  $ 5 \cdot 37 $ &  $ 739 $ &  $ 3 \cdot 13 \cdot 19 $ &  $ 2^{6} \cdot 5 \cdot 11 \cdot 61 $ & 0 & 0 & 0\rlap{*}
 \\ \hline \vphs 
756\rlap{${}^\natural$} &  $ 3^{3} \cdot 7 $ &  $ 5 \cdot 151 $ &  $ 757 $ &  $ 2^{8} \cdot 3 \cdot 43 $ & 5 & 1 & 1
 \\ \hline \vphs 
772 &  $ 193 $ &  $ 3 \cdot 257 $ &  $ 773 $ &  $ 2^{9} \cdot 3^{2} \cdot 7^{2} $ & 6 & 2 & 1
 \\ \hline \vphs 
780 &  $ 3 \cdot 5 \cdot 13 $ &  $ 19 \cdot 41 $ &  $ 11 \cdot 71 $ &  $ 2^{13} \cdot 3 \cdot 13 $ & 6 & 1 & 0
 \\ \hline \vphs 
788 &  $ 197 $ &  $ 787 $ &  $ 3 \cdot 263 $ &  $ 2^{3} \cdot 3 \cdot 11 \cdot 881 $ & 0 & 0 & 0\rlap{*}
 \\ \hline \vphs 
796 &  $ 199 $ &  $ 3 \cdot 5 \cdot 53 $ &  $ 797 $ &  $ 2^{5} \cdot 3^{2} \cdot 7 \cdot 127 $ & 2 & 2 & 1\rlap{${}^+$}
 \\ \hline \vphs 
804 &  $ 3 \cdot 67 $ &  $ 11 \cdot 73 $ &  $ 5 \cdot 7 \cdot 23 $ &  $ 2^{9} \cdot 641 $ & 4 & 3 & 0
 \\ \hline \vphs 
812 &  $ 7 \cdot 29 $ &  $ 811 $ &  $ 3 \cdot 271 $ &  $ 2^{4} \cdot 3^{3} \cdot 631 $ & 0 & 0 & 0\rlap{*}
 \\ \hline \vphs 
836\rlap{${}^\natural$} &  $ 11 \cdot 19 $ &  $ 5 \cdot 167 $ &  $ 3^{3} \cdot 31 $ &  $ 2^{8} \cdot 3^{3} \cdot 5 $ & 4 & 2 & 0
 \\ \hline \vphs 
852 &  $ 3 \cdot 71 $ &  $ 23 \cdot 37 $ &  $ 853 $ &  $ 2^{5} \cdot 3^{3} \cdot 19 \cdot 23 $ & 2 & 1 & 1
 \\ \hline \vphs 
860 &  $ 5 \cdot 43 $ &  $ 859 $ &  $ 3 \cdot 7 \cdot 41 $ &  $ 2^{7} \cdot 3 \cdot 5^{3} \cdot 7 $ & 2 & 1 & 0
 \\ \hline \vphs 
876\rlap{${}^\natural$} &  $ 3 \cdot 73 $ &  $ 5^{3} \cdot 7 $ &  $ 877 $ &  $ 2^{4} \cdot 5^{2} \cdot 43 $ & 0 & 0 & 1\rlap{${}^+$}
 \\ \hline \vphs 
884 &  $ 13 \cdot 17 $ &  $ 883 $ &  $ 3 \cdot 5 \cdot 59 $ &  $ 2^{6} \cdot 23 \cdot 239 $ & 0 & 0 & 0\rlap{*}
 \\ \hline \vphs 
916 &  $ 229 $ &  $ 3 \cdot 5 \cdot 61 $ &  $ 7 \cdot 131 $ &  $ 2^{7} \cdot 3^{3} \cdot 109 $ & 2 & 2 & 0\rlap{*}
 \\ \hline \vphs 
940 &  $ 5 \cdot 47 $ &  $ 3 \cdot 313 $ &  $ 941 $ &  $ 2^{4} \cdot 3 \cdot 11 \cdot 809 $ & 0 & 0 & 1\rlap{${}^+$}
 \\ \hline \vphs 
948 &  $ 3 \cdot 79 $ &  $ 947 $ &  $ 13 \cdot 73 $ &  $ 2^{5} \cdot 113 \cdot 139 $ & 2 & 1 & 1
 \\ \hline \vphs 
956 &  $ 239 $ &  $ 5 \cdot 191 $ &  $ 3 \cdot 11 \cdot 29 $ &  $ 2^{10} \cdot 3^{2} \cdot 7^{2} $ & 6 & 3 & 0
 \\ \hline \vphs 
972\rlap{${}^\natural$} &  $ 3^{5} $ &  $ 971 $ &  $ 7 \cdot 139 $ &  $ 2^{5} \cdot 3 \cdot 5 \cdot 13 $ & 3 & 1 & 0
 \\ \hline \vphs 
988 &  $ 13 \cdot 19 $ &  $ 3 \cdot 7 \cdot 47 $ &  $ 23 \cdot 43 $ &  $ 2^{8} \cdot 13 \cdot 139 $ & 2 & 2 & 0\rlap{*}
 \\ \hline \vphs 
996 &  $ 3 \cdot 83 $ &  $ 5 \cdot 199 $ &  $ 997 $ &  $ 2^{5} \cdot 3^{2} \cdot 2333 $ & 2 & 1 & 1
 \\ \hline \vphs 
1004 &  $ 251 $ &  $ 17 \cdot 59 $ &  $ 3 \cdot 5 \cdot 67 $ &  $ 2^{8} \cdot 3^{2} \cdot 223 $ & 4 & 3 & 0
 \\ \hline \vphs 
1012 &  $ 11 \cdot 23 $ &  $ 3 \cdot 337 $ &  $ 1013 $ &  $ 2^{5} \cdot 3 \cdot 4993 $ & 2 & 1 & 1
 \\ \hline \vphs 
1020 &  $ 3 \cdot 5 \cdot 17 $ &  $ 1019 $ &  $ 1021 $ &  $ 2^{13} \cdot 3^{4} $ & 8 & 1 & 1
 \\ \hline \vphs 
1028\rlap{${}^\natural$} &  $ 257 $ &  $ 13 \cdot 79 $ &  $ 3 \cdot 7^{3} $ &  $ 2^{7} \cdot 3^{4} $ & 3 & 1 & 0
 \\ \hline \vphs 
1060 &  $ 5 \cdot 53 $ &  $ 3 \cdot 353 $ &  $ 1061 $ &  $ 2^{6} \cdot 47 \cdot 197 $ & 1 & 1 & 1\rlap{${}^+$}
 \\ \hline \vphs 
1068 &  $ 3 \cdot 89 $ &  $ 11 \cdot 97 $ &  $ 1069 $ &  $ 2^{9} \cdot 3^{3} \cdot 7^{2} $ & 5 & 2 & 1
 \\ \hline \vphs 
1092 &  $ 3 \cdot 7 \cdot 13 $ &  $ 1091 $ &  $ 1093 $ &  $ 2^{4} \cdot 45971 $ & 0 & 0 & 1\rlap{${}^+$}
 \\ \hline  
\end{tabular}%
}%
\end{table}

Each number~$\frac{L(E_u, 2)}{(2\pi i)^2 F(u)}$ was computed with~35 relevant
decimal places, then replaced by the obvious rational approximation~$ q_u $.
Determining~$ q_u $ was made easier by considering~$ N_u \frac{L(E_u, 2)}{(2\pi i)^2 F(u)} $, which by the functional equation of the~$ L $-function equals~$ \pm \frac{L'(E_u,0)}{F(u)} $. Here $ N_u$ is the conductor of $ E_u $, equal to the product of the (odd) prime numbers dividing $\frac14 u (u^2-1) $ by Proposition~\ref{minimal}.
This was close to a negative integer except for~$ u = 4 $, where
it was close to $ - \frac{1}{11} $.

\begin{remar} We should emphasise that Remark~\ref{remark:u=4} actually \emph{proved}
that~$ q_4 $ is rational and that~$ - N_4 q_4 = \frac{1}{11} $,
where~$N_4=15$. The Bloch-Kato conjecture implies an interpretation
for the factor~11 in the denominator. For this, consider its
statement as immediately before Assumption~\ref{assumptions},
for~$ \ell = 11 $, $ u = 4 $, and~$\beta_4=\alpha$ so that $f(4)=\pm(2\pi i)^2F(4)$. On the right-hand side, the $11$-parts of the Tamagawa factors are all integral, indeed trivial. For the triviality of $\mathrm{Tam}^0_{11,\omega_{11}}(T_{11}(2))$, see Remark~\ref{BK4.1iii}, replacing  $p=2$ by the sufficiently large $p=11$ (also of good reduction). For $p=3, 5$, see the proof of Proposition \ref{tams} (2) and (3), noting that~$11\nmid \ord_p(\Delta)$. The terms in the denominator are trivial since $E_4$ does not have a rational $11$-isogeny
(like all elliptic curves over $\Q$ with three exceptions of conductor~121;
see \cite[p.79]{bi-ku:mfo}). Hence the only way to account for the $11$ in the denominator of~$ - N_4q_4$, in
accordance with the conjecture, is for~$11$ to divide~$\iota_{4,11}$. In other words, for~$u=4$ we would have to have~$ \reg_{11}(\alpha) $ divisible by~$11$ in~$ H^1(\Q, \het^1(E_{u,\Qbar}, \Z_{11}(2)))_\tf $.
\end{remar}

We have given~$ q_u $ implicitly in the tables by listing the factorisation of~$ - N_u q_u $ (note the sign).
Because~$ N_u $ is odd,
the exponent of~2 in~$ - N_u q_u $ equals~$ \ord_2(q_u) $. We have also listed
the prime factorisations of~$ u/4 $, $ u-1 $ and~$ u+1 $, and the
resulting value for~$ m_u $ as in Corollary~\ref{new-index-cor},
with~$ \ord_2(\iota_{u,2}) \le m_u -1 $.
We recall from Remark~\ref{2torsrem} that~$ m_u = 1 $
or~2 for our~$ u $. We also listed
\begin{equation} \label{Sudiff}
\hat s_u-\ord_2(\iota_{u,2}) := \ord_2(q_u) - \omega_3(u) - 2 \omega_1(u) - \omega(u^2-1) + 2,
\end{equation}
where $\hat s_u$ is the value of $s_u$ predicted by Proposition~\ref{BK2}.

In our data, we always have~$ \hat s_u-\ord_2(\iota_{u,2}) $ greater
than or equal to~$ s_u' $, so having~$ \iota_{u,2} = 1 $ would
always be compatible with the prediction in Proposition~\ref{BK2}.
In fact, there are various cases%
\footnote{In the tables, those that have~$ s_u' = 0 $ together with a marking by~$ 1^+ $, as in the next paragraph.}
where~$ s_u' = 0 $, hence~$ s_u = 0 $,
and the prediction is that~$ \iota_{u,2} = 1 $, whereas~$ m_u = 2 $.
But although there is no evidence that $ \iota_{u,2} = 2 $ for any
of our examples, we cannot exclude the possibility.

In the tables, we included some information on special situations in the column for~$ m_u-1 $.
If~$ m_u-1=0$ (so that~$ \iota_{u,2} = 1 $) and additionally~$ \hat s_u = s_u' $, we write~``$ 0^* $''.
If~$ m_u-1=1$ (so that~$ \iota_{u,2} = 1 $ or~2) and assuming~$ \iota_{u,2} = 1 $ leads to~$ \hat s_u = s_u' $, we write ``$1^+$''.
In the first case the prediction is then that~$ H_f^1(\Q, E_u[2^\infty](-1) ) $ is 2-torsion, and
is computed exactly by Corollary~\ref{2torscor});
in the second case the same holds under the assumption~$ \iota_{u,2} = 1 $.

We observe that if~$ s_u' = 1 $ then~$ H_f^1(\Q, E_u[2^\infty](-1) ) $ 
is either isomorphic to~$ \Q_2/\Z_2 $ or cyclic of finite order~$ 2^{s_u} $.
It appears that in the latter case the predicted order can be large.
We have~$ s_u' = 1 $ and~$ \iota_{u,2} = 1 $
for example, for~$ u = 164 $, where the prediction is that~$ s_u = 6 $,
and for~$ u = 5612 $, where the prediction is that~$ s_u = 11 $.
If~$ s_u' = 1 $ and~$ m_u = 2 $ then we have to modify
this because the prediction in the tables is for~$ s_u - \ord_2(\iota_{u,2}) $,
and~$ \iota_{u,2} $ could equal~1 or~2. 
For example, for~$ u = 2140 $ the prediction is that~$ H_f^1(\Q, E_u[2^\infty](-1) ) $ 
is cyclic of order~$ 2^9 $ or~$ 2^{10} $.

The cumulative maximum of $ s_u' $ seems to grow
slowly with~$ u $. We have~$ s_u' = 5 $ for the first time
at~$ u = 5116 $, $ s_u' = 6 $ at~$ u = 47 356 $,
$ s_u' = 7 $ at~$ u = 443 444 $, 
$ s_u' = 8 $ at~$ u = 1 983 164 $, 
$ s_u' = 9 $ at~$ u = 13 816 804 $,
and~$ s_u' = 10 $ at~$ u = 105 514 564 $.
(These first times remain the same if we consider only~$ u $
for which~$ u(u^2-1)/4 $ is squarefree as in Proposition~\ref{Vu-prop}.)

We can rewrite~\eqref{Sudiff} into the prediction
\begin{equation} \label{prediction}
\ord_2(q_u) + \ord_2(\iota_{u,2}) + 2 = 2 \omega_1(u) + \omega_3(u) + \omega(u^2-1) + s_u
\end{equation}
and consider when the value~$ V_u $ of the right-hand side (which
does not depend on any conjectures but could be infinity) is small.
In order to avoid complications, we assume here that~$ (u/4)(u^2-1) $
is squarefree.

\begin{prop} \label{Vu-prop}
For~$ u \ge 4 $ an integer such that~$ 4 || u $ and~$ (u/4)(u^2-1) $
is squarefree, let~$ V_u $ be the value of the right-hand
side of~\eqref{prediction}. Then~$ V_u \ge 2 $, and
\begin{itemize}
\item
$ V_u = 2 $ if and only if~$ u = 4 $,

\item
$ V_u = 3 $ if and only if~$ u = 12 $,

\item
$ V_u = 4 $ if and only if $ u -1 $ and $ u+1 $ are primes, and $ u= 12 p $
for a prime number $ p $ congruent to 3 modulo 4.
\end{itemize}
\end{prop}

\begin{proof}
Because~$ u^2 - 1 = (u-1) (u+1) $ we have~$ \omega(u) \ge 2 $,
and because~$ 4 ||u $ we can have~$ V_u = 2 $ only if~$ u = 4 $ and~$ s_u = 0 $.
Corollary~\ref{2torscor} and Remark~\ref{selmertrivial} show
that for~$ u = 4 $ we indeed have~$ s_u = 0 $, so this does occur.
For~$ u \ne 4 $ we have~$ \omega_1(u) + \omega_3(u) \ge 1 $,
so~$ V_u = 3 $  is equivalent to~$ \omega(u^2-1) = 2 $, $ \omega_1(u) = 0 $, $ \omega_3(u) = 1 $ and~$ s_u = 0 $.
Then~$ u = 12 $ because~$ u/4 $, $ u-1 $ and~$ u+1 $ must be prime numbers
that represent the residue classes of integers modulo~3, and~$ u \ne 4 $.
For~$ u = 12 $ we have~$ s_u = 0 $ by Corollary~\ref{2torscor} and Remark~\ref{selmertrivial},
hence~$ V_u = 3 $.

Because~$ u \ne 4 $ or~12 implies~$ \omega_1(u) + \omega_3(u) + \omega(u^2-1) \ne 2 $
or~3, for~$ V_u = 4 $ either
\begin{enumerate}
\item
$ \omega_1(u) = 0 $, $ \omega_3(u) = 2 $, $ \omega(u^2-1) = 2 $, and $ s_u = 0 $,
or

\item
$ \omega_1(u) = 0 $, $ \omega_3(u) = 1 $, $ \omega(u^2-1) = 3 $, and $ s_u = 0 $.
\end{enumerate}
For~(2) we would have, in the notation of Corollary~\ref{2torscor},
that~$ S' = \emptyset $, $ D'=1 $, and the sole condition that~$ D \equiv 1 $
modulo~8 is fulfilled for at least two of the eight positive
divisors of~$ u^2-1 $; hence~$ s_u = 0 $ cannot hold simultaneously.

Thus only~(1) can occur. Here~$ u-1 $ and~$ u+1 $ are prime
numbers larger than~3, hence represent the classes of~1
and~2 modulo~3, leaving~$ u/4 $ to represent the class of~0.
So~3 divides~$ u $, and~$ u = 12 p $ with~$ p $ a prime number congruent
to~3 modulo~4. For such~$ u $ we have, in the notation of Corollary~\ref{2torscor},
that~$ S' = \emptyset $, $ D' = 1 $, and with~$ u - 1 \equiv 3 $ and~$ u + 1 \equiv 5 $
modulo~8, out of the four possible~$ D $ only~$ D = 1 $
satisfies~$ D \equiv 1 $ modulo~8, so that~$ s_u = 0 $ by Remark~\ref{selmertrivial},
hence~$ V_u = 4 $.
\end{proof}

\begin{exo} \label{Vu=4-ex}
For~$ u $ in~$ \{20, \dots, 25000 \} $ such that~$ (u/4)(u^2-1) $
is squarefree, we have~$ V_u = 4 $ 
for~$ u = 228 $, 1668, 3252, 4548, 8292, 8628, 9012,
10068, 12612, 17988, 18132 and~19428.
\end{exo}

\begin{remar} \label{minimalremark}
The proof of Proposition~\ref{Vu-prop} does not use any conjectures; in
all cases, we can decide whether or not~$ s_u = 0 $ using Corollary~\ref{2torscor} and Remark~\ref{selmertrivial}. But if we assume
the prediction~\eqref{prediction}, then we expect for~$ V_u = 4 $ that
either~$ \ord_2(q_u) = 2 $ and~$ \ord_2(\iota_{u,2}) = 0 $,
or~$ \ord_2(q_u) = 1 $ and~$ \ord_2(\iota_{u,2}) = 1 $.
Under the assumptions of the proposition, for~$ V_u = 4 $ we have~$ m_u = 2 $ in Corollary~\ref{new-index-cor},
so both options are open.
For the~$ u $ in Example~\ref{Vu=4-ex} we find numerically that~$ \ord_2(q_u) = 2 $,
which would then require~$ \ord_2(\iota_{u,2}) = 0 $.

Conversely, the prediction~\eqref{prediction} implies that~$ \ord_2(q_u) = 2 $ 
only occurs if~$ V_u = 4 $ and~$ \ord_2(\iota_{u,2}) = 0 $,
or~$ V_u = 5 $ and~$ \ord_2(\iota_{u,2}) = 1 $.
But for~$ u $ as in the proposition, amongst our data~$ \ord_2(q_u) = 2 $ only happens for the~$ u $
in Example~\ref{Vu=4-ex}, with~$ V_u = 4 $.

Thus, in our data, for such~$ u $ we have~$ \ord_2(q_u) = 2 $ \emph{precisely}
when~$ V_u = 4 $, which by Proposition~\ref{Vu-prop} is equivalent
to~$ u-1 $ and~$ u+1 $ being twin primes, and~$ u = 12 p $
with~$ p $ a prime congruent to~3 modulo~4. In fact, as explained
in the introduction, it was the philosophy that limiting the power of~2 in~$ q_u $
should limit the number of primes of bad reduction (which are the odd prime numbers dividing~$ u (u^2-1) $)
that gave rise to this paper.
\end{remar}

\begin{remar} \label{notsquarefree}
If we consider when~$ V_u = 3 $ or~4 for~$ u $ as in Proposition~\ref{BK2} then other possibilities for~$ u $
occur (but~$ V_u = 2 $ still only holds for~$ u = 4 $). Among our data~$ V_u = 3 $ also occurs for~$ u = 28 $ and~108, with~$ \ord_2(q_u) = 1 $,
and~$ V_u = 4 $ for~$ u = 15972 $ and~21492, with~$ \ord_2(q_u) = 2 $.
These are also the only additional values of~$ u $ with~$ \ord_2(q_u) = 1 $
or~2.
\end{remar}

We conclude this section exploring the variety in our data.

\begin{table}[t]\caption{\label{extratable}Special situations in the numerical data for Proposition~\ref{BK2} with~$ u $ in $ \{ 20,\dots,25000 \} $.}
\resizebox{!}{175pt}{
\begin{tabular}{|c|c|c|c|c|c|c|c|c|}
\hline
\vph $u$ &  $u/4$ & $u-1$ & $u+1$ & $ - N_u q_u $ & $ \hat s_u - \ord_2(\iota_{u,2})$ & $ s_u' $ & $ m_u - 1 $
 \\ \hline \vphs
 228 &  $ 3 \cdot 19 $ &  $ 227 $ &  $ 229 $ &  $ 2^{2} \cdot 3 \cdot 5^{4} $ & 0 & 0 & 1\rlap{${}^+$}
 \\ \hline \vphs
1668 &  $ 3 \cdot 139 $ &  $ 1667 $ &  $ 1669 $ &  $ 2^{2} \cdot 3^{2} \cdot 68023 $ & 0 & 0 & 1\rlap{${}^+$}
 \\ \hline \vphs
3252 &  $ 3 \cdot 271 $ &  $ 3251 $ &  $ 3253 $ &  $ 2^{2} \cdot 3 \cdot 5 \cdot 29 \cdot 9067 $ & 0 & 0 & 1\rlap{${}^+$}
 \\ \hline \vphs
4548 &  $ 3 \cdot 379 $ &  $ 4547 $ &  $ 4549 $ &  $ 2^{2} \cdot 3^{2} \cdot 1268759 $ & 0 & 0 & 1\rlap{${}^+$}
 \\ \hline \vphs
8292 &  $ 3 \cdot 691 $ &  $ 8291 $ &  $ 8293 $ &  $ 2^{2} \cdot 3 \cdot 61 \cdot 71 \cdot 5099 $ & 0 & 0 & 1\rlap{${}^+$}
 \\ \hline \vphs
8628 &  $ 3 \cdot 719 $ &  $ 8627 $ &  $ 8629 $ &  $ 2^{2} \cdot 3^{6} \cdot 98257 $ & 0 & 0 & 1\rlap{${}^+$}
 \\ \hline \vphs
9012 &  $ 3 \cdot 751 $ &  $ 9011 $ &  $ 9013 $ &  $ 2^{2} \cdot 3 \cdot 5^{2} \cdot 7 \cdot 11 \cdot 13903 $ & 0 & 0 & 1\rlap{${}^+$}
 \\ \hline \vphs
10068 &  $ 3 \cdot 839 $ &  $ 10067 $ &  $ 10069 $ &  $ 2^{2} \cdot 107381389 $ & 0 & 0 & 1\rlap{${}^+$}
 \\ \hline \vphs
12612 &  $ 3 \cdot 1051 $ &  $ 12611 $ &  $ 12613 $ &  $ 2^{2} \cdot 3 \cdot 59 \cdot 409 \cdot 3271 $ & 0 & 0 & 1\rlap{${}^+$}
 \\ \hline \vphs
15972\rlap{${}^\natural$} &  $ 3 \cdot 11^{3} $ &  $ 15971 $ &  $ 15973 $ &  $ 2^{2} \cdot 859 \cdot 4357 $ & 0 & 0 & 1\rlap{${}^+$}
 \\ \hline \vphs
17988 &  $ 3 \cdot 1499 $ &  $ 17987 $ &  $ 17989 $ &  $ 2^{2} \cdot 1487 \cdot 396953 $ & 0 & 0 & 1\rlap{${}^+$}
 \\ \hline \vphs
18132 &  $ 3 \cdot 1511 $ &  $ 18131 $ &  $ 18133 $ &  $ 2^{2} \cdot 3^{3} \cdot 17 \cdot 59 \cdot 79 \cdot 283 $ & 0 & 0 & 1\rlap{${}^+$}
 \\ \hline \vphs
19428 &  $ 3 \cdot 1619 $ &  $ 19427 $ &  $ 19429 $ &  $ 2^{2} \cdot 3^{3} \cdot 11^{2} \cdot 283 \cdot 859 $ & 0 & 0 & 1\rlap{${}^+$}
 \\ \hline \vphs
21492\rlap{${}^\natural$} &  $ 3^{3} \cdot 199 $ &  $ 21491 $ &  $ 21493 $ &  $ 2^{2} \cdot 3^{2} \cdot 773 \cdot 17449 $ & 0 & 0 & 1\rlap{${}^+$}
 \\ \hline \hline \vphs
22660 &  $ 5 \cdot 11 \cdot 103 $ &  $ 3 \cdot 7 \cdot 13 \cdot 83 $ &  $ 17 \cdot 31 \cdot 43 $ &  $ 2^{25} \cdot 3 \cdot 43 $ & 16 & 4 & 0
 \\ \hline \vphs
2716 &  $ 7 \cdot 97 $ &  $ 3 \cdot 5 \cdot 181 $ &  $ 11 \cdot 13 \cdot 19 $ &  $ 2^{20} \cdot 3^{2} $ & 13 & 3 & 0
 \\ \hline \vphs
11452 &  $ 7 \cdot 409 $ &  $ 3 \cdot 11 \cdot 347 $ &  $ 13 \cdot 881 $ &  $ 2^{20} \cdot 3 \cdot 173 $ & 14 & 2 & 1
 \\ \hline \vphs
20460 &  $ 3 \cdot 5 \cdot 11 \cdot 31 $ &  $ 41 \cdot 499 $ &  $ 7 \cdot 37 \cdot 79 $ &  $ 2^{20} \cdot 3^{3} \cdot 5 \cdot 29 $ & 12 & 2 & 0
 \\ \hline \vphs
20596 &  $ 19 \cdot 271 $ &  $ 3 \cdot 5 \cdot 1373 $ &  $ 43 \cdot 479 $ &  $ 2^{20} \cdot 3^{2} \cdot 7 \cdot 47 $ & 15 & 3 & 0
 \\ \hline \hline \vphs
52 &  $ 13 $ &  $ 3 \cdot 17 $ &  $ 53 $ &  $ 2^{5} \cdot 3 $ & 2 & 1 & 1
 \\ \hline \vphs 
772 &  $ 193 $ &  $ 3 \cdot 257 $ &  $ 773 $ &  $ 2^{9} \cdot 3^{2} \cdot 7^{2} $ & 6 & 2 & 1
 \\ \hline \vphs 
1732 &  $ 433 $ &  $ 3 \cdot 577 $ &  $ 1733 $ &  $ 2^{9} \cdot 3^{2} \cdot 509 $ & 6 & 2 & 1
 \\ \hline \vphs 
2308 &  $ 577 $ &  $ 3 \cdot 769 $ &  $ 2309 $ &  $ 2^{10} \cdot 3^{3} \cdot 5 \cdot 37 $ & 7 & 2 & 1
 \\ \hline \vphs
19212 &  $ 3 \cdot 1601 $ &  $ 19211 $ &  $ 19213 $ &  $ 2^{7} \cdot 3^{3} \cdot 7^{2} \cdot 19 \cdot 977 $ & 4 & 1 & 1
 \\ \hline \vphs
19252\rlap{${}^\natural$} &  $ 4813 $ &  $ 3^{3} \cdot 23 \cdot 31 $ &  $ 13 \cdot 1481 $ &  $ 2^{13} \cdot 7^{2} \cdot 19 \cdot 37 $ & 8 & 3 & 1
 \\ \hline \hline \vphs
340 &  $ 5 \cdot 17 $ &  $ 3 \cdot 113 $ &  $ 11 \cdot 31 $ &  $ 2^{6} \cdot 3 \cdot 7 \cdot 17 $ & 0 & 0 & 0\rlap{*}
 \\ \hline \vphs
1508 &  $ 13 \cdot 29 $ &  $ 11 \cdot 137 $ &  $ 3 \cdot 503 $ &  $ 2^{6} \cdot 3 \cdot 7517 $ & 0 & 0 & 0\rlap{*}
 \\ \hline \vphs
7540 &  $ 5 \cdot 13 \cdot 29 $ &  $ 3 \cdot 7 \cdot 359 $ &  $ 7541 $ &  $ 2^{8} \cdot 31 \cdot 23663 $ & 0 & 0 & 1\rlap{${}^+$}
 \\ \hline \vphs
12180 &  $ 3 \cdot 5 \cdot 7 \cdot 29 $ &  $ 19 \cdot 641 $ &  $ 13 \cdot 937 $ &  $ 2^{8} \cdot 3 \cdot 5 \cdot 17 \cdot 71 \cdot 193 $ & 0 & 0 & 1\rlap{${}^+$}
 \\ \hline \vphs
12932\rlap{${}^\natural$} &  $ 53 \cdot 61 $ &  $ 67 \cdot 193 $ &  $ 3^{3} \cdot 479 $ &  $ 2^{6} \cdot 5 \cdot 250673 $ & 0 & 0 & 0\rlap{*}
\\ \hline \hline \vphs
5612 &  $ 23 \cdot 61 $ &  $ 31 \cdot 181 $ &  $ 3 \cdot 1871 $ &  $ 2^{16} \cdot 3 \cdot 349 $ & 11 & 1 & 0
 \\ \hline
\end{tabular}%
}%
\end{table}%

\begin{exo} \label{exo-special}
In the five parts of Table~\ref{extratable} we give some examples
of the following special situations in our data for~$ u = 20, \dots 25000 $.
\begin{enumerate}
\item
$ \ord_2(q_u) $ takes the minimal value $2$ (subject to the exclusion of $u=4$,
for which it is~2, as well as~$u=12$, 28, and~108, for which it is~1). This is attained for the~$ u $ in Example~\ref{Vu=4-ex}
and Remark~\ref{notsquarefree}.

\item
$ \ord_2(q_u) $ is maximal.
The largest value is~25, for~$ u=22660 $,
and the next largest value is~20, which occurs for~$ u = 2716 $,
11452, 20460 and~20596.

\item
$ (u/4) (u^2-1) $ has few prime factors but $ \ord_2(q_u) $ is
relatively large. Here the Tamagawa factors contribute little to it
and $ s_u' $ is small, but $ \hat s_u $ is large.

\item
$ s_u' = 0 $, hence~$ s_u = 0 $, but there are relatively many primes of bad reduction.
These primes make~$\ord_2(q_u)$ relatively large, via Tamagawa factors, but
do not force the 2-Selmer group to be large too. (Alternatively,
it is predicted to be cyclic of order~4 for some~$ u $,
including~6060 and 10660, where there are~8 or~9 primes of bad reduction.)
This is in contrast to the situation for the real quadratic fields $\Q(\sqrt{u^2-4})$ discussed in the introduction, where the $2$-parts of the class number and $\frac{L'(\chi_D,0)}{F_1(u)}$ are equal, bounded below by $2^{r-2}$, where $r$ is the number of ramified primes. There, there was no contribution from Tamagawa factors.

\item
$ s_u' = 1 $ but~$ \hat s_u $ is large, so the Selmer group 
should be cyclic of large order.
\end{enumerate}
\end{exo}

\end{document}